\documentclass{amsart}
\usepackage{latexsym,amssymb,amsmath,amsthm,amscd,graphicx}

\numberwithin{equation}{section}
\newtheorem{theorem}{Theorem}[section]
\newtheorem{lemma}[theorem]{Lemma}
\newtheorem{corollary}[theorem]{Corollary}
\newtheorem{proposition}[theorem]{Proposition}

\theoremstyle{definition} 

\newtheorem{remark}[theorem]{Remark}
\newtheorem{definition}[theorem]{Definition}
\newtheorem{example}[theorem]{Example}

\theoremstyle{remark}

\newcommand{\C}{{\mathbb C}}
\newcommand{\N}{{\mathbb N}}

\newcommand{\R}{{\mathbb R}}

\newcommand{\Z}{{\mathbb Z}}

\usepackage{color}
\definecolor{blue}{rgb}{0,0,0.45}
\definecolor{red}{rgb}{0.7,0,0}

\begin{document}

\title{Generalized Morrey spaces and trace operator}

\author[S.~Nakamura]{Shohei Nakamura}
\author[T.~Noi]{Takahiro Noi}
\author[Y.~Sawano]{Yoshihiro Sawano}
\address{Department of Mathematics and Information Science, 
Tokyo Metropolitan University, Hachioji, 192-0397, Japan}
\email{pokopoko9131@icloud.com, taka.noi.hiro@gmail.com, ysawano@tmu.ac.jp}

\maketitle

\begin{abstract}
The theory of generalized Besov-Morrey spaces
and generalized Triebel-Lizorkin-Morrey spaces
is developed.
Generalized Morrey spaces,
which T. Mizuhara and E. Nakai proposed, are equipped
with a parameter and a function.
The trace property is one of the main focuses
of the present paper, 
which
will clarify the role of the parameter
of generalized Morrey spaces.
The quarkonial decomposition
is obtained as an application of atomic decomposition.
In the end, the relation between
the function spaces dealt in the present paper
and
the foregoing researches is discussed.
\end{abstract}

\section{Introduction}

In the present paper,
we systematically develop
the theory of generalized Besov-Morrey spaces
and generalized Triebel-Lizorkin-Morrey spaces
and then we compare our results
with existing ones in Section \ref{s6}.
Our results will polish existing ones,
as is seen from Section \ref{s6}.

Let $0<q \le p<\infty$.
Then, the Morrey space ${\mathcal M}^p_q({\mathbb R}^n)$
is the set of all measurable functions $f$ for which
the quasi-norm
\[
\|f\|_{{\mathcal M}^p_q}
\equiv 
\sup_{Q \in {{\mathcal D}}}
|Q|^{\frac{1}{p}-\frac{1}{q}}
\left(\int_Q |f(y)|^q\,dy\right)^{\frac1q}
\]
is finite,
where ${\mathcal D}$ denotes the set of all dyadic cubes.

In the present paper,
we consider the role of the parameter $q$
in ${\mathcal M}^p_q({\mathbb R}^n)$
by considering the generalized Morrey space
${\mathcal M}^\varphi_q({\mathbb R}^n)$.
\begin{definition}{\cite{Nakai94}}
Let $0<q<\infty$ and $\varphi:(0,\infty) \to (0,\infty)$
be a function.
Then define
\[
\|f\|_{{\mathcal M}^\varphi_q}
\equiv 
\sup_{Q \in {{\mathcal D}}}
\varphi(\ell(Q))
\left(\frac{1}{|Q|}\int_Q |f(y)|^q\,dy\right)^{\frac1q}
\]
for a measurable function $f$.
The space
${\mathcal M}^\varphi_q({\mathbb R}^n)$
is the set of all measurable functions
$f$ for which the quasi-norm
$\|f\|_{{\mathcal M}^\varphi_q}$
is finite.
\end{definition}

{
We envisage the following functions as examples
of $\varphi$:
\begin{example}
Let $0<q\le p<\infty$ and $\varphi:(0,\infty) \to (0,\infty)$
be a function.
\begin{enumerate}
\item
We can recover the Morrey space ${\mathcal M}^p_q({\mathbb R}^n)$
by letting $\varphi(t)=t^{n/p}$ for $t>0$.
We discuss why we need to generalize
the parameter $p$ in Proposition \ref{prop:150312-100}
and Remark \ref{rem:150821-2}.
\item
We can recover the Lebesgue space $L^q({\mathbb R}^n)$
by letting $\varphi(t)=t^{n/q}$ for $t>0$.
\item
A simple but standard example is as follows:
\[
\varphi(t)=\frac{t}{l_n(t)} \in {\mathcal G}_1
\equiv
\bigcap_{0<s<t<\infty}
\{\varphi:(0,\infty) \to (0,\infty)\,:\,
\varphi(s) \le \varphi(t), \varphi(s)s^{-n} \ge \varphi(t)t^{-n}\},
\]
where $l_n(t)$ is given inductively by:
\[
l_0(t)=t, \quad
l_n(t)=\log(3+l_{n-1}(t)) \quad
(n=1,2,\ldots)
\]
for $t>0$.
\end{enumerate}
\end{example}
}
We shall {establish} that the parameter $q$
in ${\mathcal M}^\varphi_q({\mathbb R}^n)$
plays the role of local regularity
by considering the trace property
of generalized Besov-Morrey spaces
and
generalized Triebel-Lizorkin-Morrey spaces.
To define these spaces,
we use the following {notation} in the present paper:
\begin{itemize}
\item
By \lq \lq cube" we mean a compact cube
whose edges are parallel to the coordinate {\it axes}.
If a cube has center $x$ and {side-length} $r$,
we denote it by $Q(x,r)$.
{}From the definition of $Q(x,r)$,
\begin{equation}\label{eq:150825-1}
|Q(x,r)|=(2r)^n.
\end{equation}
We write $Q(r)$ instead of 
{
$Q(0,r)$.
}
{Conversely,}
given a cube $Q$,
we denote by $c(Q)$ {\it the center of $Q$}
and by $\ell(Q)$ the {\it {{ side-length }} of $Q$}:
$\ell(Q)=|Q|^{1/n}$,
where $|Q|$ denotes the volume of the cube $Q$.
\item
Let $a \in \R$.
Then write
$a_+\equiv \max(a,0)$ and $a_-\equiv \min(a,0)$.
The Gauss sign $[a]$ is defined
to be the largest integer $m$
which is less than or equal to $a$.
\item
Let $A,B \ge 0$.
Then $A \lesssim B$ means
that there exists a constant $C>0$
such that $A \le C B$,
where $C$ depends only on the parameters
of importance.
When $A \lesssim B \lesssim A$,
write $A \sim B$.
When we want to stress that the implicit constants
in these symbols depend on important parameters,
add them as subscripts.
For example, $A \lesssim_p B$ means that
there exists a constant $C>0$ depending only on $p$
such that $A \le C B$.
\item
Let $a\in{\mathbb R}^n$. 
We define $\langle a\rangle \equiv \sqrt{1+|a|^2}$. 
\item
For $N \in {\mathbb N}$ and $\varphi \in C^N({\mathbb R}^n)$,
one defines
\[
p_N(\varphi){\equiv}\sum_{|\alpha| \le N}
\sup_{x \in {\mathbb R}^n}
(1+|x|)^N|\partial^\alpha \varphi(x)|.
\]
\item
The Schwartz space ${\mathcal S}({\mathbb R}^n)$ is defined
to be the set of all $f \in C^\infty({\mathbb R}^n)$
for which the semi-norm $p_N(f)$ is finite for all $N \in {\mathbb N}_0$. 
\item
The space ${\mathcal S}_\infty({\mathbb R}^n)$ is the set of all
$f \in {\mathcal S}({\mathbb R}^n)$ for which 
\[
\int_{{\mathbb R}^n}x^\alpha f(x)\,dx=0
\]
for all 
{
$\alpha \in {\mathbb N}_0^{\ n}$.
}
\item
The topological dual of 
${\mathcal S}({\mathbb R}^n)$
and 
${\mathcal S}_\infty({\mathbb R}^n)$
are {denoted} by 
${\mathcal S}'({\mathbb R}^n)$
and 
${\mathcal S}_\infty'({\mathbb R}^n)$,
respectively.
Equip 
${\mathcal S}'({\mathbb R}^n)$
and 
${\mathcal S}_\infty'({\mathbb R}^n)$
with the weak-* topology.
\item
For $j \in {\mathbb Z}$ and
$m=(m_1,m_2,\ldots,m_n) \in {\mathbb Z}^n$,
we define
$\displaystyle
Q_{jm}
\equiv \prod_{k=1}^n \left[\frac{m_k}{2^j},\frac{m_k+1}{2^j}\right).
$
If notational confusion seems likely,
we write
$Q_{j,m}$ instead of $Q_{jm}$.
Denote by ${\mathcal D}={\mathcal D}({\mathbb R}^n)$ the set of such cubes.
The elements in ${\mathcal D}$ are called
dyadic cubes.
{
In the present paper, 
$\mathcal{D}$ does not stand for the set of all
compactly supported functions $C^\infty_{\rm c}{({\mathbb R}^n)}$.
}
\item
Define the Fourier transform and its inverse by
\begin{align*}
\begin{cases}
\displaystyle
{\mathcal F}f(\xi)
\equiv 
\frac{1}{\sqrt{(2\pi)^n}}\int_{{\mathbb R}^n} f(x)e^{-i x \cdot \xi} \,dx
\quad(\xi \in {\mathbb R}^n)\\
\displaystyle
{\mathcal F}^{-1}f(x)
\equiv 
\frac{1}{\sqrt{(2\pi)^n}}\int_{{\mathbb R}^n} f(\xi)e^{i x \cdot \xi} \,d\xi
\quad(x \in {\mathbb R}^n)
\end{cases}
\end{align*}
if $f$ is an integrable function.
{
In a standard way, we extend the definition of $\mathcal{F}$ and $\mathcal{F}^{-1}$
to {the space of all} tempered distributions $\mathcal{S}'(\R^n)$.
}
\item
For $\varphi \in {\mathcal S}({\mathbb R}^n)$
and $f \in {\mathcal S}'({\mathbb R}^n)$
we write $\varphi(D)f\equiv{\mathcal F}^{-1}[\varphi{\mathcal F} f]$,
or equivalently we define
\[
\varphi(D)f(x)\equiv\frac{1}{\sqrt{(2\pi)^n}}
\langle f,{\mathcal F}^{-1}\varphi(x-\cdot) \rangle.
\] 
\item
The Kronecker delta is given by:
\[
\delta_{a b} \equiv 
\begin{cases}
1&a=b,\\
0&a\ne b
\end{cases}
\]
for $a,b \in {\mathbb Z}$.
\item
We make use of the following notation:
for $m=(m_1,m_2,\ldots,m_n) \in {\mathbb Z}^n$,
we define
$m'\equiv(m_1,m_2,\ldots,m_{n-1}) \in {\mathbb Z}^{n-1}.$
Conversely, we shall write
$m=(m',m_n)$
for $m'=(m_1,m_2,\ldots,m_{n-1}) \in {\mathbb Z}^{n-1}$
and $m_n \in {\mathbb Z}$.
\item 
Denote by ${\rm BC}({\mathbb R}^n)$ 
the {Banach} space of all bounded continuous functions. 
Let $f\in {\rm BC}({\mathbb R}^n)$. Then we define $||f||_{{\rm BC}}$ 
{by} $||f||_{{\rm BC}}\equiv||f||_{L^{\infty}}$. 
\item Let $m\in\N_0$. Then, denote by ${\rm BC}^m({\mathbb R}^n)$ the linear space of functions $f\,:\,{\mathbb R}^n\longrightarrow \C$ such that 
$f\in C^m$ and $\partial^{\alpha}f\in{\rm BC}$ 
for any multi-index $\alpha$ with $|\alpha|\le m$. 
We define the norm such that 
\[
\| f\|_{{\rm BC}^m} = 
\sum_{|\alpha|\le m}\left\| \partial^{\alpha}f\right\|_{ {\rm BC} }. 
\]
\item Denote by ${\rm BUC}{({\mathbb R}^n)}$ the Banach 
space consisting of bounded uniformly continuous functions. 
Then we define 
$||f||_{{\rm BUC}}\equiv||f||_{{\infty}}$.
\item
{
Let $\nu>0$.
}
The space
$H_2^{\nu}({\mathbb R}^n)$ stands for the ($L^2$-based) potential space
of order $\nu$;
$$
{H_2^{\nu}({\mathbb R}^n)\equiv}
\{H\in{\mathcal S}'({\mathbb R}^n)\,:\,(1-\Delta)^{\nu/2}H\in 
L^2({\mathbb R}^n)\}.
$$
Equip $H_2^{\nu}({\mathbb R}^n)$ with the norm:
\[
\|H\|_{H^{\nu}_{2}}\equiv \| (1-\Delta)^{\nu/2}H\|_2
\quad (H \in H_2^{\nu}({\mathbb R}^n)).
\]
\item
Let $K$ be a compact set.
The set ${\mathcal S}'_K{({\mathbb R}^n)}$ denotes the set 
of all $f \in {\mathcal S}'{({\mathbb R}^n)}$
such that {${\mathcal F}f$ is supported on $K$}.
Likewise define
${\mathcal S}_K({\mathbb R}^n)
{\equiv}
{\mathcal S}{({\mathbb R}^n)} 
\cap 
{\mathcal S}'_K{({\mathbb R}^n)}$.
\item
When two Banach spaces $X$ and $Y$ are isomorphic,
write $X \simeq Y$.
\end{itemize}

Now let us define generalized Besov-Morrey spaces
and
generalized Triebel-Lizorkin-Morrey spaces.
{
Let $0<q< \infty$.
Denote by ${\mathcal G}_q$ 
the set of all nondecreasing functions $\varphi:(0,\infty) \to (0,\infty)$
such that
\begin{equation}\label{eq:140820-6}
\varphi(t_1)t_1{}^{-n/q}
\ge
\varphi(t_2)t_2{}^{-n/q}
\quad (0<t_1 \le t_2<\infty).
\end{equation}
}

\begin{definition}
Let $0<q<\infty$, $0<r \le \infty$, $s \in {\mathbb R}$
 and $\varphi \in {\mathcal G}_q$.
Let $\theta$ and $\tau$ be compactly supported functions
satisfying
\[
0 \notin {\rm supp}(\tau), \quad
\theta(\xi)>0 \mbox{\rm \, if \, } \xi \in Q(2), \quad
\tau(\xi)>0 \mbox{\rm \, if \,} \xi \in Q(2) \setminus Q(1).
\]
Define $\tau_k(\xi) \equiv \tau(2^{-k}\xi)$
for $\xi \in {\mathbb R}^n$ and $k \in {\mathbb N}$.
\begin{enumerate}
\item
{\it The $($nonhomogeneous$)$ generalized Besov-Morrey space}
${\mathcal N}_{{\mathcal M}^\varphi_q,r}^s({\mathbb R}^n)$
is the set of all
$f \in {\mathcal S}'({\mathbb R}^n)$
for which the quasi-norm
\begin{equation}\label{eq:140820-140}
\|f\|_{{\mathcal N}_{{\mathcal M}^\varphi_q,r}^s}
\equiv 
\begin{cases}
\displaystyle
\|\theta(D)f\|_{{\mathcal M}^\varphi_q}
+
\left(\sum_{j=1}^\infty
2^{jsr}\|\tau_j(D)f\|_{{\mathcal M}^\varphi_q}^r
\right)^{\frac1r}&(r<\infty),\\
\displaystyle
\|\theta(D)f\|_{{\mathcal M}^\varphi_q}
+
\sup_{j \in {\mathbb N}}
2^{js}\|\tau_j(D)f\|_{{\mathcal M}^\varphi_q}
&(r=\infty)
\end{cases}
\end{equation}
is finite.
\item
{\it The $($nonhomogeneous$)$ generalized Triebel-Lizorkin-Morrey space}
${\mathcal E}_{{\mathcal M}^\varphi_q,r}^s({\mathbb R}^n)$
is the set of all
$f \in {\mathcal S}'({\mathbb R}^n)$
for which the quasi-norm
\begin{equation}\label{eq:140820-141}
\|f\|_{{\mathcal E}_{{\mathcal M}^\varphi_q,r}^s}
\equiv 
\begin{cases}
\displaystyle
\|\theta(D)f\|_{{\mathcal M}^\varphi_q}
+
\left\|
\left(\sum_{j=1}^\infty
2^{jsr}|\tau_j(D)f|^r
\right)^{\frac1r}\right\|_{{\mathcal M}^\varphi_q}&(r<\infty),\\
\displaystyle
\|\theta(D)f\|_{{\mathcal M}^\varphi_q}
+\left\|
\sup_{j \in {\mathbb N}}
2^{js}|\tau_j(D)f|
\right\|_{{\mathcal M}^\varphi_q}
&(r=\infty)
\end{cases}
\end{equation}
is finite.
\item
The space
${\mathcal A}^s_{{\mathcal M}^\varphi_q,r}({\mathbb R}^n)$
denotes either
${\mathcal N}^s_{{\mathcal M}^\varphi_q,r}({\mathbb R}^n)$
or
${\mathcal E}^s_{{\mathcal M}^\varphi_q,r}({\mathbb R}^n)$.
\end{enumerate}
\end{definition}

The next theorem answers
the most fundamental question on these function spaces:
do the definitions of 
${\mathcal A}^s_{{\mathcal M}^\varphi_q,r}({\mathbb R}^n)$
depend on the different choices of admissible $\theta$ and $\tau$ ?
\begin{theorem}\label{thm:150205-1}
{
Let $0<q \le r<\infty$, $s \in {\mathbb R}$ and $\varphi \in {\mathcal G}_q$.}
Assume
{
that there exist constants $\varepsilon>0$ and $C>0$ such that
\begin{equation*}
\frac{t^{\varepsilon}}{\varphi(t)}
\le
C\frac{r^\varepsilon}{\varphi(r)}
\quad (t\ge r),
\end{equation*}
}
in the case
when
${\mathcal A}^s_{{\mathcal M}^\varphi_q,r}({\mathbb R}^n)
={\mathcal E}^s_{{\mathcal M}^\varphi_q,r}({\mathbb R}^n)$
with $r<\infty$.
Then
different choices of admissible $\theta$ and $\tau$
will yield equivalent norms. 
\end{theorem}
Theorem \ref{thm:150205-1} is a starting point
of the present paper.
Based upon this result,
we investigate
the decomposition properties and the fundamental theorems.

In the present paper we investigate the role of the parameter $q$
in ${\mathcal M}^\varphi_q({\mathbb R}^n)$.
{An experience in \cite{EGNS14,SaWa13}
shows} that a passage 
from the classical Morrey space ${\mathcal M}^p_q({\mathbb R}^n)$
to the generalized Morrey space ${\mathcal M}^\varphi_q({\mathbb R}^n)$
is not a mere quest to generalization.
It naturally emerges when we consider the limiting case
of the Sobolev embedding;
see \cite{EGNS14,SaWa13}.

{We structure
the remaining part of the present paper as follows:}
Section \ref{s2} reviews the {fundamental} property
of the underlying space ${\mathcal M}^\varphi_q({\mathbb R}^n)$;
we transform the results obtained
earlier to a form we use in the present paper.
In Section \ref{s3},
{we justify the definition of generalized Besov spaces and}
generalized Triebel-Lizorkin-Morrey spaces on ${\mathbb R}^n$.
{
We prove Theorem \ref{thm:150205-1}.
Theorem \ref{thm:150205-1} can be generalized in many directions;
see Sections \ref{s7.3} and \ref{s7.5} for some hints to generalize
what we obtain.
We also investigate some fundamental properties.
}
Section \ref{s4} considers decompositions;
atomic decomposition,
molecular decomposition
and 
quarkonial decomposition will be obtained.
As applications of {these results,} in Section \ref{s5},
we establish fundamental theorems in these function spaces.
The first one is the boundedness of the trace operator,
which is new.
{In the general setting, it is difficult to describe
the image of the trace operator.}
Next, we investigate the pointwise multiplication property.
Finally, we consider the diffeomorphism properties.
The pointwise multiplication property is partially obtained 
in \cite[{Section 5}]{LSUYY2}.
However, our result will be sharper;
in the earlier work \cite{LSUYY2}
the authors depended on the Peetre maximal operator
but in the present paper
we do not have to rely upon this maximal operator.
For the definition of the Peetre maximal operator,
see Lemma \ref{lem:140817-1} below.
Comparing \cite[Theorems 4.9 and 4.12]{SaTa2007} 
with \cite[{Section 4}]{LSUYY2},
we see that the postulates in the theorems of decomposition 
in \cite[{Theorems 4.9 and 4.12}]{SaTa2007}
can be milder than those 
in \cite[{Section 4}]{LSUYY2}.
Our results {carry} over to homogeneous spaces,
which will be done in Section \ref{s5.5}.
We discuss a property of the topologies of
${\mathcal S}({\mathbb R}^n)$ and
${\mathcal S}_\infty({\mathbb R}^n)$ as well in Section \ref{s5.5}.
Finally,
in Section \ref{s6},
{we} describe the recent development
of the related function spaces
and compare our results with the ones obtained earlier.
In particular,
let us recall that we used the following 
{notation} in \cite{SaTa2007}:
\[
{\mathcal N}^s_{pqr}({\mathbb R}^n)
=
{\mathcal N}^s_{{\mathcal M}^{\varphi}_q,r}({\mathbb R}^n), \quad
{\mathcal E}^s_{pqr}({\mathbb R}^n)
=
{\mathcal E}^s_{{\mathcal M}^{\varphi}_q,r}({\mathbb R}^n), \quad
{\mathcal A}^s_{pqr}({\mathbb R}^n)
=
{\mathcal A}^s_{{\mathcal M}^{\varphi}_q,r}({\mathbb R}^n)
\]
when $\varphi(t)=t^{n/p}$,
so that this paper will reinforce \cite{SaTa2007}.
Of course, 
{we can recover}
the Besov space $B^s_{pq}({\mathbb R}^n)$
and the Triebel-Lizorkin space $F^s_{pq}({\mathbb R}^n)$
{for any $0<p,q \le \infty$ and $s \in {\mathbb R}$:}
\[
B^s_{pq}({\mathbb R}^n)={\mathcal N}^s_{ppq}({\mathbb R}^n), \quad
F^s_{pq}({\mathbb R}^n)={\mathcal E}^s_{ppq}({\mathbb R}^n),
\]
respectively. 
As a consequence,
the notation $A^s_{pq}({\mathbb R}^n)$ agrees with
${\mathcal A}^s_{ppq}({\mathbb R}^n)$;
see \cite{SaTa2007}.
We will discuss what results were obtained earlier
and where our results in the present papers can be located.
Based on the results obtained here,
we will discuss some possible extensions of the results.

\section{Structure of ${\mathcal M}^\varphi_q({\mathbb R}^n)$}
\label{s2}

\subsection{Assumptions on $\varphi$}

The next lemma justifies our class
${\mathcal G}_q$.

\begin{lemma}\label{lem:140820-1}
Let $0<q<\infty$.
\begin{enumerate}
\item
{\rm \cite[p.446]{Nakai00}}
For all $\varphi:(0,\infty) \to (0,\infty)$,
there exists $\varphi^*\in {\mathcal G}_q$ such that
\[
{\mathcal M}^\varphi_q({\mathbb R}^n)
\simeq
{\mathcal M}^{\varphi^*}_q({\mathbb R}^n)
\]
in the sense of norm equivalence.
\item
For any function
$\varphi:(0,\infty) \to (0,\infty)$,
${\mathcal M}^{\varphi}_q \ne \{0\}$
if and only if
$$
\varphi^*(t){\equiv}\sup_{s \in [t,\infty)}
t^{\frac{n}{q}}s^{-\frac{n}{q}}\varphi(s)
$$
is finite for all $t>0$.
\end{enumerate}
\end{lemma}

\begin{proof}
(1) is known{, cf.} {\rm \cite[p.446]{Nakai00}}.
The \lq \lq if part" of (2) is trivial,
because we know that
\[
\|\chi_{[0,t]^n}\|_{{\mathcal M}^\varphi_q}
=
\varphi(t)
\]
for $t>0$.
So, we concentrate on the \lq \lq only if part" of (2).

Assume 
${\mathcal M}^\varphi_q({\mathbb R}^n)$ 
contains a nonzero function;
$f \in {\mathcal M}^{\varphi}_q \setminus \{0\}$
and suppose {to the contrary}
$\varphi^*(t)=\infty$.
Then
there exists $x_f \in {\mathbb R}^n$ such that
\[
\int_{Q(x_f,t)}|f(y)|^q\,dy>0.
\]
For each $m \in {\mathbb N}$,
we can find $s_m \in [t,\infty)$
such that
$t^{\frac{n}{q}}s_m{}^{-\frac{n}{q}}
{
\varphi
}
(s_m)>m$.
Therefore,
\begin{align*}
\|f\|_{{\mathcal M}^{\varphi}_q}
&\ge
\sup_{r>0}\varphi(r)
\left(\frac{1}{r^n}\int_{Q(x_f,r)}|f(y)|^q\,dy\right)^{1/q}\\
&\ge
\varphi(s_m)s_m{}^{-\frac{n}{q}}
\left(\int_{Q(x_f,t)}|f(y)|^q\,dy\right)^{\frac{1}{q}}\\
&\ge
m
\left(\frac{1}{t^n}\int_{Q(x_f,t)}|f(y)|^q\,dy\right)^{\frac{1}{q}}.
\end{align*}
Since $m$ is arbitrary and independent of $t$,
this contradicts to $f \in {\mathcal M}^\varphi_q({\mathbb R}^n)$.
\end{proof}

The next lemma ensures
that ${\mathcal M}^\varphi_q({\mathbb R}^n)$ 
contains a nonzero function.
\begin{lemma}{\rm \cite[Proposition A]{ES}}
Let $0<q<\infty$ and $\varphi \in {\mathcal G}_q$.
Then
\begin{equation}\label{eq:150205-3}
\|\chi_{Q(x,R)}\|_{{\mathcal M}^\varphi_q}
=
\varphi(R)
\end{equation}
for all $x \in {\mathbb R}^n$ and $R>0$.
\end{lemma}

A direct consequence of the above quantitative information is:
\begin{corollary}\label{cor:140820-1}
Let $0<q<\infty$ and $\varphi:(0,\infty) \to (0,\infty)$
be a function in the class ${\mathcal G}_q$.
If $N_0>n/q$,
then
$(1+|\cdot|)^{-N_0}\in {\mathcal M}^\varphi_q({\mathbb R}^n)$.
\end{corollary}

\begin{proof}
Since $\varphi \in {\mathcal G}_q$,
we have $\varphi(t)t^{-n/q} \le \varphi(1)$
for all $t \ge 1$.
we have
{
\[
\|(1+|\cdot|)^{-N_0}\|_{{\mathcal M}^\varphi_q}
\lesssim
\sum_{j=1}^\infty
\max(1,j-1)^{-N_0}
\|\chi_{Q(j)}\|_{{\mathcal M}^\varphi_q}
\le
\sum_{j=1}^\infty
\max(1,j-1)^{-N_0}
\varphi(j)<\infty.
\]
}
Here for the second inequality
we invoked (\ref{eq:150205-3}).
\end{proof}

We also use the following inequality:
\begin{lemma}\label{eq:141021-10}
Let $0<q<\infty$ and $\varphi:(0,\infty) \to (0,\infty)$
be a function.
Then
\[
(\|f+g\|_{{\mathcal M}^\varphi_q})^{\min(1,q)}
\le
(\|f\|_{{\mathcal M}^\varphi_q})^{\min(1,q)}
+
(\|g\|_{{\mathcal M}^\varphi_q})^{\min(1,q)}
\]
for all $f,g \in {\mathcal M}^\varphi_q({\mathbb R}^n)$.
\end{lemma}

We also verify the relation between
$f$ and $|f|^u$ in the next lemma.
\begin{lemma}
\label{2014-9-4-4}
Let $0<u,q<\infty$ and $\varphi:(0,\infty) \to (0,\infty)$.
Then
\[
\|\,|f|^u\,\|_{{\mathcal M}^\varphi_q}
=
\left(\|f\|_{{\mathcal M}^{\varphi^{1/u}}_{u q}}\right)^{u}
\]
for all 
{
$f \in {\mathcal M}^{\varphi^{1/u}}_{u q}({\mathbb R}^n)$.
}
\end{lemma}

\begin{proof}
Although the proof is simple,
we include it for reader's convenience.
We calculate that
\begin{align*}
\||f|^u\|_{{\mathcal M}^\varphi_q}
&=
\sup_{y \in {\mathbb R}^n, r>0}
\varphi(r)
\left(\frac{1}{|Q(x,r)|}\int_{Q(x,r)}
|f(y)|^{u q}\,dy
\right)^{\frac{1}{q}}
\\
&=
\sup_{y \in {\mathbb R}^n, r>0}
\left(
\varphi(r)^{\frac{1}{u}}
\left(\frac{1}{|Q(x,r)|}\int_{Q(x,r)}
|f(y)|^{u q}\,dy
\right)^{\frac1{qu}}\right)^{u}\\
&=
\left(\|f\|_{{\mathcal M}^{\varphi^{1/u}}_{u q}}\right)^{u},
\end{align*}
{as was to be shown}.
\end{proof}

In addition to general facts above,
we need to exclude some special case
where ${\mathcal M}^\varphi_{q}({\mathbb R}^n)$
is close to $L^\infty({\mathbb R}^n)$.
We invoke the following proposition
from \cite[Lemma 2]{Nakai94}.
\begin{proposition}\label{prop:Nakai}
If a nonnegative locally integrable function
$\varphi$ and a positive constant $C>0$
satisfy
\begin{equation}\label{eq:Nakai-1}
\int_r^\infty \frac{ds}{\varphi(s)s}
\le \frac{C}{\varphi(r)} \quad (r>0),
\end{equation}
then
\begin{equation}\label{eq:Nakai-2}
\int_r^\infty \frac{ds}{\varphi(s)s^{1-\varepsilon}}
\le \frac{C}{1-C\varepsilon}\cdot
\frac{r^\varepsilon}{\varphi(r)} \quad (r>0)
\end{equation}
for all $0<\varepsilon<C^{-1}$.
\end{proposition}

When we consider the vector-valued inequalities,
the following observation will be useful.
\begin{proposition}\label{prop:150312-1}
Let $\varphi$ be a nonnegative locally integrable function
such that there exists a constant $C>0$ such that
$\varphi(s) \le C\varphi(r)$
for all $r,s>0$ with $\dfrac12 \le \dfrac{r}{s} \le 2$.
Then the following are equivalent:
\begin{enumerate}
\item
$\varphi$ satisfies $(\ref{eq:Nakai-1})$.
\item
$\varphi$ satisfies $(\ref{eq:Nakai-2})$
for some $\varepsilon>0$.
\item
There exist constants $\varepsilon>0$ and $C>0$ such that
\begin{equation}\label{eq:Nakai-3}
\frac{t^\varepsilon}{\varphi(t)} 
\le 
\frac{Cr^{\varepsilon}}{\varphi(r)}
\quad (t \ge r)
\end{equation}
\end{enumerate}
If $(1)$--$(3)$ are satisfied,
then
\begin{equation}\label{eq:Nakai-4}
\int_r^\infty \frac{ds}{\varphi(s)^us}
\le \frac{C}{\varphi(r)^u} \quad (r>0)
\end{equation}
for all $0<u<\infty$,
where $C$ depends only on $u$.
\end{proposition}

\begin{proof}
The implication $(1) \Longrightarrow (2)$
follows from Proposition \ref{prop:Nakai}.

Assume (2).
Then we have
\[
\frac{t^\varepsilon}{\varphi(t)} 
{\lesssim}
\int_t^{2t}\frac{dv}{v^{1-\varepsilon}\varphi(v)}
{\lesssim}
\frac{r^\varepsilon}{\varphi(r)}
\]
thanks to the doubling property of $\varphi$,
proving (3).

If we assume (3), then we have
\[
\int_r^\infty \frac{ds}{\varphi(s)s}
=
\int_r^\infty \frac{s^\varepsilon}{\varphi(s)}
\frac{ds}{s^{1+\varepsilon}}
\lesssim
\int_r^\infty \frac{r^\varepsilon}{\varphi(r)}
\frac{ds}{s^{1+\varepsilon}}
= \frac{1}{\varepsilon\varphi(r)},
\]
which implies (1).
Note that (3) also implies (\ref{eq:Nakai-4})
because $\varphi^u$ satisfies (3) as well.
\end{proof}

\begin{remark}
Inequality (\ref{eq:Nakai-4}) is known to be
necessary for (\ref{eq:Nakai-1});
just remark that one can apply
Proposition \ref{prop:150312-1}
to $\varphi^u$.
\end{remark}

\subsection{Vector-valued maximal inequality for ${\mathcal M}^\varphi_q({\mathbb R}^n)$}

Here we prove the following vector-valued inequality:
\begin{theorem}\label{thm:vector-maximal} 
We denote by $M$ the Hardy--Littlewood maximal operator defined by 
\begin{equation}\label{eq:150825-2}
Mf(x)
{\equiv}
\sup_{Q}\frac{\chi_Q(x)}{|Q|}\int_Q|f(y)|{\rm d}y
\end{equation}
for $f\in L^{1}_{{\rm loc}}(\R^n)$, 
where the supremum is taken over all cubes $Q$. 

Let $1<q<\infty$, $1<r \le \infty$ and $\varphi:(0,\infty) \to (0,\infty)$
be a function. 
\begin{enumerate}
\item
For a measurable function $f:{\mathbb R}^n \to {\mathbb C}$,
we have
\begin{equation}\label{eq:140820-204}
\|Mf\|_{{\mathcal M}^\varphi_q}
\lesssim
\|f\|_{{\mathcal M}^\varphi_q}.
\end{equation}
In particular, for any sequence $\{f_j\}_{j=1}^\infty$
of ${\mathcal M}^\varphi_q({\mathbb R}^n)$-functions, 
\begin{equation}\label{eq:140820-203}
\left\|
\sup_{j \in {\mathbb N}}Mf_j
\right\|_{{\mathcal M}^\varphi_q}
\lesssim.
\left\|
\sup_{j \in {\mathbb N}}|f_j|
\right\|_{{\mathcal M}^\varphi_q}
\end{equation}
\item
Assume $(\ref{eq:Nakai-3})$.
Then for any sequence $\{f_j\}_{j=1}^\infty$
of ${\mathcal M}^\varphi_q({\mathbb R}^n)$-functions,
\begin{equation}\label{eq:140820-2041}
\left\|\left(\sum_{j=1}^\infty Mf_j{}^r\right)^{\frac{1}{r}}
\right\|_{{\mathcal M}^\varphi_q}
\lesssim
\left\|\left(\sum_{j=1}^\infty|f_j|^r\right)^{\frac{1}{r}}
\right\|_{{\mathcal M}^\varphi_q}.
\end{equation}
\end{enumerate}
\end{theorem}

\begin{proof}
\
\begin{enumerate}
\item
See \cite[Theorem 2.3]{Sawano08}
for (\ref{eq:140820-204}).
Note that (\ref{eq:140820-203}) is a direct consequence
of (\ref{eq:140820-204}) and
\[
\sup_{j \in {\mathbb N}}Mf_j(x)
\le
M\left[\sup_{j \in {\mathbb N}}|f_j|\right](x)
\quad (x \in {\mathbb R}^n).
\]
\item
See \cite[Theorem 2.5]{Sawano08}
for (\ref{eq:140820-2041}).
\end{enumerate}
\end{proof}

Next, we recall the following fundamental estimate:
\begin{lemma}\label{lem:140820-101}
Let $R$ be a cube.
Then
\begin{equation}\label{150818-1}
M[\chi_R](x) \sim \frac{|R|}{|R|+|x-c(R)|^n}
\quad (x \in {\mathbb R}^n),
\end{equation}
where the implicit constants in $(\ref{150818-1})$
depend only on $n$.
\end{lemma}

\begin{proof}
The proof is standard.
We content ourselves with its outline.
For the proof, we shall distinguish two cases.
\begin{enumerate}
\item
$x \in 3R$.
In this case, we can show that
\[
\frac{1}{3^n} \le M[\chi_R] \le 1, \quad
1 \le 1+\frac{|x-c(R)|}{\ell(R)} \le 1+3n.
\]
\item
$x \in 3^{l+1}R \setminus 3^l R$
for some $l \in {\mathbb N}$.
In this case, we can show that
\[
\frac{1}{3^{(l+1)n}} \le M[\chi_R](x) \le \frac{4^n}{3^{ln}}, \quad
\frac{3^l}{2} \le 1+\frac{|x-c(R)|}{\ell(R)} \le 1+n \cdot 3^{l+1}.
\]
\end{enumerate}
\end{proof}

We present a function of $f \in {\mathcal M}^\varphi_q({\mathbb R}^n)$.
\begin{proposition}
Let $1 \le q<\infty$ and $\varphi \in {\mathcal G}_q$.
Define
\[
f \equiv \sum_{j=-\infty}^\infty \frac{\chi_{[2^{-j-1},2^{-j}]^n}}{\varphi(2^{-j})}, \quad
g \equiv \sum_{j=-\infty}^\infty \frac{\chi_{[0,2^{-j}]^n}}{\varphi(2^{-j})}, \quad
h \equiv \sup_{j \in {\mathbb Z}}\frac{\chi_{[0,2^{-j}]^n}}{\varphi(2^{-j})}.
\]
Define a decreasing function $\varphi^\dagger$ by: 
$\varphi^\dagger(t){\equiv}\varphi(t)t^{-n/q}$ for $t>0$.

\begin{enumerate}
\item
Then the following are equivalent;
\begin{enumerate}
\item[$(a)$]
$f \in {\mathcal M}^\varphi_q({\mathbb R}^n)$,
\item[$(b)$]
$h \in {\mathcal M}^\varphi_q({\mathbb R}^n)$,
\item[$(c)$]
$\varphi^\dagger$ satisfies the integral condition,
or equivalently,
there exists a constant $C>0$ such that
$\displaystyle
\sum_{j=\infty}^l \frac{1}{\varphi^\dagger(2^{-j})} \le \frac{C}{\varphi^\dagger(2^{-l})}
$
{for all $l \in {\mathbb Z}$}.
\end{enumerate}
\item
There exists a constant $C>0$ such that
$\displaystyle
\sum_{j=l}^\infty \frac{1}{\varphi(2^{-j})} \le \frac{C}{\varphi(2^{-l})}
$
and that
$\displaystyle
\sum_{j=\infty}^l \frac{1}{\varphi^\dagger(2^{-j})} \le \frac{C}{\varphi^\dagger(2^{-l})}
$
for all $l \in {\mathbb Z}$ if and only if
$g \in {\mathcal M}^\varphi_q({\mathbb R}^n)$.
\end{enumerate}
\end{proposition}

\begin{proof}
\
\begin{enumerate}
\item
Note that
$f \le h \le 2^n Mf$,
where $M$ denotes the Hardy-Littlewood maximal operator
given by {(\ref{eq:150825-2})}.
Observe that $M$ is bounded on ${\mathcal M}^\varphi_q({\mathbb R}^n)$.
Thus, $(a)$ and $(b)$ are equivalent.
Since $f$ is expressed as $f=f_0(\|\cdot\|_\infty)$,
that is, there exists a function $f_0:[0,\infty) \to {\mathbb R}$ 
such that $f(x)=f_0(\|x\|_{\infty})$
for all $x\in\R^n$,
where $\|\cdot\|_\infty$ denotes the $\ell^\infty$-norm,
it follows that $(a)$ and $(c)$ are equivalent.
\item
Suppose first $g \in {\mathcal M}^\varphi_q({\mathbb R}^n)$.
Then
\[
\varphi(2^{-l})\left(\frac{1}{|[0,2^{-l}]^n|}\int_{[0,2^{-l}]^n}
\left(\sum_{j=l}^\infty \frac{1}{\varphi(2^{-j})}\right)^q\,dx\right)^{\frac{1}{q}}
\le {\|g\|_{{\mathcal M}^\varphi_q}}.
\]
Thus,
$\displaystyle
\sum_{j=l}^\infty \frac{1}{\varphi(2^{-j})} \le 
\frac{\|g\|_{{\mathcal M}^\varphi_q}}{\varphi(2^{-l})}
$
for all $l \in {\mathbb Z}$.
This implies that $f(x) \le h(x) \le g(x) {\lesssim} f(x)$.
Thus, from (1), 
$\displaystyle
\sum_{j=\infty}^l \frac{1}{\varphi^\dagger(2^{-j})} \le 
\frac{C}{\varphi^\dagger(2^{-l})}
$
holds as well.

Conversely, 
assume that 
\begin{equation}\label{eq:150824-3}
\sum_{j=\infty}^l \frac{1}{\varphi^\dagger(2^{-j})} 
\le \frac{C}{\varphi^\dagger(2^{-l})}
\end{equation}
and
\begin{equation}\label{eq:150824-4}
\sum_{j=l}^\infty \frac{1}{\varphi(2^{-j})} \le \frac{C}{\varphi(2^{-l})}
\end{equation}
hold for all $l \in {\mathbb Z}$.
Then we have $g \sim f$ from {(\ref{eq:150824-4})}.
Thus, $f\in {\mathcal M}^\varphi_q({\mathbb R}^n)$ by {(\ref{eq:150824-3})},
from which it follows that $g \in {\mathcal M}^\varphi_q({\mathbb R}^n)$.
\end{enumerate}
\end{proof}

Let $0<\eta<\infty$.
We define the powered Hardy-Littlewood maximal operator 
$M^{(\eta)}$ by
\begin{equation}\label{eq:M-eta}
M^{(\eta)}f(x)
\equiv \sup_{R>0}
\left(
\frac{1}{|Q(x,R)|}\int_{Q(x,R)}|f(y)|^\eta\,dy
\right)^\frac{1}{\eta}
\quad ({x \in {\mathbb R}^n})
\end{equation}
{for a measurable function $f$.}
When we consider the atomic decomposition,
{we use the following observation:}
\begin{lemma}{\rm \cite[Lemma A.2]{FJ90}}\label{lem:140821-1}
Let
$\kappa \ge n$
and
$\varepsilon>0$. 
Then, 
\begin{equation}\label{eq:13.18}
\left|
\sum_{m \in {\mathbb Z}^n}\lambda_{\nu m}
\langle 2^\nu x-m \rangle^{-\kappa-\varepsilon}
\right|
\lesssim_\varepsilon
M^{\left(\frac{n}{\kappa}\right)}
\left[
\sum_{m \in {\mathbb Z}^n}\lambda_{\nu m}\chi_{Q_{\nu m}}
\right](x).
\end{equation}
Here,
$M^{\left(\frac{n}{\kappa}\right)}$
denotes 
is the powered Hardy-Littlewood maximal operator 
with $\eta\equiv \frac{n}{\kappa}$.
\end{lemma}

{
\begin{remark}
The vector-valued inequality is a key ingredient
throughout the present paper.
Probably it is easier to handle Herz spaces
(see Section \ref{s7.3} for the definition)
or Musielak-Orlicz spaces
(see Section \ref{s7.5} for the definition)
than Morrey spaces.
In fact,
Herz spaces and Musielak-Orlicz spaces 
have $L^\infty_{\rm c}({\mathbb R}^n)$
as a dense subspace.
\end{remark}
}

\subsection{A Hardy type inequality}

We will need the following Hardy type inequality
for later consideration:
\begin{proposition}\label{prop:140820-11}
Let $0<r \le \infty$ and $\delta>0$.
Then for all nonnegative sequences $\{A_j\}_{j=1}^\infty$,
\begin{equation}\label{eq:140817-5}
\left(\sum_{k=1}^\infty 
\left(\sum_{j=1}^\infty 2^{-|j-k|\delta}A_j\right)^r
\right)^{\frac1r}
\lesssim
\left(\sum_{k=1}^\infty A_k^r
\right)^{\frac1r}.
\end{equation}
{
If $r=\infty$,
$(\ref{eq:140817-5})$ reads
}
;
\begin{equation}
\label{150815-1}
\sup_{k \in {\mathbb N}}\left(\sum_{j=1}^\infty 2^{-|j-k|\delta}A_j\right)
\lesssim
\sup_{k \in {\mathbb N}}A_k.
\end{equation}
\end{proposition}

\begin{proof}
{
The inequality (\ref{eq:140817-5})
is proved in \cite[Lemma A.2.1]{FJ90},
{while
the inequality (\ref{150815-1})
is known as} the discrete Hardy inequality.
}
\end{proof}

\subsection{A convolution estimate}

We will make use of the following estimate
on integrals.
We define ${\mathbb N}_0=\{0,1,2,\ldots\}$.
\begin{lemma}{\rm \cite[p.466]{Grafakos08}}\label{lem:Grafakos}
Let $\nu,\mu \in {\mathbb Z}$
with $\nu \ge \mu$, 
$M>0$ and $L \in {\mathbb N}_0$, 
and $N>M+L+n$. Suppose that 
a $C^L({\mathbb R}^n)$-function $\varphi$ and $x_\varphi$
are
such that
\begin{equation}\label{eq:140820-8}
|\nabla^L \varphi(x)| \le
\frac{2^{\mu(n+L)}}{(1+2^\mu|x-x_\varphi|)^M}
\end{equation}
for all $x \in {\mathbb R}^n$.
Assume, in addition, 
that $\psi$ is a measurable function such that
\[
\int_{{\mathbb R}^n}x^\beta \psi(x)\,dx=0, \mbox{\rm \, if \, } 
|\beta| \le L-1
\] 
and that, for some $x_\psi \in {\mathbb R}^n$,
\[
|\psi(x)| \le \frac{2^{\nu
n}}{(1+2^\nu|x-x_\psi|)^N}
\]
for all $x \in {\mathbb R}^n$.
Then 
\begin{align*}
\left|\int_{{\mathbb R}^n}\varphi(x)\psi(x)\,dx\right| 
\lesssim
\frac{2^{\mu n-(\nu-\mu)L}}{(1+2^\mu|x_\varphi-x_\psi|)^M}.
\end{align*}
\end{lemma}

Here are examples of applications of Lemma \ref{lem:Grafakos}.
\begin{example}
Let $\Theta,f \in {\mathcal S}({\mathbb R}^n)$.
By using Lemma \ref{lem:Grafakos}
with
\[
\varphi={\mathcal F}^{-1}[\Theta(2^{-k}\cdot)](-\cdot), \quad 
\psi=f, \quad x_\varphi=0, \quad x_\psi=x,
\]
and
\[
L=\mu=0, \quad \nu=k, \quad
M=N_0, \quad
N=[1+N_0],
\]
where $N_0$ is a constant obtained in Corollary \ref{cor:140820-1},
we obtain
\[
|\Theta_k(D)f(x)| \lesssim p_{[1+N_0]}(f)(1+|x|)^{-N_0}.
\]
In fact, for example, we can check (\ref{eq:140820-8})
as follows:
\begin{align*}
|\nabla^L \varphi(x)|
&=
|\nabla^L[2^{kn}{\mathcal F}^{-1}\tau(-2^k\cdot)](x)|\\
&=
2^{kn+kL}|\nabla^L[{\mathcal F}^{-1}\tau](-2^kx)|\\
&\lesssim
2^{kn+kL}
p_{[1+N_0]}({\mathcal F}^{-1}\tau)
(1+2^k|x|)^{-N_0}.
\end{align*}
\end{example}

\subsection{{Plancherel-Polya-Nikol'skii inequality}}

Let $\Omega$ be a compact subset of ${\mathbb R}^n$. 
Recall that $\mathcal{S}_{\Omega}({\mathbb R}^n)$ denotes the space of all elements $\varphi\in\mathcal{S}({\mathbb R}^n)$ which satisfies 
${\rm supp\,}{\mathcal F}\varphi\subset\Omega$. 

The following inequality will be used
throughout the present paper.

\begin{theorem}[Plancherel-Polya-Nikol'skii inequality, 
{\rm \cite[Theorem 1.3.1, Section 1.4.1]{Triebel1}}]
\label{thm:PPN}
Let
$\eta>0$
and $\varphi$ be a function in ${\mathcal S}_{Q(1)}$.
Then we have
\[
\sup_{z \in {\mathbb R}^n}
\frac{|\varphi(x-z)|}{(1+|z|)^{\frac{n}{\eta}}}
\lesssim_\eta
M^{(\eta)}\varphi(x),
\]
where $M^{(\eta)}$ is the maximal function given by
$(\ref{eq:M-eta})$.
{In particular,} 
for any $R, \eta>0$ and $\varphi\in\mathcal{S}_{{Q}(R)}$, 
we have the pointwise estimate 
\begin{align*}
\sup_{z \in {\mathbb R}^n}
\frac{|\varphi(x-z)|}{(1+R|z|)^{\frac{n}{\eta}}}
\lesssim_\eta
M^{(\eta)}\varphi(x).
\end{align*}

\end{theorem}

Here we recall a typical application
of the above theorem.
\begin{example}\label{example:140820-3}
Let $\Theta$ be a function supported in $Q(2)$.
Define
$\Theta_j(\xi) \equiv \Theta(2^{-j}\xi)$
for $j \in {\mathbb N}_0$.
Then let us prove
\begin{equation}\label{150815-2}
|\Theta_j(D)f(x)|
\lesssim
\frac{1}{\varphi(2^{-j})}\|\Theta_j(D)f\|_{{\mathcal M}^\varphi_q}.
\end{equation}
In fact,
for any points $x$ and $y$ satisfying
$|x-y| \le 2^{-j}$
we have
\begin{equation}\label{eq:140820-11}
|\Theta_j(D)f(x)|
\lesssim
\sup_{z \in {\mathbb R}^n}
\frac{|\Theta_j(D)f(z)|}{(1+2^j|z-y|)^{2n/q}}
\lesssim
M^{(q/2)}[\Theta_j(D)f](y).
\end{equation}
Therefore,
for all $x \in {\mathbb R}^n$,
\begin{align*}
|\Theta_j(D)f(x)|
&\le
\left(2^{j{(n+1)}}
\int_{Q(x,2^{-j})}
M^{(q/2)}[\Theta_j(D)f](y)\,dy
\right)^{\frac{1}{q}}\\
&=
\frac{1}{\varphi(2^{-j})}
\cdot
\varphi(2^{-j})
\left(2^{j{(n+1)}}
\int_{Q(x,2^{-j})}
M^{(q/2)}[\Theta_j(D)f](y)\,dy
\right)^{\frac{1}{q}}
\end{align*}
{
thanks to (\ref{eq:150825-1}).
By using the Morrey norm and Lemma \ref{2014-9-4-4},
we obtain}
\begin{align*}
|\Theta_j(D)f(x)|
&\le 
\frac{1}{\varphi(2^{-j})}
\cdot
\|M^{(q/2)}[\Theta_j(D)f]\|_{{\mathcal M}^\varphi_q}\\
&\le 
\frac{1}{\varphi(2^{-j})}
\cdot
(\|M[|\Theta_j(D)f|^{q/2}]\|_{{\mathcal M}^{\varphi^{q/2}}_2})^{2/q}\\
&\lesssim
\frac{1}{\varphi(2^{-j})}
\cdot
(\||\Theta_j(D)f|^{q/2}\|_{{\mathcal M}^{\varphi^{q/2}}_2})^{2/q}\\
&{=
\|\Theta_j(D)f\|_{{\mathcal M}^\varphi_q}},
\end{align*}
{
which proves (\ref{150815-2}).
}
\end{example}

The following result is a consequence
of the maximal inequalities 
{in Theorem \ref{thm:vector-maximal}}
and the Plancherel-Polya-Nikol'skii
inequality.
\begin{theorem}[Multiplier result]
\label{Multiplier result}
Let $0<q < \infty$, $0< r\le \infty$,
$s \in {\mathbb R}$ and $\varphi\in {\mathcal G}_q$
and
let
$$
\nu>\frac{n}{\min(1,q,r)}+\frac{n}{2}.
$$
\begin{enumerate}
\item
The following inequality is true:
\begin{equation}
\label{2014-9-8-1} 
\|2^{js}H_{(j)}(D)f_j^*\|_{{\mathcal M}^\varphi_q}
\lesssim \left( \sup_{k\in{\mathbb N} } \|H_{(k)}(d_k\cdot)\|_{H^{\nu}_{2}}\right)
\|2^{js}f_j\|_{{\mathcal M}^\varphi_q}
\end{equation}
holds for any $j\in{\mathbb N}$. 
\item
{Assume $(\ref{eq:Nakai-3})$ in addition.}
Suppose that,
for each $j=1,2,\ldots$, 
we are given
a compact set $K_j$ of ${\mathbb R}^n$
with diameter $d_j$,
$H_{(j)}\in H_2^{\nu}({\mathbb R}^n)$
and
$f_j\in {\mathcal M}^\varphi_q({\mathbb R}^n)$
such that
${\rm supp} ({\mathcal F}f_j)\subseteq K_j$.
\noindent
Define
{
\[
H_{(j)}(D)f_j^*(x)
\equiv
\sup_{z \in {\mathbb R}^n}
\frac{|{\mathcal F}^{-1}[H_{(j)}{\mathcal F} f_j](x-z)|}
{(1+d_j|z|)^{n/\eta}},
\]
where $\eta{\equiv}\min{(1,q,r)}/2$.
}
If a collection $\{ f_j\}_{j=1}^{\infty}$
of measurable functions satisfies
$$ \left\|
\left(\sum_{j=1}^\infty
2^{jsr}|f_j|^r
\right)^{\frac1r}\right\|_{{\mathcal M}^\varphi_q}<\infty,
$$ 
then we have 
\begin{equation}
\label{2014-9-4-3} 
{
\left\|
\left(\sum_{j=1}^\infty
2^{jsr}|\left(H_{(j)}(D)f_j\right)^*|^r
\right)^{\frac1r}\right\|_{{\mathcal M}^\varphi_q}
}
\lesssim \left( \sup_{k\in{\mathbb N} } \|H_{(k)}(d_k\cdot)\|_{H^{\nu}_{2}}\right)
 \left\|
\left(\sum_{j=1}^\infty
2^{jsr}|f_j|^r
\right)^{\frac1r}\right\|_{{\mathcal M}^\varphi_q}. 
\end{equation}
\end{enumerate}
\end{theorem}

\begin{proof}
{The heart of the matter is to prove
(\ref{eq:150825-5}) and (\ref{eq:150825-6}) below.
For the sake of the convenience for readers,
we prove 
(\ref{eq:150825-5}) and (\ref{eq:150825-6}).}
Let $\delta$ satisfy $\nu=\frac{n}{\eta}+\frac{n+\delta}{2}$. 
By the definition of $H_{(j)}(D)f_j$, we see that 
\begin{align}
\lefteqn{
| 2^{js}{\mathcal F}^{-1}[H_{(j)}{\mathcal F} f_j](x-z)|
}\nonumber\\ 
&\lesssim 
\int_{{\mathbb R}^n}2^{js}
\frac{|({\mathcal F}^{-1}H_{(j)})(x-z-y)|}{(1+d_j|x-y|)^{n/\eta}}|f_j(y)|(1+d_j|x-y|)^{n/\eta} dy \notag \\ 
&\lesssim 
\sup_{u\in{\mathbb R}^n}\frac{2^{js}|f_j(u)|}{(1+d_j|x-u|)^{n/\eta}} 
\int_{{\mathbb R}^n} 
|({\mathcal F}^{-1}H_{(j)})(x-z-y)|(1+d_j|x-y|)^{n/\eta} dy. \notag 
\end{align}
Therefore, we have 
\begin{align}
\lefteqn{
\frac{| 2^{js}{\mathcal F}^{-1}[H_{(j)}{\mathcal F} f_j](x-z)|}
{(1+d_j|z|)^{n/\eta}}
}\nonumber\\
&\lesssim 
\sup_{u\in{\mathbb R}^n}\frac{2^{js}|f_j(u)|}{(1+d_j|x-u|)^{n/\eta}} 
\int_{{\mathbb R}^n} 
|({\mathcal F}^{-1}H_{(j)})(x-z-y)|\left(\frac{1+d_j|x-y|}{1+d_j|z|}\right)^{n/\eta} dy \notag \\ 
&\lesssim 
\sup_{u\in{\mathbb R}^n}\frac{2^{js}|f_j(u)|}{(1+d_j|x-u|)^{n/\eta}} 
\int_{{\mathbb R}^n} 
|({\mathcal F}^{-1}H_{(j)})(x-z-y)|(1+d_j|x-y-z|)^{n/\eta} dy, \notag 
\end{align}
where we used $1+d_j|x-y|\le (1+d_j|x-y-z|)(1+d_j|z|)$. 
After we apply the H\"older inequality to the integrand as above, 
we use the chain rule of differentiation and 
${\mathcal F}f(d_j\cdot)=d_j^{-n}{\mathcal F}[f(d_j^{-1}\cdot)]$ ($d_j>0$), 
then we obtain 
\begin{align}
\frac{| 2^{js}{\mathcal F}^{-1}[H_{(j)}{\mathcal F} f_j](x-z)|}
{(1+d_j|z|)^{n/\eta}}
&
\lesssim 
\| H_{(j)}(d_j\cdot)\|_{H^{\nu}_{2}}
\sup_{u\in{\mathbb R}^n}\frac{2^{js}|f_j(x-u)|}{(1+d_j|u|)^{n/\eta}}. 
\notag 
\end{align}
If we combine this estimate with 
Plancherel-Polya-Nikol'skii inequality (Theorem \ref{thm:PPN}),
then we have
\begin{equation}\label{eq:150825-5}
\frac{| 2^{js}{\mathcal F}^{-1}[H_{(j)}{\mathcal F} f_j](x-z)|}
{(1+d_j|z|)^{n/\eta}}
\lesssim 
\left(
\sup_{k \in {\mathbb N}}
\| H_{(k)}(d_k\cdot)\|_{H^{\nu}_{2}}
\right)M^{(\eta)}[2^{js}f_j](x)
\end{equation}
and hence
\begin{equation}\label{eq:150825-6}
\left(\sum_{j=1}^\infty
2^{jsr}{\mathcal F}^{-1}[H_{(j)}{\mathcal F} f_j]^*(x)^r
\right)^{\frac1r}
\lesssim 
\left(
\sup_{k \in {\mathbb N}}
\| H_{(k)}(d_k\cdot)\|_{H^{\nu}_{2}}
\right)
\left(\sum_{j=1}^\infty
M^{(\eta)}[2^{js}f_j](x)^r
\right)^{\frac1r}.
\end{equation}
If we consider the ${\mathcal M}^\varphi_q({\mathbb R}^n)$-norm,
\begin{align*}
\lefteqn{
\left\|
\left(\sum_{j=1}^\infty
(2^{js}{\mathcal F}^{-1}[H_{(j)}{\mathcal F} f_j]^*)^r
\right)^{\frac1r}\right\|_{{\mathcal M}^\varphi_q}
}\\
&\lesssim 
\left(
\sup_{k \in {\mathbb N}}
\| H_{(k)}(d_k\cdot)\|_{H^{\nu}_{2}}
\right)
\left\|
\left(\sum_{j=1}^\infty
M^{(\eta)}[2^{js}f_j]^r
\right)^{\frac1r}
\right\|_{{\mathcal M}^\varphi_q}\\
&\lesssim 
\left(
\sup_{k \in {\mathbb N}}
\| H_{(k)}(d_k\cdot)\|_{H^{\nu}_{2}}
\right)
\left(
\left\|
\left(\sum_{j=1}^\infty
M[|2^{js}f_j|^\eta]^{r/\eta}
\right)^{\eta/r}
\right\|_{{\mathcal M}^{\varphi/\eta}_{q/\eta}}
\right)^{1/\eta}.
\end{align*}
Hence 
by combining 
Lemma \ref{2014-9-4-4} and Theorem \ref{thm:vector-maximal}, 
we obtain the desired inequality (\ref{2014-9-4-3}).
\end{proof}

\subsection{Reproducing formula}

Rychkov proved the following reproducing formula:
\begin{proposition}\label{prop:Rychkov}
Suppose that $\varphi_0 \in C^\infty_{\rm c}({\mathbb R}^n)$ 
with $\displaystyle{ \int_{{\mathbb R}^n} \varphi_0(x)\,dx \neq 0}$.
Set 
\begin{equation}\label{eq:140817-6}
\varphi_j(x)\equiv 2^{j n}\varphi_0(2^j x)-2^{(j-1)n}\varphi_0(2^{j-1}x)
\quad (x \in {\mathbb R}^n)
\end{equation}
for $j \in {\mathbb N}$.
Let $L \in {\mathbb N}$.
Then there exists $\psi_0 \in C^\infty_{\rm c}({\mathbb R}^n)$
such that 
\[
\int_{{\mathbb R}^n}x^\alpha \psi_1(x)\,dx=0
\] 
for all $|\alpha| \le L$
and that
\begin{equation}
\label{eq:reproducing1}
f=\sum_{j \in {\mathbb N}_0}\psi_j*\varphi_j*f \mbox{\rm \, in \, } 
{\mathcal S}'({\mathbb R}^n)
\end{equation}
for all $f \in {\mathcal S}'({\mathbb R}^n)$.
Here 
\begin{equation}\label{eq:140817-7}
\psi_j(x)\equiv 2^{j n}\psi_0(2^j x)-2^{(j-1)n}\psi_0(2^{j-1}x)
\quad (x \in {\mathbb R}^n)
\end{equation}
for $j \in {\mathbb N}$.
\end{proposition}

\begin{proof}
See \cite{Ry2}.
\end{proof}

Before we go further, {a couple of} remarks
may be in order.
\begin{remark}
\
\begin{enumerate}
\item
A rescaling {argument} allows us to assume
that $\varphi_0$ and $\psi_0$ are supported in $[-1/4,1/4]^n$.
\item
Let $L_1 \in {\mathbb N}$ be an arbitrary number.
In Proposition \ref{prop:Rychkov},
we can assume that there exists
$\Phi$ such that $\Delta^{L_1} \Phi \equiv \varphi_1$.
As a result,
\[
\int_{{\mathbb R}^n}x^\beta \varphi_1(x)\,dx=0
\]
for all $\beta$ with $|\beta| \le 2L_1-1$.
\end{enumerate}
\end{remark}

The next lemma explains how to {construct}
atoms.
\begin{lemma}
Let $\{\varphi_j\}_{j \in {\mathbb N}_0}$ and $\{\psi_j\}_{j \in {\mathbb N}_0}$
as above.
Assume in addition that
\[
\int_{{\mathbb R}^n}x^\beta \varphi_1(x)\,dx=0
\]
for all $\beta$ with $|\beta| \le L$.
Define
\begin{equation}\label{eq:140821-11}
{{\gamma_{jm}(x) }}
\equiv 
\int_{Q_{jm}}\varphi_j(x-y) \cdot f*\psi_j(y)\,dy
\quad (x \in {\mathbb R}^n)
\end{equation}
for $f \in {\mathcal S}'({\mathbb R}^n)$
and $j \in {\mathbb N}_0$ and $m \in {\mathbb Z}^n$.
Then we have;
\begin{enumerate}
\item
${{\gamma_{jm} }}\in C^\infty({\mathbb R}^n)$
for all $j \in {\mathbb N}_0$ and $m \in {\mathbb Z}^n$,
\item
${\rm supp}({{\gamma_{jm}}}) \subset 3Q_{jm}$
for all $j \in {\mathbb N}_0$ and $m \in {\mathbb Z}^n$,
\item
$\displaystyle
\int_{{\mathbb R}^n}x^\beta {{\gamma_{jm}(x)}}\,dx=0
$
for all
{$\beta\in \N_0^{\ n}$ with $|\beta|\le L$,}
$j \in {\mathbb N}$ and $m \in {\mathbb Z}^n$.
\end{enumerate}
\end{lemma}

\begin{proof}
{All the assertions are easy to check.}
For example,
we can check
$(2)$
as follows:
\begin{align*}
{\rm supp}({{\gamma_{jm}}}) 
&\subset 
Q_{jm}+[-2^{-j-2},2^{-j-2}]^n\\
&=
2^{-j}m+[0,2^{-j})^n+[-2^{-j-2},2^{-j-2}]^n\\
&=
2^{-j}m+[-2^{-j-2},5 \cdot 2^{-j-2})^n\\
&=
3Q_{jm}
\end{align*}
for all $j \in {\mathbb N}_0$ and $m\in{\mathbb Z}^n$.
\end{proof}

\section{Generalized Triebel-Lizorkin-Morrey spaces on ${\mathbb R}^n$}
\label{s3}

\subsection{Proof of Theorem \ref{thm:150205-1}}

{We start with a setup.}
We recall that $\theta$ and $\tau$ are compactly supported functions
satisfying
\[
0 \notin {\rm supp}(\tau), \quad
\theta(\xi)>0 \mbox{\rm \, if \, } \xi \in Q(2), \quad
\tau(\xi)>0 \mbox{\rm \, if \,} \xi \in Q(2) \setminus Q(1).
\]

Let $\tilde{\theta}$ and $\tilde{\tau}$ be compactly supported functions
satisfying
\[
0 \notin {\rm supp}(\tilde{\tau}), \quad
\tilde{\theta}(\xi)>0 \mbox{\rm \, if \, } \xi \in Q(2), \quad
\tilde{\tau}(\xi)>0 \mbox{\rm \, if \,} \xi \in Q(2) \setminus Q(1).
\]
We define 
\begin{equation}\label{eq:150821-112}
\tau_k(\xi) \equiv \tau(2^{-k}\xi),
\end{equation} 
and
\begin{equation}\label{eq:150821-113}
\tilde{\tau}_k(\xi) \equiv \tilde{\tau}(2^{-k}\xi)
\end{equation}
for $\xi \in {\mathbb R}^n$ and $k \in {\mathbb N}$.

We define 
$\displaystyle \|f\|_{{\mathcal A}_{{\mathcal M}^\varphi_q,r}^s(\theta, \tau)}
\equiv \|f\|_{{\mathcal A}_{{\mathcal M}^\varphi_q,r}^s}$ 
as in (\ref{eq:140820-140}) and (\ref{eq:140820-141}). 

By the symmetry, it is sufficient to prove that 
\begin{equation}
\label{2014-9-4-1}
\|f\|_{{\mathcal A}_{{\mathcal M}^\varphi_q,r}^s(\theta, \tau)} 
\lesssim 
\|f\|_{{\mathcal A}_{{\mathcal M}^\varphi_q,r}^s(\tilde{\theta}, \tilde{\tau})}. 
\end{equation}

{With the above setup in mind,} we prove 
\begin{equation}
\label{2014-9-4-2}
\left(\sum_{j=1}^\infty
2^{(j+3)sr}\|\tau_{j+3}(D)f\|_{{\mathcal M}^\varphi_q}^r
\right)^{\frac1r}
\lesssim 
\left(\sum_{j=1}^\infty
2^{jsr}\|\tilde{\tau}_j(D)f\|_{{\mathcal M}^\varphi_q}^r
\right)^{\frac1r}. 
\end{equation}
Once (\ref{2014-9-4-2}) is proved,
we can prove
\begin{equation}\label{eq:150821-111}
\left(\sum_{j=1}^3
2^{j sr}\|\tau_{j}(D)f\|_{{\mathcal M}^\varphi_q}{}^r
\right)^{\frac1r}
\lesssim
\|\tilde{\theta}(D)f\|_{{\mathcal M}^\varphi_q}
\left(\sum_{j=1}^\infty
2^{jsr}\|\tilde{\tau}_j(D)f\|_{{\mathcal M}^\varphi_q}{}^r
\right)^{\frac1r}
\end{equation}
similarly to (\ref{2014-9-4-2}).

We can prove (\ref{2014-9-4-2})
with the help of Theorem \ref{Multiplier result}.
If $j\ge 4$, we see that 
\begin{align*}
\tau_{j}(D)f 
&= 
{\mathcal F}^{-1}[\tau_j {\mathcal F}f] \\ 
&= {\mathcal F}^{-1}
\left[ (\tilde{\tau}_{j-1}+\tilde{\tau}_{j}+\tilde{\tau}_{j+1})\frac{\tau_j}{(\tilde{\tau}_{j-1}+\tilde{\tau}_{j}+\tilde{\tau}_{j+1})}{\mathcal F}f \right] \\ 
&= \frac{\tau_j}{(\tilde{\tau}_{j-1}+\tilde{\tau}_{j}+\tilde{\tau}_{j+1})}(D)
{\mathcal F}^{-1} [(\tilde{\tau}_{j-1}+\tilde{\tau}_{j}+\tilde{\tau}_{j+1}){\mathcal F}f] \\ 
&= \left( \frac{\tau}{\tilde{\tau}_{-1}+\tilde{\tau}+\tilde{\tau}_{1}}\right)_j(D)
{\mathcal F}^{-1} [(\tilde{\tau}_{j-1}+\tilde{\tau}_{j}+\tilde{\tau}_{j+1}){\mathcal F}f]. 
\end{align*}
By Theorem \ref{Multiplier result}, we have the desired assertion.

\subsection{Fundamental properties}

First, we note that the following
$\min(1,q,r)$-triangle inequality holds.
The proof being standard, we omit the proof.
\begin{lemma}\label{lem:141020-1}
Let $0<q<\infty$, $0<r \le \infty$, $s \in {\mathbb R}$
 and $\varphi \in {\mathcal G}_q$.
Assume
$(\ref{eq:Nakai-3})$
in the case
when
${\mathcal A}^s_{{\mathcal M}^\varphi_q,r}({\mathbb R}^n)
={\mathcal E}^s_{{\mathcal M}^\varphi_q,r}({\mathbb R}^n)$
with $r<\infty$.
Then
\[
(\|f_1+f_2\|_{{\mathcal A}^s_{{\mathcal M}^\varphi_q,r}})^{\min(1,q,r)}
\le
(\|f_1\|_{{\mathcal A}^s_{{\mathcal M}^\varphi_q,r}})^{\min(1,q,r)}
+
(\|f_2\|_{{\mathcal A}^s_{{\mathcal M}^\varphi_q,r}})^{\min(1,q,r)}
\]
\end{lemma}

The next proposition deals with the lifting property.
\begin{proposition}[Lift operator, Lifting property]\label{prop:140820-1}
Let $0<q<\infty$, $0<r \le \infty$, $s \in {\mathbb R}$
 and $\varphi \in {\mathcal G}_q$.
Assume in addition
that $\varphi$ satisfies
$(\ref{eq:Nakai-3})$
when $r<\infty$ and ${\mathcal A}={\mathcal E}$.
Then
\[
(1-\Delta)^{-M/2}:{\mathcal A}^s_{{\mathcal M}^\varphi_q,r}({\mathbb R}^n)
\to {\mathcal A}^{s+M}_{{\mathcal M}^\varphi_q,r}({\mathbb R}^n)
\]
is an isomorphism.
\end{proposition}

\begin{proof}
This is a consequence of Theorem \ref{Multiplier result}
or \cite[{Theorem 3.10}]{LSUYY2}.
\end{proof}

Next, we verify the embedding properties.
\begin{proposition}\label{prop:140820-2}
Let $0<q<\infty$, $0<r_1,r_2 \le \infty$, $s \in {\mathbb R}$,
$\varepsilon>0$ and $\varphi \in {\mathcal G}_q$.
Then
\[
{\mathcal A}^s_{{\mathcal M}^\varphi_q,r_1}({\mathbb R}^n)
\hookrightarrow
{\mathcal A}^{s-\varepsilon}_{{\mathcal M}^\varphi_q,r_2}({\mathbb R}^n).
\]
\end{proposition}

\begin{proof}
When $r_1\le r_2$, 
then it is easy to see that the desired inequality holds 
by 
$\ell^{r_1}{({\mathbb N}_0)} \hookrightarrow \ell^{r_2}{({\mathbb N}_0)}$. 
In the case of $r_1>r_2$, we can prove the desired inequality 
by using {the} same argument of Besov and Triebel-Lizorkin spaces situations. 
So we omit the proof.
\end{proof}

Next, we investigate the relation
between
$
{\mathcal S}({\mathbb R}^n),
{\mathcal A}^s_{{\mathcal M}^\varphi_q,r}({\mathbb R}^n)
$
and
${\mathcal S}'({\mathbb R}^n)$.

We begin with the following quantitative estimate:
\begin{lemma}\label{lem:150203-11}
Let $0<q<\infty, 0<r \le \infty$ and $\varphi \in {\mathcal G}_q$.
Assume that $s>0$ is such that
\begin{equation}\label{eq:150213-1}
\sum_{j=1}^\infty \frac{1}{2^{js}\varphi(2^{-j})}<\infty.
\end{equation}
Then
\[
{\mathcal A}^s_{{\mathcal M}^\varphi_q,r}({\mathbb R}^n)
\hookrightarrow 
B^0_{\infty 1}({\mathbb R}^n).
\]
In particular, for such $s$,
\[
{\mathcal A}^s_{{\mathcal M}^\varphi_q,r}({\mathbb R}^n)
\hookrightarrow 
B^0_{\infty 1}({\mathbb R}^n)
\hookrightarrow
{\rm BUC}({\mathbb R}^n)
\hookrightarrow
{\mathcal S}'({\mathbb R}^n).
\]
\end{lemma}

\begin{proof}
Let 
$f \in {\mathcal N}^{s}_{{\mathcal M}^\varphi_q,\infty}({\mathbb R}^n)$.
Then we have
\[
|\tau_j(D)f(x)|
\lesssim
\frac{1}{\varphi(2^{-j})}
\cdot
\|f\|_{{\mathcal N}^{s}_{{\mathcal M}^\varphi_q,\infty}}
\]
for all $x \in {\mathbb R}^n$
thanks to Example \ref{example:140820-3}.
Likewise, we have
\[
|\theta(D)f(x)|
\lesssim
\left(
\int_{Q(x,1)}M[|\theta(D)f|^{q/2}](y)^2\,dy
\right)^{\frac1q}
\lesssim
\frac{1}{\varphi(1)}
\cdot
\|f\|_{{\mathcal N}^{s}_{{\mathcal M}^\varphi_q,\infty}}.
\]
Thus,
we have 
${\mathcal A}^s_{{\mathcal M}^\varphi_q,r}({\mathbb R}^n)
\hookrightarrow 
B^0_{\infty 1}({\mathbb R}^n)
\hookrightarrow
{\rm BUC}({\mathbb R}^n)
\hookrightarrow
{\mathcal S}'({\mathbb R}^n)$.
\end{proof}

Condition (\ref{eq:150213-1}) is a natural one
as the following remark implies:
\begin{remark}
\
\begin{enumerate}
\item
(\ref{eq:150213-1}) is satisfied
when $s>\frac{n}{q}$.
\item
(\ref{eq:150213-1}) is also necessary
for
${\mathcal N}^s_{{\mathcal M}^\varphi_q,\infty}({\mathbb R}^n)
\hookrightarrow B^0_{\infty 1}({\mathbb R}^n)$.
In fact,
if $\zeta \in {\mathcal S}$ is such that
$\chi_{Q(1.6) \setminus Q(1.5)} \le \zeta \le \chi_{Q(1.7) \setminus Q(1.4)}$,
then
\[
f=\sum_{j=1}^\infty 
\frac{({\mathcal F}^{-1}\zeta)(2^j\cdot)}{\varphi(2^{-j})}
\in {\mathcal A}^s_{{\mathcal M}^\varphi_q,\infty}({\mathbb R}^n),
\]
since
\[
\|f\|_{{\mathcal N}^s_{{\mathcal M}^\varphi_q,\infty}}
=
\sup_{j \in {\mathbb N}}
\left\|
\frac{({\mathcal F}^{-1}\zeta)(2^j\cdot)}{\varphi(2^{-j})}
\right\|_{{\mathcal M}^\varphi_q}
\lesssim
\sup_{j \in {\mathbb N}}
\left\|
\frac{[M[\chi_{Q(2^{-j})}])^{2/q}}{\varphi(2^{-j})}
\right\|_{{\mathcal M}^\varphi_q}
\lesssim 1.
\]
Meanwhile,
we have
\[
\|f\|_{B^0_{\infty 1}}
\sim
\sum_{j=1}^\infty \frac{1}{\varphi(2^{-j})2^{js}}.
\]
\end{enumerate}
\end{remark}

\begin{proposition}
Let $0<q<\infty$, $0<r \le \infty$, $s \in {\mathbb R}$
 and $\varphi \in {\mathcal G}_q$.
Assume in addition
that $\varphi$ satisfies
$(\ref{eq:Nakai-3})$
when $r<\infty$ and ${\mathcal A}={\mathcal E}$.
Then
\[
{\mathcal S}({\mathbb R}^n) 
\hookrightarrow 
{\mathcal A}^s_{{\mathcal M}^\varphi_q,r}({\mathbb R}^n)
\hookrightarrow 
{\mathcal S}'({\mathbb R}^n)
\]
in the sense of continuous embeddings.
\end{proposition}

\begin{proof}
Let us prove
${\mathcal S}({\mathbb R}^n)
\hookrightarrow 
{\mathcal A}^s_{{\mathcal M}^\varphi_q,r}({\mathbb R}^n)$.
The key observation is (\ref{eq:150205-3}).
Once this is obtained,
we can resort to \cite[Theorem 3.17]{LSUYY2}.
Here for the sake of convenience for readers 
we outline the proof of \cite[Theorem 3.17]{LSUYY2}
by adapting it to our setting.

Since 
\begin{equation}\label{eq:140820-5}
(1-\Delta)^{-M/2}:
{\mathcal A}^s_{{\mathcal M}^\varphi_q,r}({\mathbb R}^n)
\to
{\mathcal A}^{s+M}_{{\mathcal M}^\varphi_q,r}({\mathbb R}^n)
\end{equation}
is an isomorphism for all $M>0$
by virtue of Proposition \ref{prop:140820-1} and
\begin{equation}\label{eq:140820-3}
{\mathcal N}^s_{{\mathcal M}^\varphi_q,\infty}({\mathbb R}^n)
\hookrightarrow 
{\mathcal A}^{s-\varepsilon}_{{\mathcal M}^\varphi_q,r}({\mathbb R}^n).
\end{equation}
in the sense of continuous embedding
thanks to Proposition \ref{prop:140820-2} for all $\varepsilon>0$,
we have only to prove
\begin{equation}\label{eq:140820-4}
{\mathcal S}({\mathbb R}^n)
\hookrightarrow
{\mathcal N}^s_{{\mathcal M}^\varphi_q,\infty}({\mathbb R}^n)
\end{equation}
for all $s \le 0$.
Indeed, combining (\ref{eq:140820-3}) and (\ref{eq:140820-4}),
we obtain
\[
{\mathcal S}({\mathbb R}^n)
\hookrightarrow 
{\mathcal A}^{s-\varepsilon}_{{\mathcal M}^\varphi_q,r}({\mathbb R}^n).
\]
By the use of (\ref{eq:140820-5}),
we have
\[
{\mathcal S}({\mathbb R}^n)
=
(1-\Delta)^{-M/2}[{\mathcal S}({\mathbb R}^n)]
\hookrightarrow 
(1-\Delta)^{-M/2}[{\mathcal A}^{s-\varepsilon}_{{\mathcal M}^\varphi_q,r}({\mathbb R}^n)]
={\mathcal A}^{s+M-\varepsilon}_{{\mathcal M}^\varphi_q,r}({\mathbb R}^n).
\]
Thus, the matters are reduced
to proving (\ref{eq:140820-4}).

Let 
$f \in {\mathcal S}({\mathbb R}^n)$
and
let $N_0$ be a constant obtained in Corollary \ref{cor:140820-1}.
Then according to Lemma \ref{lem:Grafakos},
we have
\[
|\theta(D)f(x)|+|\tau_k(D)f(x)|
\lesssim p_{[1+N_0]}(f)
(1+|x|)^{-N_0},
\]
where the implicit constant in $\lesssim$ does not depend on $k$.
Thus,
\begin{align*}
\|f\|_{{\mathcal N}^s_{{\mathcal M}^\varphi_q,\infty}}
&=
\|\theta(D)f\|_{{\mathcal M}^\varphi_q}
+
\sup_{k \in {\mathbb N}}2^{ks}
\|\tau_k(D)f\|_{{\mathcal M}^\varphi_q}\\
&\lesssim p_{[1+N_0]}(f)
\left(
\|(1+|\cdot|)^{-N_0}\|_{{\mathcal M}^\varphi_q}
+
\sup_{k \in {\mathbb N}}2^{ks}
\|(1+|\cdot|)^{-N_0}\|_{{\mathcal M}^\varphi_q}
\right)\\
&\lesssim p_{[1+N_0]}(f),
\end{align*}
which proves 
${\mathcal S}({\mathbb R}^n)
\hookrightarrow 
{\mathcal A}^s_{{\mathcal M}^\varphi_q,r}({\mathbb R}^n)$.

Let us now prove
${\mathcal A}^s_{{\mathcal M}^\varphi_q,r}({\mathbb R}^n)
\hookrightarrow 
{\mathcal S}'({\mathbb R}^n)$.
Since 
\[
(1-\Delta)^{-M/2}:
{\mathcal A}^s_{{\mathcal M}^\varphi_q,r}({\mathbb R}^n)
\to
{\mathcal A}^{s+M}_{{\mathcal M}^\varphi_q,r}({\mathbb R}^n)
\]
is an isomorphism for all $M>0$ and
\[
{\mathcal A}^s_{{\mathcal M}^\varphi_q,r}({\mathbb R}^n)
\hookrightarrow 
{\mathcal N}^{s}_{{\mathcal M}^\varphi_q,\infty}({\mathbb R}^n)
\]
in the sense of continuous embedding,
we have only to prove
\[
{\mathcal N}^{s}_{{\mathcal M}^\varphi_q,\infty}({\mathbb R}^n)
\hookrightarrow 
{\mathcal S}'({\mathbb R}^n)
\]
for $s\gg 1$,
which is already done in Lemma \ref{lem:150203-11}.
\end{proof}

\section{Decompositions}
\label{s4}

\subsection{Atomic decomposition}

We consider the atomic decomposition.
\begin{definition}
Let $0<q<\infty$, $0<r \le \infty$, $s \in {\mathbb R}$
 and $\varphi \in {\mathcal G}_q$.
\begin{enumerate}
\item
{\it The $($nonhomogeneous$)$ generalized Besov-Morrey sequence space}
${\bf n}_{{\mathcal M}^\varphi_q,r}^s({\mathbb R}^n)$
is the set of all
{
doubly indexed sequences
}
$\lambda
=\{\lambda_{jm}\}_{j \in {\mathbb N}_0, \, m \in {\mathbb Z}^n}$
for which the quasi-norm
\[
\|\lambda\|_{{\bf n}_{{\mathcal M}^\varphi_q,r}^s}
\equiv 
\begin{cases}
\displaystyle
\left(\sum_{j=0}^\infty
2^{jsr}
\left\|\sum_{m \in {\mathbb Z}^n}\lambda_{jm}\chi_{Q_{jm}}
\right\|_{{\mathcal M}^\varphi_q}^r
\right)^{\frac1r}&(r<\infty),\\
\displaystyle
\sup_{j \in {\mathbb N}_0}
2^{js}
\left\|\sum_{m \in {\mathbb Z}^n}\lambda_{jm}\chi_{Q_{jm}}
\right\|_{{\mathcal M}^\varphi_q}
&(r=\infty)
\end{cases}
\]
is finite.
\item
{\it The $($nonhomogeneous$)$ generalized Triebel-Lizorkin-Morrey sequence space}
${\bf e}_{{\mathcal M}^\varphi_q,r}^s({\mathbb R}^n)$
is the set of all
{
$\lambda=\{\lambda_{jm}\}_{j\in\N_0,m\in\Z^n}$
}
for which the quasi-norm
\[
\|\lambda\|_{{\bf e}_{{\mathcal M}^\varphi_q,r}^s}
\equiv 
\begin{cases}
\displaystyle
\left\|
\left\{\sum_{j=0}^\infty
2^{jsr}
\left(
\sum_{m \in {\mathbb Z}^n}|\lambda_{jm}|\chi_{Q_{jm}}
\right)^r
\right\}^{\frac1r}\right\|_{{\mathcal M}^\varphi_q}&(r<\infty),\\
\displaystyle
\left\|
\sup_{j \in {\mathbb N}_0}
2^{js}
\left(\sum_{m \in {\mathbb Z}^n}|\lambda_{jm}|\chi_{Q_{jm}}\right)
\right\|_{{\mathcal M}^\varphi_q}
&(r=\infty)
\end{cases}
\]
is finite.
\item
The space
${\bf a}^s_{{\mathcal M}^\varphi_q,r}({\mathbb R}^n)$
denotes either
${\bf n}^s_{{\mathcal M}^\varphi_q,r}({\mathbb R}^n)$
or
${\bf e}^s_{{\mathcal M}^\varphi_q,r}({\mathbb R}^n)$.
Assume $(\ref{eq:Nakai-3})$
in the case
when
${\bf a}^s_{{\mathcal M}^\varphi_q,r}({\mathbb R}^n)
={\bf e}^s_{{\mathcal M}^\varphi_q,r}({\mathbb R}^n)$
with $r<\infty$.
\end{enumerate}
\end{definition}

\begin{definition}
Let $L \in {\mathbb N}_0 \cup \{-1\}$ and $K \in {\mathbb N}_0$.
\begin{enumerate}
\item
Let $m \in {\mathbb Z}^n$.
A $C^K$-function $a:{\mathbb R}^n \to {\mathbb C}$
is said to be a $(K,L)$-atom supported near $Q_{0m}$,
if 
\begin{equation}\label{eq:150311-3}
|\partial^\alpha a(x)|\le \chi_{3Q_{0m}}(x)
\end{equation}
for all $\alpha$ with $|\alpha| \le K$.
\item
Let $j=1,2,\ldots$ and $m \in {\mathbb Z}^n$.
A $C^K$-function $a:{\mathbb R}^n \to {\mathbb C}$
is said to be a $(K,L)$-atom supported near $Q_{jm}$,
if 
\begin{equation}\label{eq:140817-3}
2^{-j|\alpha|}|\partial^\alpha a(x)|\le \chi_{3Q_{jm}}(x)
\end{equation}
for all $\alpha$ with $|\alpha| \le K$
and
\begin{equation}\label{eq:140817-4}
\int_{{\mathbb R}^n}x^\beta a(x)\,dx=0
\end{equation}
for all $\beta$ with $|\beta| \le L$ when $L \ge 0$.
\item
Denote by ${\mathfrak A}={\mathfrak A}({\mathbb R}^n)$ 
the set of all collections
$\{a_{jm}\}_{j \in {\mathbb N}_0, m \in {\mathbb Z}^n}$
of $C^K$-functions
such that
each $a_{jm}$ is a $(K,L)$-atom supported near $Q_{jm}$.
\end{enumerate}
\end{definition}

Before we proceed further,
a helpful remark may be in order.
\begin{remark}
The number $3$ does not count in the above definition;
any number $d$ will do as long as $d>1$.
\end{remark}

\begin{theorem}\label{thm:decomposition-1}
Let $0<q<\infty$, $0<r \le \infty$, $s \in {\mathbb R}$
 and $\varphi \in {\mathcal G}_q$.
Let also $L \in {\mathbb N}_0 \cup \{-1\}$ and $K \in {\mathbb N}_0$.
Assume
\begin{equation}\label{eq:L}
K \ge [1+s]_+, \quad L \ge \max(-1,[\sigma_q-s]),
\end{equation}
where
$\sigma_q \equiv n\left(\frac{1}{q}-1\right)_+$.
\begin{enumerate}
\item
Let $f \in {\mathcal N}^s_{{\mathcal M}^\varphi_q,r}({\mathbb R}^n)$.
Then there exist a family
$\{a_{jm}\}_{j \in {\mathbb N}_0, m \in {\mathbb Z}^n}
\in {\mathfrak A}$
and a
{
doubly indexed complex sequence
}
$\lambda=\{\lambda_{jm}\}_{j \in {\mathbb N}_0, m \in {\mathbb Z}^n}
\in {\bf n}^s_{{\mathcal M}^\varphi_q,r}({\mathbb R}^n)$
such that
\begin{equation}\label{eq:140817-101}
f=\sum_{j=0}^\infty 
\left(\sum_{m \in {\mathbb Z}^n}\lambda_{jm}a_{jm}\right)
\mbox{\rm \, in \, }
{\mathcal S}'({\mathbb R}^n)
\end{equation} 
and that
\begin{equation}\label{eq:140817-102}
\|\lambda\|_{{\bf n}^s_{{\mathcal M}^\varphi_q,r}}
\lesssim
\|f\|_{{\mathcal N}^s_{{\mathcal M}^\varphi_q,r}}.
\end{equation}
\item
Let
$\{a_{jm}\}_{j \in {\mathbb N}_0, m \in {\mathbb Z}^n}
\in {\mathfrak A}$
and 
$\lambda=\{\lambda_{jm}\}_{j \in {\mathbb N}_0, m \in {\mathbb Z}^n}
\in {\bf n}^s_{{\mathcal M}^\varphi_q,r}({\mathbb R}^n)$.
Then
$$
f \equiv \sum_{j=0}^\infty 
\left(
\sum_{m \in {\mathbb Z}^n}\lambda_{jm}a_{jm}
\right)
$$
converges in ${\mathcal S}'({\mathbb R}^n)$ 
and belongs to
${\mathcal N}^s_{{\mathcal M}^\varphi_q,r}({\mathbb R}^n)$.
Furthermore,
\[
\|f\|_{{\mathcal N}^s_{{\mathcal M}^\varphi_q,r}}
\lesssim
\|\lambda\|_{{\bf n}^s_{{\mathcal M}^\varphi_q,r}}.
\]

\end{enumerate}
\end{theorem}

\begin{theorem}\label{thm:decomposition-2}
Let $0<q<\infty$, $0<r \le \infty$, $s \in {\mathbb R}$
 and $\varphi:(0,\infty) \to (0,\infty) \in {\mathcal G}_q$.
Let also $L \in {\mathbb N}_0 \cup \{-1\}$ and $K \in {\mathbb N}_0$.
Assume
\[
K \ge [1+s]_+, \quad L \ge \max(-1,[\sigma_{qr}-s]),
\]
where $\sigma_{qr} \equiv \max(\sigma_q,\sigma_r)$.
\begin{enumerate}
\item
Let $f \in {\mathcal E}^s_{{\mathcal M}^\varphi_q,r}({\mathbb R}^n)$.
Then there exist a family
$\{a_{jm}\}_{j \in {\mathbb N}_0, m \in {\mathbb Z}^n}
\in {\mathfrak A}$
and a 
{
doubly indexed complex sequence
}
$\lambda=\{\lambda_{jm}\}_{j \in {\mathbb N}_0, m \in {\mathbb Z}^n}
\in {\bf e}^s_{{\mathcal M}^\varphi_q,r}({\mathbb R}^n)$
satisfying 
$(\ref{eq:140817-101})$ 
and that
\begin{equation}\label{eq:140817-103}
\|\lambda\|_{{\bf e}^s_{{\mathcal M}^\varphi_q,r}}
\lesssim
\|f\|_{{\mathcal E}^s_{{\mathcal M}^\varphi_q,r}}.
\end{equation}
\item
Let
$\{a_{jm}\}_{j \in {\mathbb N}_0, m \in {\mathbb Z}^n}
\in {\mathfrak A}$
and 
$\lambda=\{\lambda_{jm}\}_{j \in {\mathbb N}_0, m \in {\mathbb Z}^n}
\in {\bf e}^s_{{\mathcal M}^\varphi_q,r}({\mathbb R}^n)$.
Then
$$
f \equiv \sum_{j=0}^\infty 
\left(\sum_{m \in {\mathbb Z}^n}\lambda_{jm}a_{jm}\right)
$$
converges in ${\mathcal S}'({\mathbb R}^n)$ 
and belongs to
${\mathcal E}^s_{{\mathcal M}^\varphi_q,r}({\mathbb R}^n)$.
Furthermore,
\begin{equation}\label{eq:150826-1}
\|f\|_{{\mathcal E}^s_{{\mathcal M}^\varphi_q,r}}
\lesssim
\|\lambda\|_{{\bf e}^s_{{\mathcal M}^\varphi_q,r}}.
\end{equation}
\end{enumerate}
\end{theorem}

Theorem \ref{thm:decomposition-1}(1) and Theorem \ref{thm:decomposition-2}(1)
are already obtained in \cite[Theorem 10.15]{LSUYY2}.
So, we concentrate on the proof of
Theorem \ref{thm:decomposition-1}(2) and Theorem \ref{thm:decomposition-2}(2).
{The conditions on $K$ and $L$ are milder.}
To prove them,
we invoke Lemma \ref{lem:Grafakos}.
Its direct corollary is:
\begin{corollary}\label{cor:150821-1}
Let $P>0$ be arbitrary.
Let $K \in {\mathbb N}_0$ and $L \in {\mathbb N}_0 \cup \{-1\}$.
Suppose that we are given an atom $a_{jm}$
supported near $Q_{jm}$.
\begin{enumerate}
\item
Let $j \in {\mathbb N}_0$ and $m \in {\mathbb Z}^n$.
Then
\begin{equation}\label{eq:140817-1}
|\theta(D)a_{jm}(x)| \lesssim 2^{-j(L+1)}M[\chi_{Q_{0m}}](x)^{\frac{N}{n}}.
\end{equation}
In particular, 
\[
\left|\theta(D)\left(\sum_{m \in {\mathbb Z}^n}a_{jm}\right)(x)\right| 
\lesssim 2^{-j(L+1)}
\sum_{m \in {\mathbb Z}^n}M[\chi_{Q_{0m}}](x)^{\frac{N}{n}}.
\]
\item
Let $\nu \in {\mathbb N}$, $j \in {\mathbb N}_0$ and $m \in {\mathbb Z}^n$.
Then
\begin{equation}\label{eq:140817-2}
|\tau_\nu(D)a_{jm}(x)| \lesssim 
\begin{cases}
\displaystyle
2^{-(\nu-j)K}
M[\chi_{Q_{jm}}](x)^{\frac{N}{n}}
&(\nu \ge j),
\\
\displaystyle
2^{(\nu-j)(L+1+n-P)}
M[\chi_{Q_{jm}}](x)^{\frac{N}{n}}
&(\nu \le j).
\end{cases}
\end{equation}
In particular,
by letting
\begin{equation}\label{eq:delta}
\delta\equiv \min(L+1+n-P+s,K-s),
\end{equation}
we have
\[
2^{\nu s}
\left|\tau_\nu(D)\left[
\sum_{m \in {\mathbb Z}^n}\lambda_{jm}a_{jm}
\right](x)\right|
\lesssim
2^{-|\nu-j|\delta}
\sum_{m \in {\mathbb Z}^n}
M[2^{js}\lambda_{jm}\chi_{Q_{jm}}](x)^{\frac{N}{n}}.
\]
\end{enumerate}
\end{corollary}

\begin{proof}
(\ref{eq:140817-1}) is simpler than (\ref{eq:140817-2});
we concentrate on (\ref{eq:140817-2}).
Define
$\Phi^\nu(x)\equiv 2^{\nu n}{\mathcal F}^{-1}\tau(2^\nu x)$
for $x \in {\mathbb R}^n$.
Then we have
\[
\tau_\nu(D)a_{jm}(x)
=
\frac{1}{\sqrt{(2\pi)^n}}
\int_{{\mathbb R}^n}\Phi^\nu(x-y)a_{jm}(y)\,dy.
\]
Let $x \in {\mathbb R}^n$ be fixed with this in mind.

Let $\nu \ge j$.
Then we have
\begin{equation}\label{eq:140820-201}
\int_{{\mathbb R}^n}x^\alpha\Phi^\nu(x)\,dx
=
(2\pi)^{\frac{n}{2}}i^{|\alpha|}
\partial^\alpha[\tau(2^{-\nu}\cdot)](0)=0
\end{equation}
for all {multi-indixes} $\alpha$.
We use $(\ref{eq:140820-201})$ for all
$\alpha$ 
whose length is less than or equal to $K-1$.
We {then} obtain
\[
|\tau_\nu(D)a_{jm}(x)|
=
\frac{1}{2^{jn}\sqrt{(2\pi)^n}}
\left|\int_{{\mathbb R}^n}\Phi^\nu(x-y) \cdot 2^{jn}a_{jm}(y)\,dy
\right|
\lesssim
\frac{2^{-(\nu-j)K}}{(1+2^j|x-2^{-j}m|)^P}
\]
by letting 
$\varphi=2^{jn}a_{jm}$, 
$\psi=\Phi^{\nu}(x-\cdot)$,
$x_\varphi=2^{-j}m$,
$x_\psi=x$
in Lemma \ref{lem:Grafakos}.
In fact,
we can check (\ref{eq:140820-8}) as follows:
\[
|\nabla^\alpha \varphi(x)|
=
2^{jn}|\nabla^\alpha a_{jm}(x)|
\lesssim
2^{j(K+n)}\chi_{3Q_{jm}}(x)
\lesssim
\frac{2^{j(n+K)}}{(1+2^j|x-x_\varphi|)^P}.
\]
By using Lemma \ref{lem:140820-101},
we obtain (\ref{eq:140817-2}) for the case when $\nu \ge j$.

Let $\nu < j$.
Then we have
\[
|\partial^\alpha_y[\Phi^\nu(x-y)]|
\lesssim_\alpha
\frac{2^{(n+|\alpha|)\nu}}{(1+2^\nu|x-y|)^P}
\]
for all {multi-indexes} $\alpha$
as well as (\ref{eq:140817-3}) with $\alpha=0$ and (\ref{eq:140817-4}).
Notice that our assumption $\nu<j$ excludes the case
when $j=0$;
$a_{jm}$ does satisfy (\ref{eq:140817-4}).
Thus,
\[
|\tau_\nu(D)a_{jm}(x)|
\lesssim
\frac{2^{(\nu-j)(L+n+1)}}{(1+2^\nu|x-2^{-j}m|)^P}
\quad (x \in {\mathbb R}^n)
\]
and hence by using $j>\nu$ and Lemma \ref{lem:140820-101} again,
\[
|\tau_\nu(D)a_{jm}(x)|
\lesssim
\frac{2^{(\nu-j)(L+n+1)}}{(2^{\nu-j}+2^\nu|x-2^{-j}m|)^P}
\lesssim
2^{(\nu-j)(L+1+n-P)}
M[\chi_{Q_{jm}}](x)^{\frac{P}{n}}
\quad (x \in {\mathbb R}^n).
\]
{Thus, the proof is complete.}
\end{proof}

{
\begin{remark}
Recall that ${\mathcal E}^s_{pqr}({\mathbb R}^n)$
is a special case of 
{
generalized}
Triebel-Lizorkin-Morrey spaces;
see Section \ref{s7.2}.
The key ingredient is a counterpart 
of Theorems \ref{thm:decomposition-1} and \ref{thm:decomposition-2}
to Morrey spaces.
See Proposition \ref{prop:150312-2} below.
\end{remark}
}

\paragraph{Proof of Theorem \ref{thm:decomposition-1}(2) and Theorem \ref{thm:decomposition-2}(2)}

We prove Theorem \ref{thm:decomposition-2}(2),
the proof of Theorem \ref{thm:decomposition-1}(2) being similar.

We choose real numbers $P$ and $P'$ so that
\begin{equation}\label{eq:150821-701}
\frac{n}{\min(1,q,r)}<P'<P<L+n+1+s.
\end{equation}
Then $\delta$ given by (\ref{eq:delta}) is positive.

Let
$\{a_{jm}\}_{j \in {\mathbb N}_0, m \in {\mathbb Z}^n}
\in {\mathfrak A}$
and 
$\{\lambda_{jm}\}_{j \in {\mathbb N}_0, m \in {\mathbb Z}^n}
\in {\bf e}^s_{{\mathcal M}^\varphi_q,r}({\mathbb R}^n)$.

Let us suppose for the time being that
there exists $N \gg 1$ such that $\lambda_{jm}=0$
if $j \ge N$.
This implies
$\displaystyle
f \equiv \sum_{j=0}^\infty 
\left(
\sum_{m \in {\mathbb Z}^n}\lambda_{jm}a_{jm}
\right)
$
converges in ${\mathcal S}'({\mathbb R}^n)$.

We calculate that
\begin{align*}
\sum_{j=0}^\infty 
\left|\theta(D)\left(
\sum_{m \in {\mathbb Z}^n}\lambda_{jm}a_{jm}
\right)(x)\right|
\lesssim
2^{-j(L+n+1+s-P)}
\sum_{m \in {\mathbb Z}^n}
2^{js}|\lambda_{jm}|M[\chi_{Q{jm}}](x)^{\frac{P}{n}}
\end{align*}
and that
\begin{eqnarray*}
&&\left\{
\sum_{k=1}^\infty
\left(
2^{ks}
\sum_{j=0}^\infty 
\left|\tau_k(D)\left(
\sum_{m \in {\mathbb Z}^n}\lambda_{jm}a_{jm}
\right)(x)\right|
\right)^r\right\}^{\frac1r}\\
&&\lesssim
\left\{
\sum_{k=1}^\infty 
\sum_{j=0}^\infty
\left(
2^{-|j-k|\delta}
\sum_{m \in {\mathbb Z}^n}
2^{js}|\lambda_{jm}|M[\chi_{Q{jm}}](x)^{\frac{P}{n}}
\right)^r\right\}^{\frac1r}\\
&&\lesssim
\left\{
\sum_{j=0}^\infty 
\left(
\sum_{m \in {\mathbb Z}^n}2^{js}|\lambda_{jm}|
M[\chi_{Q{jm}}](x)^{\frac{P}{n}}
\right)^r\right\}^{\frac1r}.
\end{eqnarray*}
If we invoke Lemma \ref{lem:140821-1},
then we have
\begin{eqnarray*}
&&\left\{
\sum_{k=1}^\infty
\left(
2^{ks}
\sum_{j=0}^\infty 
\left|\varphi_k(D)\left(
\sum_{m \in {\mathbb Z}^n}\lambda_{jm}a_{jm}
\right)(x)\right|
\right)^r\right\}^{\frac1r}\\
&&\lesssim
\left\{
\sum_{j=0}^\infty 
\left(
M\left[
\left|\sum_{m \in {\mathbb Z}^n}
2^{js}\lambda_{jm}\chi_{Q_{jm}}\right|^{\frac{n}{P'}}
\right](x)^{\frac{P'}{n}}
\right)^r\right\}^{\frac1r}.
\end{eqnarray*}
Thus,
\begin{align*}
\|f\|_{{\mathcal E}^s_{{\mathcal M}^\varphi_q,r}}
&\lesssim
\left\|
\left\{
\sum_{j=0}^\infty 
\left(
M\left[
\left|\sum_{m \in {\mathbb Z}^n}
2^{js}\lambda_{jm}\chi_{Q_{jm}}\right|^{\frac{n}{P'}}
\right]^{\frac{P'}{n}}
\right)^r\right\}^{\frac1r}
\right\|_{{\mathcal M}^\varphi_q}\\
&\lesssim
\left\|
\left\{
\left(
\sum_{j=0}^\infty 
\sum_{m \in {\mathbb Z}^n}
2^{js}|\lambda_{jm}|\chi_{Q{jm}}
\right)^r\right\}^{\frac1r}
\right\|_{{\mathcal M}^\varphi_q}\\
&=\|\lambda\|_{{\bf e}^s_{{\mathcal M}^\varphi_q,r}}.
\end{align*}
This shows that Theorem \ref{thm:decomposition-2}(2)
is proved for $\lambda$ satisfying that
there exists $N \gg 1$ such that
$\lambda_{jm}=0$ if $j \ge N$.

Let us remove this assumption.
To this end, we set
\[
f_j\equiv\sum_{m \in {\mathbb Z}^n}\lambda_{jm}a_{jm}.
\]
Choose $\rho>0$ so that
\[
L \ge \max(-1,[\sigma_{qr}-s+\rho]),
\]
where $\sigma_{qr} \equiv \max(\sigma_q,\sigma_r)$.
Then according to what we have proved,
we have
\[
\|f_j\|_{{\mathcal E}^{s-\rho}_{{\mathcal M}^\varphi_q,r}}
\lesssim
2^{-\rho j}
\|\lambda\|_{{\bf e}^s_{{\mathcal M}^\varphi_q,r}}.
\]
Therefore,
$f=\sum_{j=1}^\infty f_j$
converges in 
${\mathcal E}^{s-\rho}_{{\mathcal M}^\varphi_q,r}({\mathbb R}^n)$
and hence 
${\mathcal S}'({\mathbb R}^n)$.
Again according to what we have proved,
we also have
\[
\left\|\sum_{j=1}^Nf_j
\right\|_{{\mathcal E}^{s}_{{\mathcal M}^\varphi_q,r}}
\lesssim
\|\lambda\|_{{\bf e}^s_{{\mathcal M}^\varphi_q,r}}
\]
with the constant independent of $N$.
As a result, by letting $N \to \infty$,
the Fatou property of $\mathcal{M}^\varphi_q({\mathbb R}^n)$ yields
$f \in {\mathcal E}^{s}_{{\mathcal M}^\varphi_q,r}({\mathbb R}^n)$
with
\begin{equation}\label{eq:150821-115}
\|f\|_{{\mathcal E}^{s}_{{\mathcal M}^\varphi_q,r}}
\lesssim
\|\lambda\|_{{\bf e}^s_{{\mathcal M}^\varphi_q,r}}.
\end{equation}

\subsection{Molecular decomposition}

In analogy with the atomic decomposition,
we can develop a theory of molecular decomposition as well.
\begin{definition}
Let $L \in {\mathbb N}_0 \cup \{-1\}$ and $K,N \in {\mathbb N}_0$
be such that $N>K+n$.
\begin{enumerate}
\item
Let $j \in {\mathbb N}_0$ and $m \in {\mathbb Z}^n$.
A $C^K$-function ${{ b}}:{\mathbb R}^n \to {\mathbb C}$
is said to be a $(K,L,N)$-molecule supported near $Q_{jm}$,
if
\[
\left|\frac{\partial^\alpha{{ b}}}{\partial x^\alpha}(x)\right|
\le 2^{|\alpha|j}
(1+|2^j x-m|)^{-N}
\] 
with $|\alpha| \le K$
and $(\ref{eq:140817-4})$ with $|\beta| \le L$ and $j \ge 1$
hold.
When $L=-1$, it is understood that 
$(\ref{eq:140817-4})$ is a void condition.
\item
Denote by ${\mathfrak M}={\mathfrak M}({\mathbb R}^n)$ 
the set of all collections
${{ \{b_{jm} \}_{j \in {\mathbb N}_0, m \in {\mathbb Z}^n}  }}$
of $C^K$-functions
such that
each ${{b_{jm}}}$ is a $(K,L,N)$-molecule supported near $Q_{jm}$.
\end{enumerate}
\end{definition}

\begin{theorem}\label{thm:decomposition-101}
Let $0<q<\infty$, $0<r \le \infty$, $s \in {\mathbb R}$
 and $\varphi \in {\mathcal G}_q$.
Let also $L \in {\mathbb N}_0 \cup \{-1\}$ and $K \in {\mathbb N}_0$.
Assume
\begin{equation}\label{eq:10-L}
K \ge [1+s]_+, \quad L \ge \max(-1,[\sigma_q-s]),
\end{equation}
where
$\sigma_q \equiv n\left(\frac{1}{q}-1\right)_+$.
\begin{enumerate}
\item
Let $f \in {\mathcal N}^s_{{\mathcal M}^\varphi_q,r}({\mathbb R}^n)$.
Then there exist a family
$\{{{b_{jm}}}\}_{j \in {\mathbb N}_0, m \in {\mathbb Z}^n}
\in {\mathfrak M}$
and a
{
doubly indexed complex sequence
}
$\lambda=\{\lambda_{jm}\}_{j \in {\mathbb N}_0, m \in {\mathbb Z}^n}
\in {\bf n}^s_{{\mathcal M}^\varphi_q,r}({\mathbb R}^n)$
such that
\begin{equation}\label{eq:150817-101}
f=\sum_{j=0}^\infty 
\left(\sum_{m \in {\mathbb Z}^n}\lambda_{jm}{{b_{jm}}}\right)
\mbox{\rm \, in \, }
{\mathcal S}'({\mathbb R}^n)
\end{equation} 
and that
\begin{equation}\label{eq:150817-102}
\|\lambda\|_{{\bf n}^s_{{\mathcal M}^\varphi_q,r}}
\lesssim
\|f\|_{{\mathcal N}^s_{{\mathcal M}^\varphi_q,r}}.
\end{equation}
\item
Let
$\{{{b_{jm}}}\}_{j \in {\mathbb N}_0, m \in {\mathbb Z}^n}
\in {\mathfrak M}$
and 
$\lambda=\{\lambda_{jm}\}_{j \in {\mathbb N}_0, m \in {\mathbb Z}^n}
\in {\bf n}^s_{{\mathcal M}^\varphi_q,r}({\mathbb R}^n)$.
Then
$$
f \equiv \sum_{j=0}^\infty 
\left(
\sum_{m \in {\mathbb Z}^n}\lambda_{jm}{{b_{jm}}}
\right)
$$
converges in ${\mathcal S}'({\mathbb R}^n)$ 
and belongs to
${\mathcal N}^s_{{\mathcal M}^\varphi_q,r}({\mathbb R}^n)$.
Furthermore,
\[
\|f\|_{{\mathcal N}^s_{{\mathcal M}^\varphi_q,r}}
\lesssim
\|\lambda\|_{{\bf n}^s_{{\mathcal M}^\varphi_q,r}}.
\]
\end{enumerate}
\end{theorem}

\begin{theorem}\label{thm:decomposition-102}
Let $0<q<\infty$, $0<r \le \infty$, $s \in {\mathbb R}$
 and $\varphi:(0,\infty) \to (0,\infty) \in {\mathcal G}_q$.
Let also $L \in {\mathbb N}_0 \cup \{-1\}$ and $K \in {\mathbb N}_0$.
Assume
\[
K \ge [1+s]_+, \quad L \ge \max(-1,[\sigma_{qr}-s]),
\]
where $\sigma_{qr} \equiv \max(\sigma_q,\sigma_r)$.
\begin{enumerate}
\item
Let $f \in {\mathcal E}^s_{{\mathcal M}^\varphi_q,r}({\mathbb R}^n)$.
Then there exist a family
$\{{{b_{jm}}}\}_{j \in {\mathbb N}_0, m \in {\mathbb Z}^n}
\in {\mathfrak M}$
and a 
{
doubly indexed complex sequence
}
$\lambda=\{\lambda_{jm}\}_{j \in {\mathbb N}_0, m \in {\mathbb Z}^n}
\in {\bf e}^s_{{\mathcal M}^\varphi_q,r}({\mathbb R}^n)$
satisfying 
$(\ref{eq:150817-101})$ 
and that
\begin{equation}\label{eq:150817-103}
\|\lambda\|_{{\bf e}^s_{{\mathcal M}^\varphi_q,r}}
\lesssim
\|f\|_{{\mathcal E}^s_{{\mathcal M}^\varphi_q,r}}.
\end{equation}
\item
Let
$\{{{b_{jm}}}\}_{j \in {\mathbb N}_0, m \in {\mathbb Z}^n}
\in {\mathfrak M}$
and 
$\lambda=\{\lambda_{jm}\}_{j \in {\mathbb N}_0, m \in {\mathbb Z}^n}
\in {\bf e}^s_{{\mathcal M}^\varphi_q,r}({\mathbb R}^n)$.
Then
$$
f \equiv \sum_{j=0}^\infty 
\left(\sum_{m \in {\mathbb Z}^n}\lambda_{jm}{{b_{jm}}}\right)
$$
converges in ${\mathcal S}'({\mathbb R}^n)$ 
and belongs to
${\mathcal E}^s_{{\mathcal M}^\varphi_q,r}({\mathbb R}^n)$.
Furthermore,
\[
\|f\|_{{\mathcal E}^s_{{\mathcal M}^\varphi_q,r}}
\lesssim
\|\lambda\|_{{\bf e}^s_{{\mathcal M}^\varphi_q,r}}.
\]
\end{enumerate}
\end{theorem}

We do not prove 
Theorem \ref{thm:decomposition-101}(1) and 
Theorem \ref{thm:decomposition-102}(1);
Theorem \ref{thm:decomposition-1}(1) and 
Theorem \ref{thm:decomposition-2}(1)
are stronger assertions
than
Theorem \ref{thm:decomposition-101}(1) and 
Theorem \ref{thm:decomposition-102}(1),
respectively.
We concentrate on the proof of
Theorem \ref{thm:decomposition-102}(2);
that of 
Theorem \ref{thm:decomposition-101}(2)
is similar.

\paragraph{Proof of Theorem \ref{thm:decomposition-102}(2)}

We modify Corollary \ref{cor:150821-1} as follows:
\begin{corollary}\label{cor:150817-2}
Let $K,N \in {\mathbb N}_0$ and $L \in {\mathbb N}_0 \cup \{-1\}$
with $N>K+n$.
Suppose that we are given an {molecule} ${{b_{jm}}}$
supported near $Q_{jm}$.
\begin{enumerate}
\item
Let $j \in {\mathbb N}_0$ and $m \in {\mathbb Z}^n$.
Then
\begin{equation}\label{eq:140817-10001}
|\theta(D){{b_{jm}}}(x)| \lesssim 2^{-j(L+1)}M[\chi_{Q_{0m}}](x)^{\frac{P}{n}}
\end{equation}
In particular, 
\[
\left|\theta(D)\left(\sum_{m \in {\mathbb Z}^n}{{b_{jm}}}\right)(x)\right| 
\lesssim 2^{-j(L+1)}
\sum_{m \in {\mathbb Z}^n}M[\chi_{Q_{0m}}](x)^{\frac{N}{n}}.
\]
\item
Let $\nu \in {\mathbb N}$, $j \in {\mathbb N}_0$ and $m \in {\mathbb Z}^n$.
Then
\begin{equation}\label{eq:140817-10002}
|\tau_\nu(D){{b_{jm}}}(x)| \lesssim 
\begin{cases}
\displaystyle
2^{-(\nu-j)K}
M[\chi_{Q_{jm}}](x)^{\frac{N}{n}}
&(\nu \ge j),
\\
\displaystyle
2^{(\nu-j)(L+1+n-N)}
M[\chi_{Q_{jm}}](x)^{\frac{N}{n}}
&(\nu \le j).
\end{cases}
\end{equation}
In particular,
by letting
\begin{equation}\label{eq:delta-1}
\delta\equiv \min(L+1+n-N+s,K-s),
\end{equation}
we have
\[
2^{\nu s}
\left|\tau_\nu(D)\left[
\sum_{m \in {\mathbb Z}^n}\lambda_{jm}{{b_{jm}}}
\right](x)\right|
\lesssim
2^{-|\nu-j|\delta}
\sum_{m \in {\mathbb Z}^n}
M[2^{js}\lambda_{jm}\chi_{Q_{jm}}](x)^{\frac{N}{n}}.
\]
\end{enumerate}
\end{corollary}
The proof is the same as that for Corollary 
\ref{cor:150821-1}.

The proof of Theorem \ref{thm:decomposition-102}(2) is a modification
of the corresponding assertions.
Since we are assuming 
$L \ge [\sigma_{qr}-s]$,
we have
$\frac{n}{\min(1,q,r)}<L+n+1+s.$
Let $\tilde{N}$ be a real number slightly less than
$L+n+1+s$.
By considering $\min(N,\tilde{N})$,
we can assume that 
$N<L+n+1+s$.
Choose $P'>0$ so that
\begin{equation}\label{eq:150821-702}
\frac{n}{\min(1,q,r)}<P'<N<L+n+1+s.
\end{equation}
Note that conditions
(\ref{eq:150821-701})
and
(\ref{eq:150821-702})
are the same if we let $P=N$.
Therefore, we can go through the same argument
as we did in Theorem \ref{thm:decomposition-2}(2).

\subsection{Quarkonial decomposition}

By using the atomic decomposition,
we can consider the quarkonial decomposition.
All the results in this section
are new;
the quarkonial decomposition was not obtained
in \cite{LSUYY2}.

\begin{definition}[$\psi$ for the quarkonial decomposition]
Throughout this section,
the function
$\psi \in {\mathcal S}$
is fixed so that
$\{\psi(\cdot-m)\}_{m \in {\mathbb Z}^n}$
forms a {\it partition of unity}{\rm;}
\begin{equation}\label{eq:partition of unity}
\sum_{m \in {\mathbb Z}^n}\psi(\cdot-m) \equiv 1.
\end{equation}
Accordingly,
choose $R>0$ so that
\begin{equation}\label{quark-1}
{\rm supp}(\psi) \subset Q(2^R).
\end{equation}
\end{definition}

With $\psi$ specified as above,
we define quarks.
\begin{definition}[Regular quark]
Let
$\beta \in {\mathbb N}_0{}^n, \ \nu \in {\mathbb N}_0$
and 
$m \in {\mathbb Z}^n$. 
Then define
a function $\psi^\beta$
and
the quark 
$(\beta qu)_{\nu m}=(\beta qu)_{\nu,m}$
by:
\begin{eqnarray}\label{eq:150311-2}
\psi^\beta(x)
\equiv x^\beta \psi(x), \,
(\beta qu)_{\nu m}(x)
\equiv 
\psi^\beta(2^\nu x-m)
=
(2^\nu x-m)^\beta\psi(2^\nu x-m)
\end{eqnarray}
for $x \in {\mathbb R}^n$.
{Each}
$(\beta qu)_{\nu m}$
is called the {\it quark}.
\end{definition}

\begin{remark}
\label{remark:8-26-1} 
As it is mentioned in \cite[Discussion 2.5, p12]{Triebel2}, 
there exists $d>0$ such that 
\begin{equation}
{\rm supp\,}(\beta qu)_{\nu m}\subset dQ_{\nu m}
\end{equation}
for any $\nu\in{\mathbb N}_0$ and $m\in{\mathbb Z}^n$. 

By $(\ref{quark-1})$, we have $|\psi^{\beta}(x)|\le 2^{R|\beta|}$ 
for any {$\beta\in{\mathbb N}_0^{\ n}$}. 
Therefore we obtain 
\[
|(\beta qu)_{\nu m}|\le 2^{R|\beta|} 
\ \ \text{with} \ \ 
{\beta\in{\mathbb N}_0^{\ n}.} 
\]

Fix any {$\alpha\in{\mathbb N}_0^{\ n}$} with $|\alpha|\le K$
and $K \ge 0$.
Then, by the definition of quarks, 
we see that 
\[
|\partial^{\alpha}\psi_{\beta}(x)|=|\partial^{\alpha}[x^{\beta}\psi(x)]|
\le \sum_{|\alpha'|\le K} |\partial^{\alpha'}[x^{\beta}\psi(x)]|
\le c_1 (1+|\beta|)^K2^{R|\beta|}
\le c_2 2^{(R+\epsilon)|\beta|}
\]
for any $\epsilon>0$, 
where $c_1$ and $c_2$ are constants independent of $\beta$ 
but {depend} on $\psi$, $K$ and $\epsilon$. 
This and the chain rule of differentiation imply that 
\begin{equation}\label{eq:141021-1}
{
|\partial^{\alpha}(\beta qu)_{\nu m}(x)|
}
\lesssim 2^{\nu|\alpha| +(R+\epsilon)|\beta|}
\end{equation} 
holds for any {$\beta\in{\mathbb N}_0^{\ n}$}, 
$\nu\in{\mathbb N}_0$ and $m\in{\mathbb Z}^n$. 
\end{remark}

\begin{lemma}
Let ${\{\Lambda_{\nu m}\}_{\nu \in {\mathbb N}_0, m \in {\mathbb Z}^n}}$
be a bounded sequence
{
and let $\kappa\in\mathcal{S}$ satisfy 
$\chi_{Q(3)}\le \kappa\le \chi_{Q(3+1/100)}$.
}
Then
\[
\sum_{m \in {\mathbb Z}^n}
\sum_{l \in {\mathbb Z}^n}
\sum_{\beta \in {\mathbb N}_0{}^n}
\frac{2^{-\rho|\beta|}}{\beta!}
\Lambda_{\nu m}
\partial^\beta{\mathcal F}^{-1}\kappa(2^{-\rho}l-m)
\psi^\beta(2^{\nu+\rho}\cdot-m)
\]
is convergent in the weak star topology of $L^\infty({\mathbb R}^n)$.
More precisely,
by writing 
$$\displaystyle
M\equiv
{\sup_{\nu \in {\mathbb N}_0, m \in {\mathbb Z}^n}}|\Lambda_{\nu m}|,
$$
we have
\[
\int_{{\mathbb R}^n}
\sum_{m \in {\mathbb Z}^n}
\sum_{l \in {\mathbb Z}^n}
\sum_{\beta \in {\mathbb N}_0{}^n}
\frac{2^{-\rho|\beta|}}{\beta!}
|\Lambda_{\nu m}
\partial^\beta{\mathcal F}^{-1}\kappa(2^{-\rho}l-m)
\psi^\beta(2^{\nu+\rho}x-m)f(x)|\,dx \lesssim 
M\|f\|_{L^1}
\]
for $f \in L^1({\mathbb R}^n)$.
\end{lemma}

\begin{proof}
It suffices to prove that
\begin{equation}\label{eq:141029-1}
\sum_{m \in {\mathbb Z}^n}
\sum_{l \in {\mathbb Z}^n}
\sum_{\beta \in {\mathbb N}_0{}^n}
\frac{2^{-\rho|\beta|}}{\beta!}
|\Lambda_{\nu m}
\partial^\beta{\mathcal F}^{-1}\kappa(2^{-\rho}l-m)
\psi^\beta(2^{\nu+\rho}\cdot-m)|
\in L^\infty({\mathbb R}^n)
\end{equation}
thanks to the Lebesgue convergence theorem.

Note that
\[
{
|\partial^\beta{\mathcal F}^{-1}\kappa(2^{-\rho}l-m)|
}
\lesssim \langle \beta \rangle^{2N}\langle 2^{-\rho}l-m \rangle^{-2N}
\mbox{ and that }
|\psi^\beta(2^{\nu+\rho}\cdot-m)| \lesssim
2^{(R+\varepsilon)|\beta|}.
\]
For each $x \in {\mathbb R}^n$,
let us fix $m_{x,1},\ldots,m_{x,N(R)}$ such that
$\psi^\beta(2^{\nu+\rho}x-m) \ne 0$,
where $N(R)$ is a geometric constant depending only on $R$.
Then $|m_{x,j}-2^{\nu+\rho}x| \le 2^R$
for each $j=1,2,\ldots,N(R)$.

We calculate that:
\begin{eqnarray*}
&&\sum_{m \in {\mathbb Z}^n}
\sum_{l \in {\mathbb Z}^n}
\sum_{\beta \in {\mathbb N}_0{}^n}
\frac{2^{-\rho|\beta|}}{\beta!}
|\Lambda_{\nu m}
\partial^\beta{\mathcal F}^{-1}\kappa(2^{-\rho}l-m)
\psi^\beta(2^{\nu+\rho}x-m)|\\
&&\lesssim
M
\sum_{j=1}^{N(R)}
\sum_{l \in {\mathbb Z}^n}
\sum_{\beta \in {\mathbb N}_0{}^n}
\frac{2^{(R+\varepsilon-\rho)|\beta|}}{\beta!}\langle \beta \rangle^{2N}\langle 2^{-\rho}l-m_{x,j} \rangle^{-2N}\\
&&\lesssim
M
\sum_{l \in {\mathbb Z}^n}
\sum_{\beta \in {\mathbb N}_0{}^n}
\frac{2^{(R+\varepsilon-\rho)|\beta|}}{\beta!}\langle \beta \rangle^{2N}\langle 2^{-\rho}l-2^{\nu+\rho}x \rangle^{-2N}\\
&&\lesssim
M
\sum_{\beta \in {\mathbb N}_0{}^n}
\frac{2^{(R+\varepsilon-\rho)|\beta|}}{\beta!}
\langle \beta \rangle^{2N}
\int_{{\mathbb R}^n}\langle z-2^{\nu+\rho}x \rangle^{-2N}\,dz
\sim
M
\sum_{\beta \in {\mathbb N}_0{}^n}
\frac{2^{(R+\varepsilon-\rho)|\beta|}}{\beta!}
\langle \beta \rangle^{2N}
\sim M.
\end{eqnarray*}
Thus, (\ref{eq:141029-1}) is obtained.
\end{proof}

\begin{definition}[Sequence spaces for quarkonial decomposition]
Let {$R,\rho$ satisfy} (\ref{quark-1}) and $\rho>R$.
For a {triply indexed complex sequence}
$\lambda
=\{\lambda^\beta_{\nu m}\}_{\beta \in {\mathbb N}_0{}^n, \ 
\nu \in {\mathbb N}_0, \ m \in {\mathbb Z}^n},
$
define
\begin{equation}
\lambda^\beta
{\equiv}
\{\lambda^\beta_{\nu m}
\}_{\nu \in {\mathbb N}_0, \ m \in {\mathbb Z}^n}, \quad 
\| \lambda \|_{{\bf a}_{{\mathcal M}^s_q,r},\rho}
\equiv 
\sup_{\beta \in {\mathbb N}_0{}^n}
2^{\rho|\beta|}\| \lambda^\beta \|_{{\bf a}^s_{{\mathcal M}^\varphi_q,r}}.
\end{equation} 
\end{definition}

\begin{theorem}\label{thm:quark-1}
Let $0<q<\infty$, $0<r \le \infty$, $s \in {\mathbb R}$ 
and $\varphi \in {\mathcal G}_q$.
Assume $\rho{\equiv}[R+1]>R$,
where $R$ is a constant in $(\ref{quark-1})$.
\begin{enumerate}
\item
Let $s>\sigma_q$ and 
$f \in {\mathcal N}^s_{{\mathcal M}^\varphi_q,r}({\mathbb R}^n)$.
Then there exists a
{
triply indexed complex sequence
}
\[
\lambda=\{\lambda^\beta_{\nu m}\}_{\nu \in {\mathbb N}_0,
m \in {\mathbb Z}^n, \beta \in {\mathbb N}_0{}^n}
\]
such that
\begin{equation}\label{eq:150206-1}
f=
\sum_{\beta \in {\mathbb N}_0{}^n}
\sum_{\nu=0}^\infty
{\sum_{m \in {\mathbb Z}^n}}
\lambda^\beta_{\nu m}
(\beta qu)_{\nu m}
\end{equation}
converges in ${\mathcal S}'({\mathbb R}^n)$
and
\begin{equation}\label{eq:150206-2}
\|\lambda\|_{{\bf n}_{{\mathcal M}^s_q,r},\rho}
\lesssim
\|f\|_{{\mathcal N}^s_{{\mathcal M}^\varphi_q,r}}.
\end{equation}
The constant $\lambda^\beta_{\nu m}$ depends
continuously and linearly on $f$.
\item
If $s>\sigma_q$ and 
$\lambda=\{\lambda^\beta_{\nu m}\}_{\nu \in {\mathbb N}_0,
m \in {\mathbb Z}^n, \beta \in {\mathbb N}_0{}^n}$
satisfies
$\|\lambda\|_{{\bf n}_{{\mathcal M}^s_q,r},\rho}
<\infty,$
then
\begin{equation}\label{eq:150313-3}
f\equiv
\sum_{\beta \in {\mathbb N}_0{}^n}
\sum_{\nu=0}^\infty
{\sum_{m \in {\mathbb Z}^n}}
\lambda^\beta_{\nu m}
(\beta qu)_{\nu m}
\end{equation}
converges in ${\mathcal S}'({\mathbb R}^n)$
and belongs to
${\mathcal N}^s_{{\mathcal M}^\varphi_q,r}({\mathbb R}^n)$.
Furthermore,
\begin{equation}\label{eq:150313-4}
\|f\|_{{\mathcal N}^s_{{\mathcal M}^\varphi_q,r}}
\lesssim
\|\lambda\|_{{\bf n}_{{\mathcal M}^s_q,r},\rho}.
\end{equation}
\item
If $s>\sigma_{qr}$ and 
$f \in {\mathcal E}^s_{{\mathcal M}^\varphi_q,r}({\mathbb R}^n)$,
then there exists a 
{
triply indexed complex sequence
}
$\lambda=\{\lambda^\beta_{\nu m}\}_{\nu \in {\mathbb N}_0,
m \in {\mathbb Z}^n, \beta \in {\mathbb N}_0{}^n}$
such that
\begin{equation}\label{eq:150313-1}
f=
\sum_{\beta \in {\mathbb N}_0{}^n}
\sum_{\nu=0}^\infty
{\sum_{m \in {\mathbb Z}^n}}
\lambda^\beta_{\nu m}
(\beta qu)_{\nu m}
\end{equation}
in ${\mathcal S}'({\mathbb R}^n)$
and
\begin{equation}\label{eq:150313-2}
\|\lambda\|_{{\bf e}_{{\mathcal M}^s_q,r},\rho}
\lesssim
\|f\|_{{\mathcal E}^s_{{\mathcal M}^\varphi_q,r}}.
\end{equation}
The constant $\lambda^\beta_{\nu m}$ depends
continuously and linearly on $f$.
\item
If $s>\sigma_{qr}$ and 
$\lambda=\{\lambda^\beta_{\nu m}\}_{\nu \in {\mathbb N}_0,
m \in {\mathbb Z}^n, \beta \in {\mathbb N}_0{}^n}$
satisfies
$\|\lambda\|_{{\bf e}_{{\mathcal M}^s_q,r},\rho}
<\infty,$
{then 
\[
f\equiv
\sum_{\beta \in {\mathbb N}_0{}^n}
\sum_{\nu=0}^\infty
{\sum_{m \in {\mathbb Z}^n}}
\lambda^\beta_{\nu m}
(\beta qu)_{\nu m}
\]
converges in ${\mathcal S}'({\mathbb R}^n)$
}
and belongs to
${\mathcal E}^s_{{\mathcal M}^\varphi_q,r}({\mathbb R}^n)$.
Furthermore,
\[
\|f\|_{{\mathcal E}^s_{{\mathcal M}^\varphi_q,r}}
\lesssim
\|\lambda\|_{{\bf e}_{{\mathcal M}^s_q,r},\rho}.
\]
\end{enumerate}
\end{theorem}

To prove Theorem \ref{thm:quark-1}, we need several Lemmas. 

From now on we assume 
that $\theta$ and $\tau$ both belong to ${\mathcal S}({\mathbb R}^n)$
that 
\[
\chi_{Q(2)}\le \theta \le \chi_{Q(3)},
\] 
and that
\[
\tau_{j}=\theta(2^{-j}\cdot)-\theta(2^{-j+1}\cdot)
\]
for $j\in{\mathbb N}$. 

\begin{lemma}
\label{lemma:8-22-1}
Let $\kappa\in{\mathcal S}({\mathbb R}^n)$ satisfy 
$\chi_{Q(3)}\le \kappa\le \chi_{Q(3+1/100)}$. 
{Let $f \in {\mathcal S}'({\mathbb R}^n)$.}
\begin{enumerate}
\item
{\rm \cite[Theorem 5.1.22]{Sawano-book}}
Whenever ${\rm supp}({\mathcal F}f) \subset Q(3 \cdot 2^{\nu})$, 
{$f$} can be written as:
\begin{align}
\label{eq:150821-122}
f=\frac{1}{\sqrt{(2\pi)^n}}\sum_{m\in{\mathbb Z}^n}
f(2^{-\nu}m){\mathcal F}^{-1}\kappa(2^{\nu}\cdot -m). 
\end{align}
\item
{\rm \cite[Corollary 5.1.23]{Sawano-book}}
{Generally,}
\begin{align}
\nonumber
f&=
\frac{1}{\sqrt{(2\pi)^n}}
\sum_{m \in {\mathbb Z}^n}\tau(D)f(m){\mathcal F}^{-1}\kappa(\cdot-m)\\
\label{eq:150821-121}
&\quad+
\frac{1}{\sqrt{(2\pi)^n}}
\sum_{\nu=1}^\infty
\left(
\sum_{m \in {\mathbb Z}^n}
\varphi_\nu(D)f(2^{-\nu}m){\mathcal F}^{-1}\kappa(2^\nu \cdot-m)
\right).
\end{align}
\end{enumerate}
\end{lemma}

\begin{lemma}\label{lemma:8-22-2}
Let $l\in{\mathbb Z}^n$, $0<\eta<\min(1,q,r)$. 
Then 
\begin{equation}\label{eq:8-22-3}
\|\lambda^{l}\|_{{\bf a}_{{\mathcal M}^s_q,r}}
\lesssim 
\langle l \rangle^{n/\eta} 
\| \lambda\|_{{\bf a}_{{\mathcal M}^s_q,r}}.
\end{equation}
\end{lemma}

\begin{proof}
We can prove this Lemma 
by using the same argument of the proof of \cite[Lemma 5.1.24]{Sawano-book}. 
So we omit the proof. 
\end{proof}

\begin{lemma}{\cite[Lemma 5.1.25]{Sawano-book}}\label{lemma:8-22-3}
Let $\kappa\in{\mathcal S}({\mathbb R}^n)$ satisfy 
$\chi_{Q(3)}<\kappa<\chi_{Q(3+1/100)}$. 
Then 
\begin{equation}
|\partial^{\alpha}{\mathcal F}^{-1}\kappa(y)| 
\lesssim_N
\langle \alpha \rangle^{2N}\langle y\rangle^{-2N} 
\quad (\alpha \in {\mathbb N}_0{}^n, \quad y \in {\mathbb R}^n)
\label{eq:8-22-4}
\end{equation}
hold for any $N\gg 1$. 
\end{lemma}

\paragraph{Proof of Theorem \ref{thm:quark-1}}

Firstly, we prove ${\rm (2)}$ and ${\rm (4)}$. 
We let
\[
\eta_0 \equiv \min(q,r,1).
\]
Assuming $\rho>R$, 
we can take $\epsilon>0$ such that $0<\epsilon<\rho-R$. 
By the assumption $s>\sigma_q$ and $s>\sigma_{qr}$
in ${\rm (2)}$ and ${\rm (4)}$ respectively, 
the atoms in ${\mathcal A}_{{\mathcal M}^\varphi_q,r}^s({\mathbb R}^n)$ 
are not required to satisfy any moment conditions. 
This and Remark $\ref{remark:8-26-1}$ imply that we can regard 
$2^{-(R+\epsilon)|\beta|}(\beta qu)_{\nu m}$ 
as a $(K,-1)$-atom supported near $Q_{\nu m}$
$($modulo a multiplicative constant 
independent of $\lambda$ and 
$\{a_{\nu m}\}_{\nu \in {\mathbb N}_0, m \in {\mathbb Z}^n}$
$)$. 

We define
\begin{equation}
f^{\beta} 
\equiv 
\sum_{\nu\in{\mathbb N}_0}\sum_{m\in{\mathbb Z}^n}
\lambda_{\nu m}^{\beta}(\beta qu)_{\nu m} \label{eq:8-22-1}
\end{equation}
for each {$\beta \in {\mathbb N}_0^{\ n}$.}

By using Theorems \ref{thm:decomposition-1} and \ref{thm:decomposition-2}, 
we have 
\[
\| f^{\beta}\|_{{\mathcal A}^s_{{\mathcal M}^\varphi_q,r}} 
\lesssim 
\| 2^{(R+\epsilon)|\beta|}\lambda^{\beta} 
\|_{{\bf a}^s_{{\mathcal M}^\varphi_q,r}}
=2^{-(\rho-R-\epsilon)|\beta|}\|2^{\rho|\beta|} \lambda^{\beta} 
\|_{{\bf a}^s_{{\mathcal M}^\varphi_q,r}} 
\lesssim 2^{-(\rho-R-\epsilon)|\beta|}\| \lambda \|_{{\bf a}^s_{{\mathcal M}^\varphi_q,r},\rho}.
\]
Therefore, by the $\eta_0$-triangle inequality
(see Lemma \ref{lem:141020-1}), we see that 
\begin{align*}
(\| f\|_{{\mathcal A}^s_{{\mathcal M}^\varphi_q,r}})^{\eta_0} 
= 
\left(
\left\| 
{\sum_{\beta\in{\mathbb N}_0^{\ n}}f^{\beta}
}
\right\|_{{\mathcal A}^s_{{\mathcal M}^\varphi_q,r}}\right)^{\eta_0} 
\le 
{
\sum_{\beta\in{\mathbb N}_0^{\ n}}
}
(\| f^{\beta}\|_{{\mathcal A}^s_{{\mathcal M}^\varphi_q,r}})^{\eta_0} 
\lesssim 
(\| \lambda \|_{{\bf a}^s_{{\mathcal M}^\varphi_q,r},\rho})^{\eta_0}. 
\end{align*}
This implies that $(2)$ and $(4)$ hold. 

Finally, we prove $(1)$ and $(3)$. 
Let $f\in {\mathcal A}^s_{{\mathcal M}^\varphi_q,r}({\mathbb R}^n)$. 
By Lemma \ref{lemma:8-22-1} we have 
\begin{align}
f&=\frac{1}{\sqrt{(2\pi)^n}}\sum_{m\in{\mathbb Z}^n}
\theta(D)f(m){\mathcal F}^{-1}\kappa(\cdot -m) \notag \\ 
&\quad +
\frac{1}{\sqrt{(2\pi)^n}}
\sum_{\nu\in{\mathbb N}}\left( \sum_{m\in{\mathbb Z}^n} \tau_{\nu}(D)f\left( 2^{-\nu}m\right){\mathcal F}^{-1}\kappa (2^{\nu} \cdot -m)\right). \label{8-22-5}
\end{align} 
We put 
\begin{equation}
\Lambda_{\nu m}
\equiv
\begin{cases} 
\theta(D)f(m) \ \ &\text{if $\nu=0$,} \\ 
\tau_{\nu}(D)f\left( 2^{-\nu}m\right) \ \ &\text{if $\nu\neq 0$.}
\end{cases}
\label{8-22-6}
\end{equation}
Then we can rewrite $(\ref{8-22-5})$ as 
\begin{equation}\label{8-22-7}
f\sim_n \sum_{\nu\in{\mathbb N}_0}\sum_{m\in{\mathbb Z}^n}
\Lambda_{\nu m}{\mathcal F}^{-1}\kappa (2^{\nu} \cdot -m), 
\end{equation}
where the symbol $A\sim_n B$ denotes that there exists a constant
$c_n \ne 0$ such that $A=c_n B$.
Since we may consider that $\rho$ is a big integer, 
we can use the Taylor expansion 
to ${\mathcal F}^{-1}\kappa(2^{\nu}\cdot-m)$
at $x=2^{-\nu-\rho}l$. 
Therefore we see that 
\begin{align*}
\psi(2^{\nu+\rho}x-l){\mathcal F}^{-1}\kappa(2^{\nu}x-m) 
&= 
{
\sum_{\beta\in{\mathbb N}_0^{\ n}}
}
\frac{\partial^{\beta} {\mathcal F}^{-1}\kappa(2^{-\rho}l-m)(2^{\nu}x-2^{-\rho}l)^{\beta}\psi(2^{\nu+\rho}x-l)}{\beta!} \\ 
&= 
{
\sum_{\beta\in{\mathbb N}_0^{\ n}}
}
\frac{2^{-\rho|\beta|}\partial^{\beta} {\mathcal F}^{-1}\kappa(2^{-\rho}l-m)(\beta qu)_{\nu+\rho,l}(x)}{\beta!}. 
\end{align*}
Furthermore we have 
\begin{align}
\tau_\nu(D)f 
&= \frac{1}{\sqrt{(2\pi)^n}}
\sum_{m\in{\mathbb Z}^n} 
\tau_{\nu}(D)f\left( 2^{-\nu}m\right)
{\mathcal F}^{-1}\kappa (2^{\nu} \cdot -m) \notag \\ 
&\sim_n \sum_{m\in{\mathbb Z}^n}\sum_{l\in{\mathbb Z}^n}
\sum_{\beta\in{\mathbb N}_0^{\ n}}
\frac{2^{-\rho|\beta|}}{\beta!}\Lambda_{\nu m}\partial^{\beta} 
{\mathcal F}^{-1}\kappa(2^{-\rho}l-m)(\beta qu)_{\nu+\rho,l}
\label{eq:8-22-8}
\end{align}
by $(\ref{eq:partition of unity})$. 
Since the convergence of $(\ref{eq:8-22-8})$ takes place
also in the weak-* topology
 of $L^{\infty}({\mathbb R}^n)$, 
we can change the order of {summation} in $(\ref{eq:8-22-8})$ as follows: 
\begin{align}
\tau_\nu(D)f 
&\sim_n \sum_{l\in{\mathbb Z}^n}
\sum_{\beta\in{\mathbb N}_0^{\ n}}
\sum_{m\in{\mathbb Z}^n}
\frac{2^{-\rho|\beta|}}{\beta!}\Lambda_{\nu m}\partial^{\beta} 
{\mathcal F}^{-1}\kappa(2^{-\rho}l-m)(\beta qu)_{\nu+\rho,l} \notag \\ 
&\sim_n \sum_{l\in{\mathbb Z}^n}
\sum_{\beta\in{\mathbb N}_0^{\ n}}
\lambda_{\nu+\rho,l}^{\beta}(\beta qu)_{\nu+\rho,l}, \label{eq:8-22-9}
\end{align}
where $\displaystyle \lambda_{\nu+\rho,l}^{\beta}
\equiv
\frac{2^{-\rho|\beta|}}{\beta!}\sum_{m\in{\mathbb Z}^n}
\Lambda_{\nu m}\partial^{\beta} {\mathcal F}^{-1}\kappa(2^{-\rho}l-m)$. 

Let $l_0$ be a lattice point 
in $[0,2^{\rho})^n$ and $x\in Q_{\nu+\rho,2^{\rho}l+l_0}$. 
Then we obtain 
\[
| \lambda_{\nu+\rho,2^{\rho}l+l_0}^{\beta}|
\lesssim 
2^{-\rho|\beta|}\sum_{m\in{\mathbb Z}^n}\langle l-m\rangle^{-N}|\Lambda_{\nu m}|
=
2^{-\rho |\beta|}\sum_{m\in{\mathbb Z}^n}\langle m\rangle^{-N}
|\Lambda_{\nu,m+l}|
\]
by Lemma $\ref{lemma:8-22-3}$. 
Put
\begin{equation}
\label{eq:8-23-1}
\Lambda^m
\equiv \{|\Lambda_{\nu,m+l}|\}_{\nu\in{\mathbb N}_0, l\in{\mathbb Z}^n}
\quad (m\in{\mathbb Z}^n).
\end{equation}
Then we have 
\begin{align*}
\| \lambda^{\beta}\|_{{\bf a}_{{\mathcal M}^s_q,r}} 
\lesssim 2^{-\rho|\beta|} 
\left\| \sum_{m\in{\mathbb Z}^n}\langle m\rangle^{-N}|\Lambda^m| 
\right\|_{{\bf a}^s_{{\mathcal M}^\varphi_q,r}} 
\lesssim 2^{-\rho|\beta|} \left( 
\sum_{m\in{\mathbb Z}^n}\langle m\rangle^{-N\eta_0}
(\| \Lambda^m \|_{{\bf a}^s_{{\mathcal M}^\varphi_q,r}})^{\eta_0} 
\right)^{1/\eta_0}. 
\end{align*}
Since we can take $N$ {sufficiently} large, 
by Lemma $\ref{lemma:8-22-2}$
with $\eta=\eta_0/2$, 
we see that 
\begin{equation}
\| \lambda^{\beta}\|_{{\bf a}_{{\mathcal M}^s_q,r}} 
\lesssim 2^{-\rho|\beta|} 
\left( \sum_{m\in{\mathbb Z}^n}
\langle m\rangle^{(2n/\eta_0-N)\eta_0}
(\| \Lambda \|_{{\bf a}^s_{{\mathcal M}^\varphi_q,r}})^{\eta_0} 
\right)^{1/\eta_0}
\sim 2^{-\rho|\beta|} 
\| \Lambda \|_{{\bf a}^s_{{\mathcal M}^\varphi_q,r}}. \label{eq:8-23-2}
\end{equation}
This implies 
$\displaystyle 
\| \lambda\|_{{\bf a}_{{\mathcal M}^s_q,r},\rho} 
\lesssim 
\| \Lambda \|_{{\bf a}^s_{{\mathcal M}^\varphi_q,r}}$. 

For any $y\in Q_{\nu, m}$, by the Plancherel-Polya-Nikol'skii inequality 
(Theorem $\ref{thm:PPN}$), we have 
\[
\frac{1}{\left( 1+ 2^{\nu}|y-2^{-\nu}m| \right)^{\frac{2n}{\eta_0}} } 
\left| 
\tau_{\nu}(D)f\left( 2^{-\nu}m \right) 
\right|
\lesssim M^{(\eta_0/2)}[\tau_{\nu}(D)f](y). 
\]
Hence we see that 
\begin{align*}
|\Lambda_{\nu, m}| 
&= \left| 
\tau_{\nu}(D)f\left( 2^{-\nu}m \right) 
\right| \notag \\ 
&= \left( 1+ 2^{\nu}|y-2^{-\nu}m| \right)^{\frac{2n}{\eta_0}} \cdot 
\frac{1}{\left( 1+ 2^{\nu}|y-2^{-\nu}m| \right)^{\frac{2n}{\eta_0}} } 
\left| 
\tau_{\nu}(D)f\left( 2^{-\nu}m \right) 
\right| \notag \\ 
&\lesssim 
(1+\sqrt{n})^{\frac{2n}{\eta_0} } 
\frac{1}{\left( 1+ 2^{\nu}|y-2^{-\nu}m| \right)^{\frac{2n}{\eta_0}} } 
\left| 
\tau_{\nu}(D)f\left( 2^{-\nu}m \right) 
\right| \notag \\ 
&\lesssim M^{(\eta_0/2)}[\tau_{\nu}(D)f](y). 
\end{align*}
Since we have 
\[
\left| \Lambda_{\nu,m} 
\right| 
\lesssim 
\inf_{y\in Q_{\nu,m}}\left(M^{(\eta_0/2)}[\tau_{\nu}(D)f](y) \right), 
\]
we obtain 
$||\Lambda||_{{\bf a}^s_{{\mathcal M}^\varphi_q,r}} 
\lesssim ||f||_{{\mathcal A}^s_{{\mathcal M}^\varphi_q,r}}$. 
This proves the necessity of the quarkonial decomposition.

\section{Fundamental theorems}
\label{s5}

\subsection{Trace operator}

In this section,
we aim to extend the trace operator,
which is initially defined on ${\mathcal S}({\mathbb R}^n)$
by:
\begin{equation}\label{eq:150821-82}
f \in {\mathcal S}({\mathbb R}^n)
\mapsto
f(\cdot',0_n) \in {\mathcal S}({\mathbb R}^{n-1}).
\end{equation}
Our main result is as follows:
\begin{theorem}\label{thm:trace}
{Let $n \ge 2$.}
Suppose that we are given parameters
$q,r,s$ and a function $\varphi \in {\mathcal G}_q$.
Define $s^*$ and $\varphi^*$ by
\begin{equation}\label{eq:s-star}
s^*{\equiv}s-\frac1q
\end{equation}
and 
\begin{equation}\label{eq:varphi-star}
\varphi^*(t){\equiv}\varphi(t)t^{-1/q}
\quad (t>0).
\end{equation}
Assume in addition that $\varphi^*$ is increasing and satisfies
\begin{equation}\label{eq:150828-3}
\sum_{j=0}^{\infty}
\frac{1}{\varphi^*(2^j s)}
\lesssim 
\frac{1}{\varphi^*(s)} \quad (0<s \le 1).
\end{equation}
\begin{enumerate}
\item
Let
\begin{equation}\label{eq:150311-9}
s>\frac{1}{q}+(n-1)\left(\frac{1}{\min(1,q)}-1\right).
\end{equation}
Then
we can extend
the trace operator $f \mapsto f(\cdot',0_n)$
to a bounded linear operator from
${\mathcal N}^s_{{\mathcal M}^\varphi_q,r}({\mathbb R}^n)$
to
${\mathcal N}^{s^*}_{{\mathcal M}^{\varphi^*}_q,r}({\mathbb R}^{n-1})$.
\item
Let
\begin{equation}\label{eq:140820-101}
s>\frac{1}{q}+(n-1)\left(\frac{1}{\min(1,q,r)}-1\right).
\end{equation}
Then
we can extend
the trace operator $f \mapsto f(\cdot',0_n)$
to a bounded linear operator from
${\mathcal E}^s_{{\mathcal M}^\varphi_q,r}({\mathbb R}^n)$
to
${\mathcal E}^{s^*}_{{\mathcal M}^{\varphi^*}_q,q}({\mathbb R}^{n-1})$.
\end{enumerate}
\end{theorem}

\begin{proof}
We shall prove (2).
The proof of (1) is even simpler.
Let $Q' \in {\mathcal D}({\mathbb R}^{n-1})$.
We have only to prove
\begin{equation}\label{eq:140817-301}
\varphi^*(\ell(Q'))
\left(\frac{1}{|Q'|}\int_{Q'}
\sum_{j=j_{Q'}}^\infty
\left|2^{js^*}
\sum_{m' \in {\mathbb Z}^{n-1}}\lambda_{j(m',0)}\chi_{Q_{j m'}}(x')
\right|^q
\,dx'
\right)^{\frac{1}{q}}
\lesssim
\|\lambda\|_{{\bf e}^s_{{\mathcal M}^\varphi_q,r}}
\end{equation}
and
\begin{equation}\label{eq:140817-302}
\varphi^*(\ell(Q'))
\left(\frac{1}{|Q'|}\int_{Q'}
\sum_{j=0}^{j_{Q'}}
\left|2^{js^*}
\sum_{m' \in {\mathbb Z}^{n-1}}\lambda_{j(m',0)}\chi_{Q_{j m'}}(x')
\right|^q
\,dx'
\right)^{\frac{1}{q}}
\lesssim
\|\lambda\|_{{\bf e}^s_{{\mathcal M}^\varphi_q,r}}
\end{equation}
for all 
$\lambda=\{\lambda_{jm}\}_{j \in {\mathbb N}_0, m \in {\mathbb Z}^n}
\in {\bf e}^s_{{\mathcal M}^\varphi_q,r}({\mathbb R}^n)$,
where
$j_{Q'}\equiv -\log_2 \ell(Q')$.
{
Assuming (\ref{eq:140817-301}) and (\ref{eq:140817-302})
for the time being,}
let us conclude the proof.

For $f \in {\mathcal E}^s_{{\mathcal M}^\varphi_q,r}({\mathbb R}^n)$,
we use Theorem \ref{thm:decomposition-1}(2);
\begin{equation}\label{eq:150821-80}
f=\sum_{j=0}^\infty
\left(
\sum_{m \in {\mathbb Z}^n}\lambda_{jm}a_{jm}
\right),
\end{equation}
where $\lambda=\{\lambda_{jm}\}_{j \in {\mathbb N}_0, m \in {\mathbb Z}^n}
\in {\bf e}^s_{{\mathcal M}^\varphi_q,r}({\mathbb R}^n)$
satisfies
\[
\|\lambda\|_{{\bf e}^s_{{\mathcal M}^\varphi_q,r}}
\lesssim
\|f\|_{{\mathcal E}^s_{{\mathcal M}^\varphi_q,r}}.
\]
and
$\{a_{jm}\}_{j \in {\mathbb N}_0, m \in {\mathbb Z}^n}
\in {\mathfrak A}$.
Define ${\rm Tr}f$ by:
\begin{equation}\label{eq:150821-81}
{\rm Tr}f
\equiv
\sum_{j=0}^\infty\left(
\sum_{m \in {\mathbb Z}^n}\lambda_{jm}a_{jm}(\cdot',0_n)
\right).
\end{equation}
The definition of ${\rm Tr}f$ makes {sense;
see Remark \ref{rem:150826-1}.}
Since we are assuming (\ref{eq:140820-101}),
we see
$$
\{a_{j(m',0)}(\cdot',0_n)\}_{j \in {\mathbb N}_0, {m' \in {\mathbb Z}^{n-1}}}
\in {\mathfrak A}({\mathbb R}^{n-1});
$$
recall that no moment condition (\ref{eq:140817-4}) is required
because we can suppose $L=-1$.

Since each $a_{jm}$ is supported in 
${\overline{3Q_{jm}}}=2^{-j}m+[-2^{-j},2^{j+1}]^n$,
in order that $a_{jm}(\cdot',0_n)$ is not a zero function,
we need 
$2^{-j}m_n-2^{-j}<0<2^{-j}m_n+2^{-j+1}$,
or equivalently,
$m_n=0,-1$.
Thus,
we have
\[
{\rm Tr}f
\equiv
\sum_{j=0}^\infty\left(
\sum_{m' \in {\mathbb Z}^{n-1}}
\lambda_{j(m',0)}a_{j(m',0)}(\cdot',0_n)
\right)
+
\sum_{j=0}^\infty\left(
\sum_{m' \in {\mathbb Z}^{n-1}}
\lambda_{j(m',-1)}a_{j(m',-1)}(\cdot',0_n)
\right).
\]
Let us set
$\lambda'\equiv\{\lambda_{j(m',0)}
\}_{{j\in\N_0,m' \in {\mathbb Z}^{n-1}}}$
and
$\lambda^\dagger\equiv\{\lambda_{j(m',-1)}
\}_{{j\in\N_0,m' \in {\mathbb Z}^{n-1}}}$.
Combining (\ref{eq:140817-301}) and (\ref{eq:140817-302}),
we have
$
\|\lambda'\|_{{\bf e}^{s^*}_{{\mathcal M}^{\varphi^*}_q,q}}
\lesssim
\|\lambda\|_{{\bf e}^s_{{\mathcal M}^{\varphi}_q,r}}.
$
By Theorem \ref{thm:vector-maximal},
we have a similar estimate
for $\lambda^\dagger$:
$
\|\lambda^\dagger\|_{{\bf e}^{s^*}_{{\mathcal M}^{\varphi^*}_q,q}}
\lesssim
\|\lambda\|_{{\bf e}^s_{{\mathcal M}^{\varphi}_q,r}}.
$
Hence
\[
\|{\rm Tr}f\|_{{\mathcal E}^{s^*}_{{\mathcal M}^{\varphi^*}_q,q}}
\lesssim
\|\lambda'\|_{{\bf e}^{s^*}_{{\mathcal M}^{\varphi^*}_q,q}}
+
\|\lambda^\dagger\|_{{\bf e}^{s^*}_{{\mathcal M}^{\varphi^*}_q,q}}
\lesssim
\|\lambda\|_{{\bf e}^s_{{\mathcal M}^{\varphi}_q,r}}
\lesssim
\|f\|_{{\mathcal E}^s_{{\mathcal M}^{\varphi}_q,r}}.
\]
Thus, the matters are reduced
to (\ref{eq:140817-301}) and (\ref{eq:140817-302}).

To prove (\ref{eq:140817-301}) and (\ref{eq:140817-302}),
let us set
\[
{\rm I}{\equiv}
\varphi^*(\ell(Q'))
\left(\frac{1}{|Q'|}\int_{Q'}
\sum_{j=j_{Q'}}^\infty
\left|2^{js^*}
\sum_{m' \in {\mathbb Z}^{n-1}}\lambda_{j(m',0)}\chi_{Q_{j m'}}(x')
\right|^q
\,dx'
\right)^{\frac{1}{q}}
\]
and
\[
{{\rm II}}{\equiv}
\varphi^*(\ell(Q'))
\left(\frac{1}{|Q'|}\int_{Q'}
\sum_{j=0}^{j_{Q'}}
\left|2^{js^*}
\sum_{m' \in {\mathbb Z}^{n-1}}\lambda_{j(m',0)}\chi_{Q_{j m'}}(x')
\right|^q
\,dx'
\right)^{\frac{1}{q}}.
\]
Let us start with simplifying (\ref{eq:140817-301}).
Note that 
\[
{\rm I}=
\varphi^*(\ell(Q'))
\left(\frac{1}{|Q'|}
\sum_{j=j_{Q'}}^\infty
2^{js^*q}
\sum_{\substack{m' \in {\mathbb Z}^{n-1} \\ Q_{j m'} \subset Q'}}
2^{-j(n-1)}|\lambda_{j(m',0)}|^q
\right)^{\frac{1}{q}}.
\]
If we write
\begin{equation}\label{eq:150828-1}
E(S)\equiv S \times (\ell(S),2\ell(S))
\end{equation}
and 
\begin{equation}\label{eq:150828-2}
F(S)\equiv S \times (0,2\ell(S))
\end{equation}
for 
$S \in {\mathcal D}({\mathbb R}^{n-1})$,
then
$\{E(S)\}_{S \in {\mathcal D}({\mathbb R}^{n-1})}$
is disjoint and hence
\begin{eqnarray*}
{\rm I}&=&
2^{\frac1q}
\varphi(\ell(Q'))
\left(\frac{1}{|F(Q')|}\int_{F(Q')}
\sum_{j=j_{Q'}}^\infty
2^{jsq}
\left|\sum_{\substack{m' \in {\mathbb Z}^{n-1} \\ Q_{j m'} \subset Q'}}
\lambda_{j(m',0)}\chi_{E(Q_{j m'})}(x)
\right|^q
\,dx
\right)^{\frac{1}{q}}\\
&=&
2^{\frac1q}
\varphi(\ell(Q'))
\left\{\frac{1}{|F(Q')|}\int_{F(Q')}
\left(
\sum_{j=j_{Q'}}^\infty
2^{jsr}
\left|\sum_{\substack{m' \in {\mathbb Z}^{n-1} \\ Q_{j m'} \subset Q'}}
\lambda_{j(m',0)}\chi_{E(Q_{j m'})}(x)
\right|^r
\right)^{\frac{q}{r}}
\,dx
\right\}^{\frac{1}{q}}.
\end{eqnarray*}
Observe that we have a pointwise estimate:
\[
\left|\sum_{\substack{m' \in {\mathbb Z}^{n-1} \\ Q_{j m'} \subset Q'}}
\lambda_{j(m',0)}\chi_{E(Q_{j m'})}(x)
\right|
\lesssim
\left(
M\left[
\left|\sum_{\substack{m' \in {\mathbb Z}^{n-1} \\ Q_{j m'} \subset Q'}}
\lambda_{j(m',0)}\chi_{Q_{j(m',0)}}
\right|^u
\right](x)
\right)^{\frac{1}{u}}
\]
for all $0<u<\infty$.
Set
$Q \equiv Q' \times [0,\ell(Q')]$ and $j_Q \equiv -\log_2\ell(Q)=j_{Q'}$.
By the Fefferman-Stein inequality
{for $L^q(\ell^r)$},
\begin{eqnarray*}
{\rm I}
&\lesssim&
\varphi(\ell(Q'))
\left\{\frac{1}{|F(Q')|}\int_{{{\mathbb R}^n}}
\left(
\sum_{j=j_{Q'}}^\infty
2^{jsr}
\left|\sum_{\substack{m' \in {\mathbb Z}^{n-1} \\ Q_{j m'} \subset Q'}}
\lambda_{j(m',0)}\chi_{Q_{j(m',0)})}(x)
\right|^r
\right)^{\frac{q}{r}}
\,dx
\right\}^{\frac{1}{q}}\\
&\sim& 
\varphi(\ell(Q))
\left\{\frac{1}{|Q|}\int_{Q}
\left(
\sum_{j=j_{Q'}}^\infty
2^{jsr}
\left|\sum_{\substack{m' \in {\mathbb Z}^{n-1} \\ Q_{j m'} \subset Q'}}
\lambda_{j(m',0)}\chi_{Q_{j(m',0)})}(x)
\right|^r
\right)^{\frac{q}{r}}
\,dx
\right\}^{\frac{1}{q}}
\\
&\le&
\varphi(\ell(Q))
\left\{\frac{1}{|Q|}\int_{Q}
\left(
\sum_{j=j_{Q'}}^\infty
2^{jsr}
\left|\sum_{\substack{m' \in {\mathbb Z}^{n-1} \\ Q_{j m'} \subset Q'}}
\lambda_{j(m',0)}\chi_{Q_{j(m',0)})}(x)
\right|^r
\right)^{\frac{q}{r}}
\,dx
\right\}^{\frac{1}{q}}\\
&\lesssim&
\|\lambda\|_{{\bf e}^s_{{\mathcal M}^\varphi_q,r}},
\end{eqnarray*}
which proves (\ref{eq:140817-301}).

It remains to prove
(\ref{eq:140817-302}).
For all $j$ with $j < j_Q$,
we can find a unique cube 
$Q_{j m'(j)} \in {\mathcal D}({\mathbb R}^{n-1})$ 
such that
$Q_{j m'(j)} \supset Q'$,
where $m'(j) \in {\mathbb Z}^{n-1}$.
Recall also $s^*$ is defined by (\ref{eq:s-star}).
Thus,
the left-hand side of
(\ref{eq:140817-302}) {simplifies} to read;
\begin{eqnarray*}
{\rm II}=
\varphi^*(\ell(Q'))
\left(\sum_{j=0}^{j_{Q'}}
2^{j s^*q}|\lambda_{j(m'(j),0)}|^q\right)^{\frac1q}
=
{\frac{\varphi(\ell(Q'))}{\ell(Q')^{1/q}}}
\left(\sum_{j=0}^{j_{Q'}}
2^{j s q-j}|\lambda_{j(m'(j),0)}|^q\right)^{\frac1q}.
\end{eqnarray*}
For each $j=0,1,2,\ldots,j_Q'$,
we write $Q_{(j)}$ to be the unique dyadic cube
$R \in {\mathcal D}_j$ containing $Q$.
By a trivial estimate
$2^{js}|\lambda_{j(m'(j),0)}| \le \varphi(2^{-j})^{-1}
\|\lambda\|_{{\bf e}^s_{{\mathcal M}^\varphi_q,r}}$
and (\ref{eq:varphi-star}),
we obtain
\begin{align*}
{\rm II}
&=
\frac{\varphi(\ell(Q'))}{\ell(Q')^{1/q}}
\left(\sum_{j=0}^{j_{Q'}}
2^{jsq-j}|\lambda_{j(m'(j),0)}|^q\right)^{\frac1q}\\
&=
\frac{\varphi(\ell(Q))}{\ell(Q)^{1/q}}
\left(\sum_{j=0}^{j_{Q}}
\frac{2^{-j}}{|Q|}
\int_{Q}
2^{jsq}|\lambda_{j(m'(j),0)}|^q\chi_{Q_{j(m'(j),0)}}(x)\,dx
\right)^{\frac1q}\\
&\le
\frac{\varphi(\ell(Q))}{\ell(Q)^{1/q}}
\left(\sum_{j=0}^{j_{Q}}
\frac{2^{-j}}{|Q_{(j)}|}
\int_{Q_{(j)}}
\left|2^{js}
\sum_{m \in {\mathbb Z}^n}|\lambda_{j m}|\chi_{Q_{j m}}(x)
\right|^q\,dx
\right)^{\frac1q}\\
&\le
\frac{\varphi(\ell(Q))}{\ell(Q)^{1/q}}
\left(\sum_{j=0}^{j_{Q}}
\frac{2^{-j}}{\varphi(\ell(Q_{(j)}))^q}
\right)^{\frac1q}
\|\lambda\|_{{\bf e}^s_{{\mathcal M}^\varphi_q, r}}
=
\frac{\varphi(\ell(Q))}{\ell(Q)^{1/q}}
\left(\sum_{j=0}^{j_{Q}}
\frac{\ell(Q_{(j)})}{\varphi(\ell(Q_{(j)}))^q}
\right)^{\frac1q}
\|\lambda\|_{{\bf e}^s_{{\mathcal M}^\varphi_q, r}}\\
&\lesssim
\|\lambda\|_{{\bf e}^s_{{\mathcal M}^\varphi_q, r}},
\end{align*}
where we used
(\ref{eq:150828-3}) for the last estimate.
Thus, (\ref{eq:140817-302}) is proved.
\end{proof}

A helpful remark on the definition of the trace operator
may be in order.
\begin{remark}\label{rem:150826-1}
The trace operator defined by (\ref{eq:150821-81})
is a linear operator which coincides
with (\ref{eq:150821-82}).
According to the proof of Theorem \ref{thm:decomposition-2}(1),
for each $j \in {\mathbb N}_0$ and $m \in {\mathbb Z}^n$,
there exists a continuous linear operator
$I_{jm}:{\mathcal E}^s_{{\mathcal M}^\varphi_q,r}({\mathbb R}^n)
\to L^\infty_{\rm c}({\mathbb R}^n)$
such that
$I_{jm}(f)=\lambda_{jm}a_{jm}$.
Therefore,
we can write
\[
{\rm Tr}f
=
\sum_{j=0}^\infty\left(
\sum_{m \in {\mathbb Z}^n}I_{jm}(f)(\cdot',0_n)
\right).
\]
Meanwhile if $f \in {\mathcal S}({\mathbb R}^n)$,
then the limit (\ref{eq:150821-80}) 
takes place in ${\rm BC}({\mathbb R}^n)$
since $f \in B^\varepsilon_{\infty\infty}({\mathbb R}^n)
\hookrightarrow {\rm BC}({\mathbb R}^n)$.
More precisely,
\[
\lim_{J \to \infty}
\left(
\sup_{x \in {\mathbb R}^n}
\left|
f(x)-\sum_{j=0}^J
\sum_{m \in {\mathbb Z}^n}I_{jm}(x)\right|
\right)=0.
\]
Therefore, the convergence of (\ref{eq:150821-80}) 
takes place pointwise,
meaning that 
${\rm Tr}f$, defined by (\ref{eq:150821-81}),
agrees with the standard definition.
\end{remark}

We discuss the surjectivity of the trace operator.
\begin{theorem}\label{thm:trace-2}
The trace operator defined in Theorem $\ref{thm:trace}$
in $(1)$ and $(2)$
is surjective.
\end{theorem}

\begin{proof}
We shall prove
that the trace operator 
{defined} in Theorem \ref{thm:trace}(2)
is surjective,
since the one 
{defined} in Theorem \ref{thm:trace}(1) can be proved
surjective in a similar manner.
To this end, it suffices to prove
\begin{equation}\label{eq:140817-401}
\|\lambda\|_{{\bf e}^s_{{\mathcal M}^\varphi_q},r}
\lesssim
\|\lambda'\|_{{\bf e}^{s^*}_{{\mathcal M}^{\varphi^*}_q},q},
\end{equation}
where 
$\lambda'=\{\lambda'_{jm'}\}_{j \in {\mathbb N}_0, m' \in {\mathbb Z}^{n-1}}$
is a given 
{
doubly indexed complex sequence
}
and we define a {doubly indexed complex} sequence by:
\begin{equation}\label{eq:140821-401}
\lambda
= 
\{\lambda_{jm}\}_{j \in {\mathbb N}_0, m \in {\mathbb Z}^n}
\equiv
\{\delta_{m_n0}\lambda'_{jm'}\}_{j \in {\mathbb N}_0, m \in {\mathbb Z}^n}.
\end{equation}
{Assuming} (\ref{eq:140817-401}) for a while,
let us prove that any 
$f \in {\mathcal E}^{s^*}_{{\mathcal M}^{\varphi^*}_q,r}({\mathbb R}^{n-1})$ 
is realized as
$f={\rm Tr}g$ for some
$g \in {\mathcal E}^s_{{\mathcal M}^\varphi_q,r}({\mathbb R}^n)$.
By the use of Theorem \ref{thm:decomposition-2}(1),
there {exist}
$\lambda'=\{\lambda'_{jm'}\}_{j \in {\mathbb N}_0, m' \in {\mathbb Z}^{n-1}}
\in {\bf e}^{s^*}_{{\mathcal M}^{\varphi^*}_q,r}({\mathbb R}^{n-1})$
and
$\{a_{jm}'\}_{j \in {\mathbb N}_0, m' \in {\mathbb Z}^{n-1}}
\in {\mathfrak A}({\mathbb R}^{n-1})$
such that
\[
f=\sum_{j=0}^\infty
\left(
\sum_{m' \in {\mathbb Z}^{n-1}}
\lambda_{jm'}a_{jm'}
\right)
\]
in ${\mathcal S}'({\mathbb R}^{n-1})$.
Choose a function $\Theta \in C^\infty({\mathbb R})$
such that $\chi_{[-1/4,1/4]} \le \Theta \le \chi_{[-1/2,1/2]}$.
Define
$\lambda$ by (\ref{eq:140821-401}) and
a function $A_{jm} \in C^K$
with $j \in {\mathbb N}_0$ and $m \in {\mathbb Z}^n$
by:
\[
A_{jm}(x)=A_{jm}(x',x_n)
=
\begin{cases}
a_{jm'}(x')\Theta(2^j x_n)&\mbox{\rm \, if \, }m_n=0,\\
0&\mbox{\rm \, otherwise}.
\end{cases}
\]

Let us write the left-hand side of (\ref{eq:140817-401})
out in full by using an equivalent expression:
\[
\|\lambda\|_{{\bf e}^s_{{\mathcal M}^\varphi_q,r}}
\sim
\sup_{Q \in {\mathcal D}({\mathbb R}^n)}
\varphi(\ell(Q))
\left\{\frac{1}{|Q|}\int_Q
\left(
\sum_{j=0}^\infty
2^{jsr}
\left|\sum_{m' \in {\mathbb Z}^{n-1}}
\lambda'_{jm'}\chi_{Q_{j(m',0)}}(x)\right|^r\,dx
\right)^{\frac{q}{r}}
\right\}^{\frac{1}{q}}.
\]
Then we have
\[
g\equiv
\sum_{j=0}^\infty
\left(
\sum_{m \in {\mathbb Z}^n}\lambda_{jm}A_{jm}
\right)
\in {\mathcal E}^s_{{\mathcal M}^\varphi_q,r}({\mathbb R}^n)
\]
from (\ref{eq:140817-401}) and Theorem \ref{thm:decomposition-2}(2).
Likewise, we have
${\rm Tr}g=f$
since $\Theta(0)=1$.
Thus, admitting (\ref{eq:140817-401}),
we can construct the desired $g$.

We need to prove
\begin{equation}\label{eq:140817-4021}
\varphi(\ell(Q))
\left\{\frac{1}{|Q|}
\int_{Q}
\left(
\sum_{j=0}^\infty
\left|2^{js}\sum_{m' \in {\mathbb Z}^{n-1}}
\lambda'_{jm'}\chi_{Q_{j(m',0)}}(x)\right|^r
\right)^{\frac{q}{r}}\,dx
\right\}^{\frac{1}{q}}
\lesssim
\|\lambda'\|_{{\bf e}^{s^*}_{{\mathcal M}^{\varphi^*}_q,q}}
\end{equation}
for the proof of (\ref{eq:140817-401}).
Let us write $G(Q')\equiv Q' \times (0,\ell(Q'))$
for $Q' \in {\mathcal D}({\mathbb R}^{n-1})$.
Suppose for a while that $Q \in {\mathcal D}({\mathbb R}^n)$
is expressed as $Q=G(Q')$ for some
$Q' \in {\mathcal D}({\mathbb R}^{n-1})$.
With this in mind, let us decompose
estimate (\ref{eq:140817-4021})
into two parts:
\begin{equation}\label{eq:140817-402}
{
\frac{\varphi(\ell(Q'))^q}{|G(Q')|}
\int_{G(Q')}
\left(
\sum_{j=j_{Q'}}^\infty
\left|2^{js}\sum_{m' \in {\mathbb Z}^{n-1}}
\lambda'_{jm'}\chi_{Q_{j(m',0)}}(x)\right|^r
\right)^{\frac{q}{r}}\,dx
\lesssim
(\|\lambda'\|_{{\bf e}^{s^*}_{{\mathcal M}^{\varphi^*}_q,q}})^q
}
\end{equation}
and
\begin{equation}\label{eq:140817-403}
{
\frac{\varphi(\ell(Q'))^q}{|G(Q')|}
\int_{G(Q')}
\left(
\sum_{j=0}^{j_{Q'}}
\left|2^{js}\sum_{m' \in {\mathbb Z}^{n-1}}
\lambda'_{jm'}\chi_{Q_{j(m',0)}}(x)\right|^r
\right)^{\frac{q}{r}}\,dx
\lesssim
(\|\lambda'\|_{{\bf e}^{s^*}_{{\mathcal M}^{\varphi^*}_q,q}})^q
}.
\end{equation}

Let us prove (\ref{eq:140817-402}).
Set
\[
\alpha{\equiv}\frac{\min(1,q,r)}{2}.
\]
Let us also recall that
$E(S)$ is given by (\ref{eq:150828-1}).
By the Fefferman-Stein inequality,
we have
\begin{eqnarray*}
&&\varphi(\ell(Q'))
\left\{\frac{1}{|G(Q')|}
\int_{G(Q')}
\left(
\sum_{j=j_{Q'}}^\infty
2^{jsr}\left|\sum_{m' \in {\mathbb Z}^{n-1}}
\lambda'_{jm'}\chi_{Q_{j(m',0)}}(x)\right|^r
\right)^{\frac{q}{r}}\,dx
\right\}^{\frac{1}{q}}\\
&&\lesssim
\varphi(\ell(Q'))
\left\{\frac{1}{|G(Q')|}
\int_{G(Q')}
\left(
\sum_{j=j_{Q'}}^\infty
2^{jsr}M\left[\left|\sum_{m' \in {\mathbb Z}^{n-1}}
\lambda'_{jm'}\chi_{E(Q_{j m'})}\right|^{\alpha}
\right](x)^\frac{r}{\alpha}
\right)^{\frac{q}{r}}\,dx
\right\}^{\frac{1}{q}}\\
&&\lesssim
\varphi(\ell(Q'))
\left\{\frac{1}{|G(Q')|}
\int_{{\mathbb R}^n}
\left(
\sum_{j=j_{Q'}}^\infty
2^{jsr}M\left[\left|\sum_{m' \in {\mathbb Z}^{n-1}}
\lambda'_{jm'}\chi_{E(Q_{j m'})}\right|^{\alpha}
\right](x)^\frac{r}{\alpha}
\right)^{\frac{q}{r}}\,dx
\right\}^{\frac{1}{q}}\\
&&=
\varphi(\ell(Q'))
\left\{\frac{1}{|G(Q')|}
\int_{3G(Q')}
\left(
\sum_{j=j_{Q'}}^\infty
2^{jsr}\left|\sum_{m' \in {\mathbb Z}^{n-1}}
\lambda'_{jm'}\chi_{E(Q_{j m'})}(x)\right|^r
\right)^{\frac{q}{r}}\,dx
\right\}^{\frac{1}{q}}.
\end{eqnarray*}
Note that, 
for any fixed $j \in {\mathbb N}_0$ and $x \in {\mathbb R}^n$,
there exists
at most only one $m{ \in {\mathbb Z}^n}$ such that
$\lambda'_{jm'}\chi_{E(Q_{j m'})}(x){\ne}0$.
Thus,
\begin{eqnarray*}
&&
\varphi(\ell(Q'))
\left\{\frac{1}{|G(Q')|}
\int_{3G(Q')}
\left(
\sum_{j=j_{Q'}}^\infty
2^{jsr}\left|\sum_{m' \in {\mathbb Z}^{n-1}}
\lambda'_{jm'}\chi_{E(Q_{j m'})}(x)\right|^r
\right)^{\frac{q}{r}}\,dx
\right\}^{\frac{1}{q}}\\
&&=
\varphi(\ell(Q'))
\left\{\frac{1}{|G(Q')|}
\int_{3G(Q')}
\left(
\sum_{j=j_{Q'}}^\infty
2^{jsq}\left|\sum_{m' \in {\mathbb Z}^{n-1}}
\lambda'_{jm'}\chi_{E(Q_{j m'})}(x)\right|^q
\right)^{\frac{q}{q}}\,dx
\right\}^{\frac{1}{q}}\\
&&=
\varphi(\ell(Q'))^*
\left\{\frac{1}{|Q'|}
\int_{Q'}
\left(
\sum_{j=j_{Q'}}^\infty
2^{jsq}\left|\sum_{m' \in {\mathbb Z}^{n-1}}
\lambda'_{jm'}\chi_{Q_{jm'}}(x')\right|^q
\right)^{\frac{q}{q}}\,dx'
\right\}^{\frac{1}{q}}
\le
\|\lambda'\|_{{\bf e}^{s^*}_{{\mathcal M}^{\varphi^*}_q,q}}.
\end{eqnarray*}
This proves (\ref{eq:140817-402}).

The proof of (\ref{eq:140817-403}) is similar
to (\ref{eq:140817-302}).
We omit the detail.

Finally, suppose that $Q$ is not 
of the form
$Q=G(Q')$ for some $Q' \in {\mathcal D}({\mathbb R}^{n-1})$.
This means that $Q$ 
has a form
as $Q=Q' \times [j\ell(Q'),(j+1)\ell(Q'){)}$
or $Q=Q' \times [-(j+1)\ell(Q'),-j\ell(Q'){)}$
for some $Q' \in {\mathcal D}({\mathbb R}^{n-1})$
and $j \ge 1$.
Due to symmetry let us suppose 
$Q=Q' \times [j\ell(Q'),(j+1)\ell(Q'){)}$.
If $j\ell(Q') \ge 1$,
then the left-hand side of (\ref{eq:140817-4021})
is zero and there is nothing to prove.
Assume otherwise;
$1 \le j <\ell(Q')^{-1}$.
Then by letting $w \equiv \frac{1}{2}\min{(1,q,r)}$, we have
\begin{eqnarray*}
&&\varphi(\ell(Q))
\left\{\frac{1}{|Q|}
\int_{Q}
\left(
\sum_{j=0}^\infty
\left|2^{js}\sum_{m' \in {\mathbb Z}^{n-1}}
\lambda'_{jm'}\chi_{Q_{j(m',0)}}(x)\right|^r
\right)^{\frac{q}{r}}\,dx
\right\}^{\frac{1}{q}}\\
&&\lesssim
\varphi(\ell(Q))
\left\{\frac{1}{|G(Q')|}
\int_{Q}
\left(
\sum_{j=0}^\infty
\left|2^{js}\sum_{m' \in {\mathbb Z}^{n-1}}
\lambda'_{jm'}\chi_{Q_{j(m',0)}}(x)\right|^r
\right)^{\frac{q}{r}}\,dx
\right\}^{\frac{1}{q}}\\
&&\lesssim
\varphi(\ell(Q))
\left\{\frac{1}{|G(Q')|}
\int_{Q}
\left(
\sum_{j=0}^\infty
\left|2^{js}\sum_{m' \in {\mathbb Z}^{n-1}}
\lambda'_{jm'}M[\chi_{E(Q_{jm}) \cap Q}]^{w}(x)\right|^r
\right)^{\frac{q}{r}}\,dx
\right\}^{\frac{1}{q}}\\
&&\lesssim
\varphi(\ell(Q))
\left\{\frac{1}{|G(Q')|}
\int_{Q}
\left(
\sum_{j=0}^\infty
\left|2^{js}\sum_{m' \in {\mathbb Z}^{n-1}}
\lambda'_{jm'}\chi_{E(Q_{jm}) \cap Q}(x)\right|^r
\right)^{\frac{q}{r}}\,dx
\right\}^{\frac{1}{q}}\\
&&\sim
\varphi(\ell(Q))
\left\{\frac{1}{|G(Q')|}
\int_{Q}
\left(
\sum_{j=0}^\infty
\left|2^{js}\sum_{m' \in {\mathbb Z}^{n-1}}
\lambda'_{jm'}\chi_{E(Q_{jm}) \cap Q}(x)\right|^q
\right)^{\frac{q}{q}}\,dx
\right\}^{\frac{1}{q}}\\
&&\lesssim
\|\lambda'\|_{{\bf e}^{s^*}_{{\mathcal M}^{\varphi^*}_q,q}},
\end{eqnarray*}
as was to be shown.
\end{proof}

\subsection{Pointwise multiplication}

In this section 
we {shall} prove 
the following boundedness {of} pointwise 
multiplication operators.

\begin{theorem}[{Pointwise} multiplication]
\label{thm:150821-110}
Let $0<q<\infty$, $0<r \le \infty$, $s \in {\mathbb R}$,
$k \in {\mathbb N}$ and $\varphi \in {\mathcal G}_q$.
\begin{enumerate}
\item
If
$k>s>\sigma_r$,
then the mapping
\[
g \in {\mathcal S}({\mathbb R}^n) \mapsto 
g \cdot f \in {\rm BC}^k({\mathbb R}^n)
\]
extends continuously to a bounded linear operator
from
${\mathcal N}^s_{{\mathcal M}^\varphi_q,r}({\mathbb R}^n)$
to itself
so that 
\[
\|g \cdot f\|_{{\mathcal N}^s_{{\mathcal M}^\varphi_q,r}}
\lesssim
\|g\|_{{\rm BC}^k}
\|f\|_{{\mathcal N}^s_{{\mathcal M}^\varphi_q,r}}
\]
for all 
$f \in {\mathcal N}^s_{{\mathcal M}^\varphi_q,r}({\mathbb R}^n)$
and
$g \in {\rm BC}^k({\mathbb R}^n)$.
\item
If
$k>s>\sigma_{qr}$,
then the mapping
\[
g \in {\mathcal S}({\mathbb R}^n) \mapsto 
g \cdot f \in {\rm BC}^k({\mathbb R}^n)
\]
extends continuously to a bounded linear operator
from
${\mathcal E}^s_{{\mathcal M}^\varphi_q,r}({\mathbb R}^n)$
to itself
so that 
\[
\|g \cdot f\|_{{\mathcal E}^s_{{\mathcal M}^\varphi_q,r}}
\lesssim
\|g\|_{{\rm BC}^k}
\|f\|_{{\mathcal E}^s_{{\mathcal M}^\varphi_q,r}}
\]
for all 
$f \in {\mathcal E}^s_{{\mathcal M}^\varphi_q,r}({\mathbb R}^n)$
and
$g \in {\rm BC}^k({\mathbb R}^n)$.
\end{enumerate}
\end{theorem}

\begin{proof}
We shall concentrate on generalized Triebel-Lizorkin-Morrey spaces,
since we can handle
generalized Besov-Morrey spaces similarly.
Let $\varepsilon \in(0,\rho-R)$.
Let $f \in {\mathcal E}^s_{{\mathcal M}^\varphi_q,r}({\mathbb R}^n)$.
Then there exists a {triply indexed complex sequence}
$
\lambda=\{\lambda^\beta_{\nu m}\}_{\nu \in {\mathbb N}_0,
m \in {\mathbb Z}^n, \beta \in {\mathbb N}_0{}^n}
$
satisfying
(\ref{eq:150313-3})
and
(\ref{eq:150313-4}).
Then
we claim that
$$
g \cdot f=
\sum_{\beta \in {\mathbb N}_0{}^n}
\sum_{\nu=0}^\infty
\sum_{m \in {\mathbb Z}^n}
\lambda^\beta_{\nu m}
[g(\beta qu)_{\nu m}]
$$
makes sense;
the right-hand side is convergent 
in ${\mathcal S}'({\mathbb R}^n)$
and satisfies the desired estimate.
If we set
\[
f^\beta\equiv
\sum_{\nu=0}^\infty
\sum_{m \in {\mathbb Z}^n}
\lambda^\beta_{\nu m}(\beta qu)_{\nu m}
\]
for each $\beta \in {\mathbb N}_0{}^n$,
then we have
\[
\|f^\beta\|_{{\mathcal E}^s_{{\mathcal M}^\varphi_q,r}}
\lesssim
2^{-(\rho-R-\varepsilon)|\beta|}
\|\lambda\|_{{\bf e}^{s}_{{\mathcal M}^\varphi_q,r},\rho}.
\]
By using the atomic decomposition theorem,
we obtain
\[
\|g \cdot f^\beta\|_{{\mathcal E}^s_{{\mathcal M}^\varphi_q,r}}
\lesssim
2^{-(\rho-R-\varepsilon)|\beta|}
\|g\|_{{\rm BC}^k}
\|\lambda\|_{{\bf e}^{s}_{{\mathcal M}^\varphi_q,r},\rho}
\lesssim
2^{-(\rho-R-\varepsilon)|\beta|}
\|g\|_{{\rm BC}^k}
\|f\|_{{\mathcal E}^s_{{\mathcal M}^\varphi_q,r}}.
\]
This estimate is summable over all $\beta \in {\mathbb N}_0{}^n$
with the desired estimate.
\end{proof}

\subsection{Diffeomorphism}

A $C^M$-diffeomorphism
$\psi:{\mathbb R}^n \to {\mathbb R}^n$
is said to be {\it regular},
if 
$\psi$
and its inverse
belong to 
${\rm BC}^M({\mathbb R}^n)$.

\begin{theorem}[Diffeomorphism]\label{thm:317-3}
Let $0<q<\infty$, $0<r \le \infty$, $s \in {\mathbb R}$,
$k \in {\mathbb N}$ and $\varphi \in {\mathcal G}_q$.
Assume in addition that $\psi$
is a regular $C^k$-diffeomorphism. 
\begin{enumerate}
\item
Let $k>s>\sigma_q$.
Then, the composition mapping
$
\varphi \in {\rm BC}^k({\mathbb R}^n) 
\mapsto \varphi \circ \psi \in {\rm BC}^k({\mathbb R}^n)
$
induces a continuous mapping
$
f \in {\mathcal N}^s_{{\mathcal M}^\varphi_q,r}({\mathbb R}^n) 
\mapsto 
f \circ \psi \in {\mathcal N}^s_{{\mathcal M}^\varphi_q,r}({\mathbb R}^n)
$
and, for all $f \in {\mathcal N}^s_{{\mathcal M}^\varphi_q,r}({\mathbb R}^n)$,
we {have}
$\displaystyle
\| f \circ \psi \|_{{\mathcal N}^s_{{\mathcal M}^\varphi_q,r}}
\lesssim_\psi
\| f \|_{{\mathcal N}^s_{{\mathcal M}^\varphi_q,r}}.
$
\item
Let $k>s>\sigma_{qr}$.
Then, the composition mapping
$
\varphi \in {\rm BC}^k({\mathbb R}^n) 
\mapsto \varphi \circ \psi \in {\rm BC}^k({\mathbb R}^n)
$
induces a continuous mapping
$
f \in {\mathcal E}^s_{{\mathcal M}^\varphi_q,r}({\mathbb R}^n) 
\mapsto 
f \circ \psi \in {\mathcal E}^s_{{\mathcal M}^\varphi_q,r}({\mathbb R}^n)
$
and, for all $f \in {\mathcal E}^s_{{\mathcal M}^\varphi_q,r}({\mathbb R}^n)$,
we {have}
$\displaystyle
\| f \circ \psi \|_{{\mathcal E}^s_{{\mathcal M}^\varphi_q,r}}
\lesssim_\psi
\| f \|_{{\mathcal E}^s_{{\mathcal M}^\varphi_q,r}}.
$

\end{enumerate}
\end{theorem}

{
For the proof of Theorem \ref{thm:317-3},
we need a setup.
Now that $\psi$ is bi-Lipschitz,
that is, both $\psi$ and $\psi^{-1}$
are Lipschitz continuous,
there exist $I \in {\mathbb N}$ and $D>0$
depending on $\psi$ such that;
for each $\nu$,
${\mathbb Z}^n$
is partitioned into $M_1^\nu,M_2^\nu,\ldots,M_I^\nu$,
and there exist injections
\begin{equation}
\iota_1^\nu:M_1^\nu \to {\mathbb Z}^n,\iota_2^\nu:M_2^\nu \to {\mathbb Z}^n,\ldots,
\iota_I^\nu:M_I^\nu \to {\mathbb Z}^n
\end{equation}
such that;
for all $i=1,2,\ldots,I$, 
$\nu \in {\mathbb N}_0$
and multi-index $\beta \in {\mathbb N}_0{}^n$,
we have
$$
\psi^{-1}({\rm supp}((\beta qu)_{\nu m})) 
\subset D\,Q_{\nu \, \iota_i^\nu(m)}. 
$$
Note that $\iota_i^\nu$
is a bijection
from $M_i^\nu$ to $\iota_i^\nu(M_i^\nu)$. 
For $i=1,2,\ldots,I$ and $\nu \in {\mathbb N}_0$,
we write
$\theta^\nu_i\equiv (\iota_i^\nu)^{-1}$. 

The integer $I$ is independent of $\nu$
as the following lemma shows:
\begin{lemma}
We have a bound
\begin{equation}\label{eq:130502-3}
I \lesssim 1,
\end{equation}
where the implicit constant does not depend on $\nu$.
\end{lemma}

\begin{proof}

Since
$\nabla\psi,\nabla[\psi^{-1}]$ are bounded functions, 
\[
(\|\nabla[\psi^{-1}]\|_{\infty})^{-1}|x-y|
\le
|\psi(x)-\psi(y)|
\le
\|\nabla \psi\|_{\infty}|x-y|.
\]
Here and below,
we set
\[
C_0\equiv 
\max\{\|\nabla \psi\|_{\infty},\|\nabla[\psi^{-1}]\|_{\infty}\}.
\]
Then,
$C_0{}^{-1}|x-y| \le |\psi(x)-\psi(y)| \le C_0|x-y|$.
Fix
$m_0 \in {\mathbb Z}^n$.
Once we show
that the number of $m \in {\mathbb Z}^n$
satisfying
\[
\psi^{-1}({\rm supp}(\beta qu)_{\nu m_0})
\subset DQ_{\nu m}
\]
is bounded,
we obtain the estimate of $I$ from above.

The diameter of
$\psi^{-1}({\rm supp}(\beta qu)_{\nu m_0})$,
which is given by
\[
\sup\{|x-y|\,:\,x,y \in \psi^{-1}({\rm supp}(\beta qu)_{\nu m_0})\},
\]
satisfies $2^{\nu+1}r \times C_0$.
Let
$|{\rm supp}((\beta qu)_{\nu m})| \le (2r)^n$.
Then
$\{DQ_{\nu m}\}_{m \in {\mathbb Z}^n}$
overlaps at most $[D+2]^n$ times.
That is,
\[
\sum_{m \in {\mathbb Z}^n}\chi_{DQ_{\nu m}} \le [D+2]^n.
\]
Hence,
$$\psi^{-1}({\rm supp}(\beta qu)_{\nu m_0})$$
intersects at most
$[\sqrt{n}\times 2^{1}r \times C_0+1]^n \times [D+2]^n$ cubes
belonging to 
$\{DQ_{\nu m}\}_{m \in {\mathbb Z}^n}$.
Thus, we conclude
$I \le [\sqrt{n}\times 2^{1}r \times C_0+1]^n \times [D+2]^n$
and that the proof is complete.
\end{proof}
}

\begin{proof}
We concentrate on the generalized Triebel-Lizorkin-Morrey space 
${\mathcal E}^s_{{\mathcal M}^\varphi_q,r}({\mathbb R}^n)$.
Then, maintaining the notation of Theorem \ref{thm:quark-1},
we let $\rho>R$. 
We shall invoke the quarkonial decomposition;
see Theorem \ref{thm:quark-1}.
We {expand}
$$\displaystyle
f=\sum_{\beta \in {\mathbb N}_0{}^n}\sum_{\nu \in {\mathbb N}_0}\sum_{m \in {\mathbb Z}^n}
\lambda^\beta_{\nu m}(\beta qu)_{\nu m}.
$$
Here the coefficient
$\lambda
=\{\lambda^\beta_{\nu m}\}
_{\beta \in {\mathbb N}_0{}^n, \, (\nu,m) \in {\mathbb N}_0 \times {\mathbb Z}^n}
$ 
satisfies
\begin{equation}
\label{eq:16.5}
\| \lambda \|_{{\bf e}^s_{{\mathcal M}^\varphi_q,r},\rho}
\lesssim 
\| f \|_{{\mathcal E}^s_{{\mathcal M}^\varphi_q,r}}.
\end{equation}
Here,
we set
\begin{equation*}
\lambda_{\nu,\bar{m}}^{\beta,i} \equiv 
\begin{cases} \lambda_{\nu,\theta_i^{\nu}(\bar{m})}^{\beta} \ \ &\text{$\bar{m}\in \iota_{i}^{\nu}(M_{i}^{\nu})$,} \\ 
 0 \ \ &\text{otherwise,} 
\end{cases}
\mbox{\rm \, and \,}
(\beta{\rm qu})_{\nu,\bar{m}}^i \equiv \begin{cases} (\beta{\rm qu})_{\nu, \theta_i^{\nu}(\bar{m})}\circ\psi \ \ 
&\text{$\bar{m}\in \iota_{i}^{\nu}(M_{i}^{\nu})$,} \\ 
0 \ \ &\text{otherwise}. 
\end{cases}
\end{equation*}
Then,
we want to define
\begin{equation}\label{eq:16.8}
f \circ \psi
=\sum_{i=1}^I\sum_{\beta \in {\mathbb N}_0{}^n}\sum_{\nu \in {\mathbb N}_0}
\sum_{\overline{m} \in {\mathbb Z}^n}
\lambda^{\beta,i}_{\nu \overline{m}}
(\beta qu)^i_{\nu \overline{m}}.
\end{equation}
Let us verify that the infinite sum
defining (\ref{eq:16.8}) makes sense. 
Set
$$\displaystyle
f^{i,\beta}\equiv \sum_{\nu \in {\mathbb N}_0}
\sum_{\overline{m} \in {\mathbb Z}^n}
\lambda^{\beta,i}_{\nu \overline{m}}
(\beta qu)^i_{\nu \overline{m}}
$$
for
$\beta \in {\mathbb N}_0{}^n, \, i=1,2,\ldots,I$.
Then, Theorem \ref{thm:decomposition-2}
yields
\begin{eqnarray*}
\| f^{i,\beta} \|_{{\mathcal E}^s_{{\mathcal M}^\varphi_q,r}}
\lesssim_\psi\,
2^{(R+\varepsilon)|\beta|}
\| \{\lambda^{\beta,i}_{\nu \overline{m}}\}
_{\nu \in {\mathbb N}_0, \, \overline{m} \in {\mathbb Z}^n} 
\|_{{\bf e}^s_{{\mathcal M}^\varphi_q,r}}
\lesssim_\psi\,
2^{(R+\varepsilon)|\beta|}
\| \lambda^\beta \|_{{\bf e}^s_{{\mathcal M}^\varphi_q,r}},
\end{eqnarray*}
where
$\displaystyle
\lambda^\beta
\equiv 
\{\lambda^\beta_{\nu m}\}_{(\nu,m) \in {\mathbb N}_0 \times {\mathbb Z}^n}
$. 

Let $\delta=\rho-R-\varepsilon>0$.
By the estimate of the quarkonial decompositions,
$$\displaystyle
\| f^{i,\beta} \|_{{\mathcal E}^s_{{\mathcal M}^\varphi_q,r}}
\lesssim_\psi\,
2^{(R+\varepsilon)|\beta|}
\| f^\beta \|_{{\mathcal E}^s_{{\mathcal M}^\varphi_q,r}}
\lesssim_\psi\,2^{-\delta|\beta|}
\| f \|_{{\mathcal E}^s_{{\mathcal M}^\varphi_q,r}}.
$$
 
Hence, 
if we use
the $\min(q,r,1)$-triangle inequality 
to the sum
$$\displaystyle
f \circ \psi
=\sum_{i=1}^I \sum_{\beta \in {\mathbb N}_0{}^n}f^{i,\beta},
$$
then we have
$\displaystyle
\| f \circ \psi \|_{{\mathcal E}^s_{{\mathcal M}^\varphi_q,r}}
\lesssim 
\| f \|_{{\mathcal E}^s_{{\mathcal M}^\varphi_q,r}}.
$
Hence, (\ref{eq:16.8}) defines $f \circ \psi$. 
\end{proof}

\section{Homogeneous spaces}
\label{s5.5}

\subsection{The space ${\mathcal S}_\infty'({\mathbb R}^n)$}

Our {results} in this paper 
{carry} over to the homogeneous setting.
\begin{definition}
Let $0<q<\infty$, $0<r \le \infty$, $s \in {\mathbb R}$
 and $\varphi:(0,\infty) \to (0,\infty)$
be a function in ${\mathcal G}_q$.
Let $\tau$ be compactly supported functions satisfying
\[
0 \notin {\rm supp}(\tau), \quad
\tau(\xi)>0 \mbox{\rm \, if \,} \xi \in Q(2) \setminus Q(1).
\]
define $\tau_k(\xi) \equiv \tau(2^{-k}\xi)$
for $\xi \in {\mathbb R}^n$ and $k \in {\mathbb Z}$.
\begin{enumerate}
\item
{\it The $(${homogeneous}$)$ generalized Besov-Morrey space}
$\dot{\mathcal N}_{{\mathcal M}^\varphi_q,r}^s({\mathbb R}^n)$
is the set of all
$f \in {\mathcal S}'_\infty({\mathbb R}^n)$
for which the quasi-norm
\[
\|f\|_{\dot{\mathcal N}_{{\mathcal M}^\varphi_q,r}^s}
\equiv 
\begin{cases}
\displaystyle
\left(\sum_{j=-\infty}^\infty
2^{jsr}\|\tau_j(D)f\|_{{\mathcal M}^\varphi_q}^r
\right)^{\frac1r}&(r<\infty),\\
\displaystyle
\sup_{j \in {\mathbb Z}}
2^{js}\|\tau_j(D)f\|_{{\mathcal M}^\varphi_q}
&(r=\infty)
\end{cases}
\]
is finite.
\item
{\it The $(${homogeneous}$)$ generalized Triebel-Lizorkin-Morrey space}
$\dot{\mathcal E}_{{\mathcal M}^\varphi_q,r}^s({\mathbb R}^n)$
is the set of all
$f \in {\mathcal S}'_\infty({\mathbb R}^n)$
for which the quasi-norm
\[
\|f\|_{\dot{\mathcal E}_{{\mathcal M}^\varphi_q,r}^s}
\equiv 
\begin{cases}
\displaystyle
\left\|
\left(\sum_{j=-\infty}^\infty
2^{jsr}|\tau_j(D)f|^r
\right)^{\frac1r}\right\|_{{\mathcal M}^\varphi_q}&(r<\infty),\\
\displaystyle
\left\|
\sup_{j \in {\mathbb Z}}
2^{js}|\tau_j(D)f|
\right\|_{{\mathcal M}^\varphi_q}
&(r=\infty)
\end{cases}
\]
is finite.
\item
The space
$\dot{\mathcal A}^s_{{\mathcal M}^\varphi_q,r}({\mathbb R}^n)$
denotes either
$\dot{\mathcal N}^s_{{\mathcal M}^\varphi_q,r}({\mathbb R}^n)$
or
$\dot{\mathcal E}^s_{{\mathcal M}^\varphi_q,r}({\mathbb R}^n)$.
\end{enumerate}
\end{definition}

\begin{theorem}\label{thm*:150205-1}
Assume
$(\ref{eq:Nakai-3})$
in the case
when
$\dot{\mathcal A}^s_{{\mathcal M}^\varphi_q,r}({\mathbb R}^n)
=\dot{\mathcal E}^s_{{\mathcal M}^\varphi_q,r}({\mathbb R}^n)$
with $r<\infty$.
Then
different choices of admissible $\tau$
will yield equivalent norms. 
\end{theorem}

\begin{proof}
The proof is almost the same as Theorem \ref{thm:150205-1}.
We indicate the necessary change.
Define
$\tau_k$ and $\tilde{\tau}_k$ 
by (\ref{eq:150821-112}) and (\ref{eq:150821-113}),
respectively.
Here, we let $k \in {\mathbb Z}$ instead of $k \in {\mathbb N}$.
Then we can prove (\ref{eq:150821-111})
by mimicking the proof of Theorem \ref{thm:150205-1}.
Further details are omitted.
\end{proof}

\subsection{Atomic decomposition}

{
We can consider atomic decompositions
for the homogeneous spaces.
}
The proof is similar to the nonhomogeneous case.
So, we outline the proof based on the nonhomogeneous case.
\begin{definition}
Let $0<q<\infty$, $0<r \le \infty$, $s \in {\mathbb R}$
 and $\varphi \in {\mathcal G}_q$.
\begin{enumerate}
\item
{\it The $(${homogeneous}$)$ generalized Besov-Morrey sequence space}
$\dot{\bf n}_{{\mathcal M}^\varphi_q,r}^s({\mathbb R}^n)$
is the set of all 
{
doubly indexed complex sequences
}
$\lambda
=\{\lambda_{jm}\}_{j \in {\mathbb Z}, \, m \in {\mathbb Z}^n}$
for which the quasi-norm
\[
\|\lambda\|_{\dot{\bf n}_{{\mathcal M}^\varphi_q,r}^s}
\equiv 
\begin{cases}
\displaystyle
\left(\sum_{j=-\infty}^\infty
2^{jsr}
\left\|\sum_{m \in {\mathbb Z}^n}\lambda_{jm}\chi_{Q_{jm}}
\right\|_{{\mathcal M}^\varphi_q}^r
\right)^{\frac1r}&(r<\infty),\\
\displaystyle
\sup_{j \in {\mathbb Z}}
2^{js}
\left\|\sum_{m \in {\mathbb Z}^n}\lambda_{jm}\chi_{Q_{jm}}
\right\|_{{\mathcal M}^\varphi_q}
&(r=\infty)
\end{cases}
\]
is finite.
\item
{\it The $(${homogeneous}$)$ generalized Triebel-Lizorkin-Morrey sequence space}
$\dot{\bf e}_{{\mathcal M}^\varphi_q,r}^s({\mathbb R}^n)$
is the set of all
{
doubly indexed complex sequences
$\lambda=\{\lambda_{jm}\}_{j\in\Z, m\in\Z^n}$
}
for which the quasi-norm
\[
\|\lambda\|_{\dot{\bf e}_{{\mathcal M}^\varphi_q,r}^s}
\equiv 
\begin{cases}
\displaystyle
\left\|
\left\{\sum_{j=-\infty}^\infty
2^{jsr}
\left(
\sum_{m \in {\mathbb Z}^n}|\lambda_{jm}|\chi_{Q_{jm}}
\right)^r
\right\}^{\frac1r}\right\|_{{\mathcal M}^\varphi_q}&(r<\infty),\\
\displaystyle
\left\|
\sup_{j \in {\mathbb Z}}
2^{js}
\left(\sum_{m \in {\mathbb Z}^n}|\lambda_{jm}|\chi_{Q_{jm}}\right)
\right\|_{{\mathcal M}^\varphi_q}
&(r=\infty)
\end{cases}
\]
is finite.
\item
The space
$\dot{\bf a}^s_{{\mathcal M}^\varphi_q,r}({\mathbb R}^n)$
denotes either
$\dot{\bf n}^s_{{\mathcal M}^\varphi_q,r}({\mathbb R}^n)$
or
$\dot{\bf e}^s_{{\mathcal M}^\varphi_q,r}({\mathbb R}^n)$.
Assume $(\ref{eq:Nakai-3})$
in the case
when
$\dot{\bf a}^s_{{\mathcal M}^\varphi_q,r}({\mathbb R}^n)
=\dot{\bf e}^s_{{\mathcal M}^\varphi_q,r}({\mathbb R}^n)$
with $r<\infty$.
\end{enumerate}
\end{definition}

\begin{definition}
Let $L \in {\mathbb N}_0 \cup \{-1\}$ and $K \in {\mathbb N}_0$.
\begin{enumerate}
\item
Let $j \in {\mathbb Z}$ and $m \in {\mathbb Z}^n$.
A $C^K$-function $a:{\mathbb R}^n \to {\mathbb C}$
is said to be a $(K,L)$-atom supported near $Q_{jm}$,
if $(\ref{eq:140817-3})$ with $|\alpha| \le K$
and $(\ref{eq:140817-4})$ with $|\beta| \le L$ 
hold.
When $L=-1$, it is understood that 
$(\ref{eq:140817-4})$ is a void condition.
\item
Denote by $\dot{\mathfrak A}=\dot{\mathfrak A}({\mathbb R}^n)$ 
the set of all collections
$\{a_{jm}\}_{j \in {\mathbb Z}, m \in {\mathbb Z}^n}$
of $C^K$-functions
such that
each $a_{jm}$ is a $(K,L)$-atom supported near $Q_{jm}$.
\end{enumerate}
\end{definition}

\begin{theorem}\label{thm*:decomposition-1}
Let $0<q<\infty$, $0<r \le \infty$, $s \in {\mathbb R}$
 and $\varphi \in {\mathcal G}_q$.
Let also $L \in {\mathbb N}_0 \cup \{-1\}$ and $K \in {\mathbb N}_0$
satisfy
\[
K \ge [1+s]_+, \quad L \ge \max(-1,[\sigma_q-s]),
\]
where
$\sigma_q \equiv n\left(\frac{1}{q}-1\right)_+$.
\begin{enumerate}
\item
Let $f \in \dot{\mathcal N}^s_{{\mathcal M}^\varphi_q,r}({\mathbb R}^n)$.
Then there exist a family
$\{a_{jm}\}_{j \in {\mathbb Z}, m \in {\mathbb Z}^n}
\in \dot{\mathfrak A}$
and a 
{
doubly indexed complex sequence
}
$\lambda=\{\lambda_{jm}\}_{j \in {\mathbb Z}, m \in {\mathbb Z}^n}
\in \dot{\bf n}^s_{{\mathcal M}^\varphi_q,r}({\mathbb R}^n)$
such that
\begin{equation}\label{eq:150311-1}
f=\sum_{j=-\infty}^\infty 
\left(\sum_{m \in {\mathbb Z}^n}\lambda_{jm}a_{jm}\right)
\mbox{\rm \, in \, }
{\mathcal S}_\infty'({\mathbb R}^n)
\end{equation}
and that
\[
\|\lambda\|_{\dot{\bf n}^s_{{\mathcal M}^\varphi_q,r}}
\lesssim
\|f\|_{\dot{\mathcal N}^s_{{\mathcal M}^\varphi_q,r}}.
\]
\item
Let
$\{a_{jm}\}_{j \in {\mathbb Z}, m \in {\mathbb Z}^n}
\in \dot{\mathfrak A}$
and 
$\lambda=\{\lambda_{jm}\}_{j \in {\mathbb Z}, m \in {\mathbb Z}^n}
\in \dot{\bf n}^s_{{\mathcal M}^\varphi_q,r}({\mathbb R}^n)$.
Then
$$
f \equiv \sum_{j=-\infty}^\infty 
\left(
\sum_{m \in {\mathbb Z}^n}\lambda_{jm}a_{jm}
\right)
$$
converges in ${\mathcal S}'_\infty({\mathbb R}^n)$ 
and belongs to
$\dot{\mathcal N}^s_{{\mathcal M}^\varphi_q,r}({\mathbb R}^n)$.
Furthermore,
\begin{equation}\label{eq:150821-114}
\|f\|_{\dot{\mathcal N}^s_{{\mathcal M}^\varphi_q,r}}
\lesssim
\|\lambda\|_{\dot{\bf n}^s_{{\mathcal M}^\varphi_q,r}}.
\end{equation}
\end{enumerate}
\end{theorem}

\begin{proof}
(1) is already obtained in \cite[Theorem 10.15]{LSUYY2}.
The proof of (2) is almost the same as Theorem \ref{thm:decomposition-1}.
The only difference is the care of the convegence
of the sum
$\displaystyle \sum_{j=-J}^\infty 
\left(\sum_{m \in {\mathbb Z}^n}\lambda_{jm}a_{jm}\right)$
as $J \to \infty$.

By invoking (\ref{eq:140817-2}),
we can prove (\ref{eq:150821-114})
when there exists $J$ such that $\lambda_{jm}=0$
for all $j \in {\mathbb Z}$ and $m \in {\mathbb Z}^n$
with $|j| \ge J$.
Let $\delta>0$ be given by $(\ref{eq:delta})$. 
Then we have
\[
\left\|\sum_{m \in {\mathbb Z}^n}\lambda_{jm}a_{jm}
\right\|_{\dot{\mathcal N}^{s+\delta/2}_{{\mathcal M}^\varphi_q,r}}
\lesssim
2^{j\delta/2}
\|\lambda\|_{\dot{\bf n}^s_{{\mathcal M}^\varphi_q,r}}
\]
for all $j \in {\mathbb Z} \cap (-\infty,0]$
from Corollary \ref{cor:150817-2} and (\ref{eq:150821-114}).
Thus,
since we can show that
$\dot{\mathcal N}^{s+\delta/2}_{{\mathcal M}^\varphi_q,r}({\mathbb R}^n)
\hookrightarrow {\mathcal S}_\infty'({\mathbb R}^n)$
analogously to the nonhomogeneous case,
\[
f_- \equiv
\sum_{j=-J}^0 
\left(\sum_{m \in {\mathbb Z}^n}\lambda_{jm}a_{jm}\right)
\]
is convergent
in ${\mathcal S}_\infty'({\mathbb R}^n)$.
Thus, we are in the position of applying the Fatou property
as we did in (\ref{eq:150821-115}).
\end{proof}

\begin{theorem}\label{thm*:decomposition-2}
Let $0<q<\infty$, $0<r \le \infty$, $s \in {\mathbb R}$
 and $\varphi:(0,\infty) \to (0,\infty) \in {\mathcal G}_q$.
Let also $L \in {\mathbb N}_0 \cup \{-1\}$ and $K \in {\mathbb N}_0$.
Assume
\[
K \ge [1+s]_+, \quad L \ge \max(-1,[\sigma_{qr}-s]),
\]
where $\sigma_{qr} \equiv \max(\sigma_q,\sigma_r)$.
\begin{enumerate}
\item
Let $f \in \dot{\mathcal E}^s_{{\mathcal M}^\varphi_q,r}({\mathbb R}^n)$.
Then there exist a family
$\{a_{jm}\}_{j \in {\mathbb Z}, m \in {\mathbb Z}^n}
\in \dot{\mathfrak A}$
and a 
{
doubly indexed complex sequence
}
$\lambda=\{\lambda_{jm}\}_{j \in {\mathbb Z}, m \in {\mathbb Z}^n}
\in \dot{\bf e}^s_{{\mathcal M}^\varphi_q,r}({\mathbb R}^n)$
satisfying 
$(\ref{eq:150311-1})$ 
and that
\begin{equation}\label{eq:140817-1038}
\|\lambda\|_{\dot{\bf e}^s_{{\mathcal M}^\varphi_q,r}}
\lesssim
\|f\|_{\dot{\mathcal E}^s_{{\mathcal M}^\varphi_q,r}}.
\end{equation}
\item
Let
$\{a_{jm}\}_{j \in {\mathbb Z}, m \in {\mathbb Z}^n}
\in \dot{\mathfrak A}$
and 
$\lambda=\{\lambda_{jm}\}_{j \in {\mathbb Z}, m \in {\mathbb Z}^n}
\in \dot{\bf e}^s_{{\mathcal M}^\varphi_q,r}({\mathbb R}^n)$.
Then
$$
f \equiv \sum_{j=-\infty}^\infty 
\left(\sum_{m \in {\mathbb Z}^n}\lambda_{jm}a_{jm}\right)
$$
converges in ${\mathcal S}'_\infty({\mathbb R}^n)$ 
and belongs to
$\dot{\mathcal E}^s_{{\mathcal M}^\varphi_q,r}({\mathbb R}^n)$.
Furthermore,
\[
\|f\|_{\dot{\mathcal E}^s_{{\mathcal M}^\varphi_q,r}}
\lesssim
\|\lambda\|_{\dot{\bf e}^s_{{\mathcal M}^\varphi_q,r}}.
\]
\end{enumerate}
\end{theorem}
 
\begin{proof}
Combine the ideas of 
Theorems 
\ref{thm:decomposition-2}
and
\ref{thm*:decomposition-1}.
\end{proof}

\subsection{Molecular decomposition}

{
As a direct corollary of
Theorems \ref{thm*:decomposition-18} and \ref{thm*:decomposition-21}, 
we can show that
${\mathcal S}_\infty{({\mathbb R}^n)} \subset 
\dot{\mathcal A}^s_{{\mathcal M}^\varphi_q,r}({\mathbb R}^n)$.
}
\begin{definition}
Let $L \in {\mathbb N}_0 \cup \{-1\}$ and {{ $K, N \in {\mathbb N}_0$ be such that $N>K+n$. }}
\begin{enumerate}
\item
Let $j \in {\mathbb Z}$ and $m \in {\mathbb Z}^n$.
A $C^K$-function ${{b}}:{\mathbb R}^n \to {\mathbb C}$
is said to be a $(K,L,N)$-molecule supported near $Q_{jm}$,
if $(\ref{eq:140817-3})$ with $|\alpha| \le K$
and $(\ref{eq:140817-4})$ with $|\beta| \le L$ 
hold.
When $L=-1$, it is understood that 
$(\ref{eq:140817-4})$ is a void condition.
\item
Denote by $\dot{\mathfrak M}=\dot{\mathfrak M}({\mathbb R}^n)$ 
the set of all collections
$\{{{b_{jm}}}\}_{j \in {\mathbb Z}, m \in {\mathbb Z}^n}$
of $C^K$-functions
such that
each ${{b_{jm}}}$ is a $(K,L,{{N}})$-molecule supported near $Q_{jm}$.
\end{enumerate}
\end{definition}

\begin{theorem}\label{thm*:decomposition-18}
Let $0<q<\infty$, $0<r \le \infty$, $s \in {\mathbb R}$
 and $\varphi \in {\mathcal G}_q$.
Let also $L \in {\mathbb N}_0 \cup \{-1\}$ and $K \in {\mathbb N}_0$.
Assume
\[
K \ge [1+s]_+, \quad L \ge \max(-1,[\sigma_q-s]),
\]
where
$\sigma_q \equiv n\left(\frac{1}{q}-1\right)_+$.
\begin{enumerate}
\item
Let $f \in \dot{\mathcal N}^s_{{\mathcal M}^\varphi_q,r}({\mathbb R}^n)$.
Then there exist a family
$\{{{b_{jm}}}\}_{j \in {\mathbb Z}, m \in {\mathbb Z}^n}
\in \dot{\mathfrak M}$
and a 
{
doubly indexed complex sequence
}
$\lambda=\{\lambda_{jm}\}_{j \in {\mathbb Z}, m \in {\mathbb Z}^n}
\in \dot{\bf n}^s_{{\mathcal M}^\varphi_q,r}({\mathbb R}^n)$
such that
\begin{equation}\label{eq:150311-1*}
f=\sum_{j=-\infty}^\infty 
\left(\sum_{m \in {\mathbb Z}^n}\lambda_{jm}{{b_{jm}}}\right)
\mbox{\rm \, in \, }
{\mathcal S}_\infty'({\mathbb R}^n)
\end{equation}
and that
\[
\|\lambda\|_{\dot{\bf n}^s_{{\mathcal M}^\varphi_q,r}}
\lesssim
\|f\|_{\dot{\mathcal N}^s_{{\mathcal M}^\varphi_q,r}}.
\]
\item
Let
$\{{{b_{jm}}}\}_{j \in {\mathbb Z}, m \in {\mathbb Z}^n}
\in \dot{\mathfrak M}$
and 
$\lambda=\{\lambda_{jm}\}_{j \in {\mathbb Z}, m \in {\mathbb Z}^n}
\in \dot{\bf n}^s_{{\mathcal M}^\varphi_q,r}({\mathbb R}^n)$.
Then
$$
f \equiv \sum_{j=-\infty}^\infty 
\left(
\sum_{m \in {\mathbb Z}^n}\lambda_{jm}{{b_{jm}}}
\right)
$$
converges in ${\mathcal S}'_\infty({\mathbb R}^n)$ 
and belongs to
$\dot{\mathcal N}^s_{{\mathcal M}^\varphi_q,r}({\mathbb R}^n)$.
Furthermore,
\[
\|f\|_{\dot{\mathcal N}^s_{{\mathcal M}^\varphi_q,r}}
\lesssim
\|\lambda\|_{\dot{\bf n}^s_{{\mathcal M}^\varphi_q,r}}.
\]
\end{enumerate}
\end{theorem}

\begin{proof}
Combine the ideas of 
Theorems 
\ref{thm:decomposition-101}
and
\ref{thm*:decomposition-1}.
\end{proof}

\begin{theorem}\label{thm*:decomposition-21}
Let $0<q<\infty$, $0<r \le \infty$, $s \in {\mathbb R}$ 
and $\varphi \in {\mathcal G}_q$.
Let also $K \in {\mathbb N}_0$ and $L \in {\mathbb N}_0 \cup \{-1\}$.
Assume
\[
K \ge [1+s]_+, \quad L \ge \max(-1,[\sigma_{qr}-s]),
\]
where $\sigma_{qr} \equiv \max(\sigma_q,\sigma_r)$.
\begin{enumerate}
\item
Let $f \in \dot{\mathcal E}^s_{{\mathcal M}^\varphi_q,r}({\mathbb R}^n)$.
Then there exist a family
$\{{{b_{jm}}}\}_{j \in {\mathbb Z}, m \in {\mathbb Z}^n}
\in \dot{\mathfrak M}$
and a 
{
doubly indexed complex sequence
}
$\lambda=\{\lambda_{jm}\}_{j \in {\mathbb Z}, m \in {\mathbb Z}^n}
\in \dot{\bf e}^s_{{\mathcal M}^\varphi_q,r}({\mathbb R}^n)$
satisfying 
$(\ref{eq:150311-1*})$ 
and that
\begin{equation}\label{eq:140817-1038*}
\|\lambda\|_{\dot{\bf e}^s_{{\mathcal M}^\varphi_q,r}}
\lesssim
\|f\|_{\dot{\mathcal E}^s_{{\mathcal M}^\varphi_q,r}}.
\end{equation}
\item
Let
$\{{{b_{jm}}}\}_{j \in {\mathbb Z}, m \in {\mathbb Z}^n}
\in \dot{\mathfrak M}$
and 
$\lambda=\{\lambda_{jm}\}_{j \in {\mathbb Z}, m \in {\mathbb Z}^n}
\in \dot{\bf e}^s_{{\mathcal M}^\varphi_q,r}({\mathbb R}^n)$.
Then
$$
f \equiv \sum_{j=-\infty}^\infty 
\left(\sum_{m \in {\mathbb Z}^n}\lambda_{jm}{{b_{jm}}}\right)
$$
converges in ${\mathcal S}'_\infty({\mathbb R}^n)$ 
and belongs to
$\dot{\mathcal E}^s_{{\mathcal M}^\varphi_q,r}({\mathbb R}^n)$.
Furthermore,
\[
\|f\|_{\dot{\mathcal E}^s_{{\mathcal M}^\varphi_q,r}}
\lesssim
\|\lambda\|_{\dot{\bf e}^s_{{\mathcal M}^\varphi_q,r}}.
\]
\end{enumerate}
\end{theorem}

\begin{proof}
Combine the ideas of 
Theorems 
\ref{thm:decomposition-102}
and
\ref{thm*:decomposition-1}.
\end{proof}

\subsection{Quarkonial decomposition}

As we have seen in the nonhomogeneous case,
quarkonial decompositon can be obtained
on the basis of the atomic decomposition.
We content ourselves with indicating how to modify
the related definition and stating our results without proofs.
\begin{definition}[Regular quark]
Let
$\beta \in {\mathbb N}_0{}^n, \ \nu \in {\mathbb Z}$
and 
$m \in {\mathbb Z}^n$. 
Then define
a function $\psi^\beta$
and
the quark 
$(\beta qu)_{\nu m}\equiv (\beta qu)_{\nu,m}$
by $(\ref{eq:150311-2})$.
{Each}
$(\beta qu)_{\nu m}$
is called the {\it quark}.
\end{definition}

\begin{definition}[Sequence spaces for quarkonial decomposition]
{Let $R,\rho>0$ satisfy $(\ref{quark-1})$ and $\rho>R$.}
For a triply indexed complex sequence
$\lambda
=\{\lambda^\beta_{\nu m}\}_{\beta \in {\mathbb N}_0{}^n, \ 
\nu \in {\mathbb Z}, \ m \in {\mathbb Z}^n},
$
define
\begin{equation}
\lambda^\beta
\equiv
\{\lambda^\beta_{\nu m}\}
_{\nu \in {\mathbb Z}, \ m \in {\mathbb Z}^n}, \quad 
\| \lambda \|_{\dot{\bf a}_{{\mathcal M}^s_q,r},\rho}
\equiv 
\sup_{\beta \in {\mathbb N}_0{}^n}
2^{\rho|\beta|}\| \lambda^\beta \|_{\dot{\bf a}^s_{{\mathcal M}^\varphi_q,r}}.
\end{equation} 
\end{definition}

\begin{theorem}\label{thm*:quark-1}
Let $0<q<\infty$, $0<r \le \infty$, $s \in {\mathbb R}$ 
and $\varphi \in {\mathcal G}_q$.
Assume $\rho{\equiv}[R+1]>R$,
where $R$ is a constant in $(\ref{quark-1})$.
\begin{enumerate}
\item
Let $s>\sigma_q$ and 
$f \in \dot{\mathcal N}^s_{{\mathcal M}^\varphi_q,r}({\mathbb R}^n)$.
Then there exists a 
{
triply indexed complex sequence
}
\[
\lambda=\{\lambda^\beta_{\nu m}\}_{\nu \in {\mathbb Z},
m \in {\mathbb Z}^n, \beta \in {\mathbb N}_0{}^n}
\]
such that
\begin{equation}\label{eq:15020698-1}
f=
\sum_{\beta \in {\mathbb N}_0{}^n}
\sum_{\nu=-\infty}^\infty
{
\sum_{m \in {\mathbb Z}^n}
}
\lambda^\beta_{\nu m}
(\beta qu)_{\nu m}
\end{equation}
converges in ${\mathcal S}'_\infty({\mathbb R}^n)$
and
\begin{equation}\label{eq:15020680-2}
\|\lambda\|_{\dot{\bf n}_{{\mathcal M}^s_q,r},\rho}
\lesssim
\|f\|_{\dot{\mathcal N}^s_{{\mathcal M}^\varphi_q,r}}.
\end{equation}
The constant $\lambda^\beta_{\nu m}$ depends
continuously and linearly on $f$.
\item
If $s>\sigma_q$ and 
$\lambda=\{\lambda^\beta_{\nu m}\}_{\nu \in {\mathbb Z},
m \in {\mathbb Z}^n, \beta \in {\mathbb N}_0{}^n}$
satisfies
$\|\lambda\|_{\dot{\bf n}_{{\mathcal M}^s_q,r},\rho}
<\infty,$
then
\[
f\equiv
\sum_{\beta \in {\mathbb N}_0{}^n}
\sum_{\nu=-\infty}^\infty
{
\sum_{m \in {\mathbb Z}^n}
}
\lambda^\beta_{\nu m}
(\beta qu)_{\nu m}
\]
converges in ${\mathcal S}'_\infty({\mathbb R}^n)$
and belongs to
$\dot{\mathcal N}^s_{{\mathcal M}^\varphi_q,r}({\mathbb R}^n)$.
Furthermore,
\[
\|f\|_{\dot{\mathcal N}^s_{{\mathcal M}^\varphi_q,r}}
\lesssim
\|\lambda\|_{\dot{\bf n}_{{\mathcal M}^s_q,r},\rho}.
\]
\item
If $s>\sigma_{qr}$ and 
$f \in \dot{\mathcal E}^s_{{\mathcal M}^\varphi_q,r}({\mathbb R}^n)$,
then there exists a 
{
triply indexed complex sequnece
}
$\lambda=\{\lambda^\beta_{\nu m}\}_{\nu \in {\mathbb Z},
m \in {\mathbb Z}^n, \beta \in {\mathbb N}_0{}^n}$
such that
\[
f=
\sum_{\beta \in {\mathbb N}_0{}^n}
\sum_{\nu=-\infty}^\infty
{
\sum_{m \in {\mathbb Z}^n}
}
\lambda^\beta_{\nu m}
(\beta qu)_{\nu m}
\]
in ${\mathcal S}'_\infty({\mathbb R}^n)$
and
\[
\|\lambda\|_{\dot{\bf e}_{{\mathcal M}^s_q,r},\rho}
\lesssim
\|f\|_{\dot{\mathcal E}^s_{{\mathcal M}^\varphi_q,r}}.
\]
The constant $\lambda^\beta_{\nu m}$ depends
continuously and linearly on $f$.
\item
If $s>\sigma_{qr}$ and 
$\lambda=\{\lambda^\beta_{\nu m}\}_{\nu \in {\mathbb Z},
m \in {\mathbb Z}^n, \beta \in {\mathbb N}_0{}^n}$
satisfies
$\|\lambda\|_{\dot{\bf e}_{{\mathcal M}^s_q,r},\rho}
<\infty,$
{then
\[
f\equiv
\sum_{\beta \in {\mathbb N}_0{}^n}
\sum_{\nu=-\infty}^\infty
\sum_{m \in {\mathbb Z}^n}
\lambda^\beta_{\nu m}
(\beta qu)_{\nu m}
\]
converges} 
in ${\mathcal S}'_\infty({\mathbb R}^n)$
and belongs to
$\dot{\mathcal E}^s_{{\mathcal M}^\varphi_q,r}({\mathbb R}^n)$.
Furthermore,
\[
\|f\|_{\dot{\mathcal E}^s_{{\mathcal M}^\varphi_q,r}}
\lesssim
\|\lambda\|_{\dot{\bf e}_{{\mathcal M}^s_q,r},\rho}.
\]
\end{enumerate}
\end{theorem}

\begin{proof}
{Let us prove (1) and (3).}
We modify (\ref{eq:150821-121}) 
by using (\ref{eq:150821-122}) as follows:
\[
f=
\frac{1}{\sqrt{(2\pi)^n}}
\sum_{\nu=-\infty}^\infty
\left(
\sum_{m \in {\mathbb Z}^n}
\varphi_\nu(D)f(2^{-\nu}m){\mathcal F}^{-1}\kappa(2^\nu \cdot-m)
\right),
\]
where the convergence takes place
in ${\mathcal S}_\infty'({\mathbb R}^n)$.
Then, we go through the same argument
as Theorem \ref{thm:quark-1}.

We can prove {(2) and (4)} by using
Theorems \ref{thm*:decomposition-1} and \ref{thm*:decomposition-2}
instead of
Theorems \ref{thm:decomposition-1} and \ref{thm:decomposition-2}.
\end{proof}
\subsection{Fundamental theorems--trace operator, pointwise multiplication and diffeomorphism}

The results are obtained similarly to the {ones} for nonhomogeneous spaces;
we content ourselves with stating the results.

\begin{theorem}\label{thm*:trace}
{Let $n \ge 2$.}
Suppose that we are given parameters
$(q,r,s) \in (0,\infty) \times (0,\infty] \times (0,\infty)$ 
and a function $\varphi \in {\mathcal G}_q$.
Define $s^*$ and $\varphi^*$ by $(\ref{eq:s-star})$
and $(\ref{eq:varphi-star})$
\begin{enumerate}
\item
Let $s$ satisfy $(\ref{eq:150311-9})$.
Then
we can extend
the trace operator
$f \in {\mathcal S}_\infty({\mathbb R}^n) 
\mapsto f(\cdot',0_n) \in C^\infty({\mathbb R}^n)$
to a bounded surjective linear operator from
$\dot{\mathcal N}^s_{{\mathcal M}^\varphi_q,r}({\mathbb R}^n)$
to
$\dot{\mathcal N}^{s^*}_{{\mathcal M}^{\varphi^*}_q,r}({\mathbb R}^{n-1})$.
\item
Let $s$ satisfy $(\ref{eq:140820-101})$.
Then
we can extend
the trace operator
$f \in {\mathcal S}_\infty({\mathbb R}^n) 
\mapsto f(\cdot',0_n) \in C^\infty({\mathbb R}^n)$
to a bounded surjective linear operator from
$\dot{\mathcal E}^s_{{\mathcal M}^\varphi_q,r}({\mathbb R}^n)$
to
$\dot{\mathcal E}^{s^*}_{{\mathcal M}^{\varphi^*}_q,q}({\mathbb R}^{n-1})$.
\end{enumerate}
\end{theorem}

\begin{proof}
We just mimic the proof of Theorems \ref{thm:trace} and \ref{thm:trace-2}
by replacing ${\mathbb N}_0$ by ${\mathbb Z}$
and $j_Q'=\max(0,-\log_2\ell(Q))$ by $j_Q \equiv -\log_2\ell(Q)$,
respectively.
\end{proof}

\begin{theorem}[{Pointwise} multiplication]
Let $0<q<\infty$, $0<r \le \infty$, $s \in {\mathbb R}$,
$k \in {\mathbb N}$ and $\varphi \in {\mathcal G}_q$.
\begin{enumerate}
\item
If
$k>s>\sigma_r$,
then the mapping
\[
g \in {\mathcal S}_\infty({\mathbb R}^n) \mapsto 
g \cdot f \in {\rm BC}^k({\mathbb R}^n)
\]
extends continuously to a bounded linear operator
from
$\dot{\mathcal N}^s_{{\mathcal M}^\varphi_q,r}({\mathbb R}^n)$
to itself
so that 
\[
\|g \cdot f\|_{\dot{\mathcal N}^s_{{\mathcal M}^\varphi_q,r}}
\lesssim
\|g\|_{{\rm BC}^k}
\|f\|_{\dot{\mathcal N}^s_{{\mathcal M}^\varphi_q,r}}
\]
for all 
$f \in \dot{\mathcal N}^s_{{\mathcal M}^\varphi_q,r}({\mathbb R}^n)$
and
$g \in {\rm BC}^k({\mathbb R}^n)$.
\item
If
$k>s>\sigma_{qr}$,
then the mapping
\[
g \in {\mathcal S}_\infty({\mathbb R}^n) \mapsto 
g \cdot f \in {\rm BC}^k({\mathbb R}^n)
\]
extends continuously to a bounded linear operator
from
$\dot{\mathcal E}^s_{{\mathcal M}^\varphi_q,r}({\mathbb R}^n)$
to itself
so that 
\[
\|g \cdot f\|_{\dot{\mathcal E}^s_{{\mathcal M}^\varphi_q,r}}
\lesssim
\|g\|_{{\rm BC}^k}
\|f\|_{\dot{\mathcal E}^s_{{\mathcal M}^\varphi_q,r}}
\]
for all 
$f \in \dot{\mathcal E}^s_{{\mathcal M}^\varphi_q,r}({\mathbb R}^n)$
and
$g \in {\rm BC}^k({\mathbb R}^n)$.
\end{enumerate}
\end{theorem}

\begin{proof}
Mimic the proof of Theorem \ref{thm:150821-110}.
\end{proof}

\begin{theorem}[Diffeomorphism]\label{thm*:317-3}
Let $0<q<\infty$, $0<r \le \infty$, $s \in {\mathbb R}$,
$k \in {\mathbb N}$ and $\varphi \in {\mathcal G}_q$.
Assume in addition that $\psi$
is a regular $C^k$-diffeomorphism. 
\begin{enumerate}
\item
Let $k>s>\sigma_q$.
Then, the composition mapping
$
\varphi \in {{\mathcal S}_\infty({\mathbb R}^n)}
\mapsto \varphi \circ \psi \in {\rm BC}^k({\mathbb R}^n)
$
induces a continuous mapping
$
f \in \dot{\mathcal N}^s_{{\mathcal M}^\varphi_q,r}({\mathbb R}^n) 
\mapsto 
f \circ \psi \in \dot{\mathcal N}^s_{{\mathcal M}^\varphi_q,r}({\mathbb R}^n)
$
and, for all $f \in \dot{\mathcal N}^s_{{\mathcal M}^\varphi_q,r}({\mathbb R}^n)$,
we {have}
$\displaystyle
\| f \circ \psi \|_{\dot{\mathcal N}^s_{{\mathcal M}^\varphi_q,r}}
\lesssim_\psi
\| f \|_{\dot{\mathcal N}^s_{{\mathcal M}^\varphi_q,r}}.
$
\item
Let $k>s>\sigma_{qr}$.
Then, the composition mapping
$
\varphi \in {{\mathcal S}_\infty({\mathbb R}^n)}
\mapsto \varphi \circ \psi \in {\rm BC}^k({\mathbb R}^n)
$
induces a continuous mapping
$
f \in \dot{\mathcal E}^s_{{\mathcal M}^\varphi_q,r}({\mathbb R}^n) 
\mapsto 
f \circ \psi \in \dot{\mathcal E}^s_{{\mathcal M}^\varphi_q,r}({\mathbb R}^n)
$
and, for all $f \in \dot{\mathcal E}^s_{{\mathcal M}^\varphi_q,r}({\mathbb R}^n)$,
we {have}
$\displaystyle
\| f \circ \psi \|_{\dot{\mathcal E}^s_{{\mathcal M}^\varphi_q,r}}
\lesssim_\psi
\| f \|_{\dot{\mathcal E}^s_{{\mathcal M}^\varphi_q,r}}.
$
\end{enumerate}
\end{theorem}

\begin{proof}
Mimic the proof of Theorem \ref{thm:317-3}.
\end{proof}

\subsection{Generalized Hardy spaces and generalized Triebel-Lizorkin-Morrey spaces}

{Let $0<q<\infty$ and $\varphi \in {\mathcal G}_q$.}
The generalized Hardy-Morrey space
$H{\mathcal M}^{\varphi}_{q}({\mathbb R}^n)$
is the set of all
$f \in {\mathcal S}'({\mathbb R}^n)$ 
satisfying
$\sup\limits_{t>0}|e^{t\Delta}f(\cdot)|
\in {\mathcal M}^{\varphi}_{q}({\mathbb R}^n)$.
We equip
$H{\mathcal M}^{\varphi}_{q}({\mathbb R}^n)$
with the following norm:
\begin{equation}\label{HM}
\|f\|_{H{\mathcal M}^{\varphi}_{q}}
\equiv
\left\|\sup\limits_{t>0}|e^{t\Delta}f|
\right\|_{{\mathcal M}^{\varphi}_{q}}
\quad
(f \in H{\mathcal M}^{\varphi}_{q}({\mathbb R}^n)).
\end{equation}

We invoke the following decomposition result 
from \cite[Theorem 15]{AGNS15} when $0<{q} \le 1$
and
from \cite[Theorem 1.1]{GHS15} when $1 \le {q}<\infty$.
{Here ${\mathcal Q}$ denotes the set of all cubes.}
\begin{lemma}\label{lem:2}
Let $0<{q}\le 1$,
{$\varphi \in {\mathcal G}_q$} 
and $f \in H{\mathcal M}^{\varphi}_{q}({\mathbb R}^n)$.
Let $L \in ({\mathbb N} \cup \{0\}) \cap 
[\sigma_{q},\infty).$
Assume that
${\varphi},\eta \in {\mathcal G}_1$ satisfy
\begin{equation}\label{eq:140327-3}
\int_r^\infty {\varphi}(s)\frac{ds}{s} 
{\lesssim \varphi}(r)
\end{equation}
for $r>0$.
Then 
there exists a triplet
$\{\lambda_j\}_{j=1}^\infty \subset [0,\infty)$,
$\{Q_j\}_{j=1}^\infty \subset {\mathcal Q}({\mathbb R}^n)$
and
$\{a_j\}_{j=1}^\infty \subset L^\infty({\mathbb R}^n)$
such that
$f=\sum_{j=1}^\infty \lambda_j a_j$
in ${\mathcal S}'({\mathbb R}^n)$ and that,
for all $v>0$
\begin{equation}\label{eq:thm2-1}
|a_j| \le 
\frac{\chi_{Q_j}}{\|\chi_{Q_j}\|_{{\mathcal M}^{\varphi}_{q}}}, \quad
\int_{{\mathbb R}^n}x^\alpha a_j(x)\,dx=0, \quad
\left\|\left(\sum_{j=1}^\infty
\left(
\frac{\lambda_j}{\|\chi_{Q_j}\|_{{\mathcal M}^{\varphi}_{q}}}
\chi_{Q_j}\right)^{v}
\right)^{1/v}\right\|_{{\mathcal M}^{\varphi}_{q}}
\lesssim_v\|f\|_{H{\mathcal M}^{\varphi}_{q}}
\end{equation}
for all {multi-indexes} $\alpha$ with $|\alpha| \le L$.
\end{lemma}

Going through the same argument
as \cite[Theorem 5.5]{NaSa12-2} and \cite[Theorem 5.5]{NaSa2014},
we can prove the following theorem;
\begin{lemma}\label{t5.4}
{Let $0<q<\infty$.}
Let $\varphi \in {\mathcal G}_q$ satisfy
$(\ref{eq:140327-3})$.
Let $k \in \mathcal{S}({\mathbb R}^n)$.
Write
\[
A_m
\equiv
\sup_{x \in {\mathbb R}^n}|x|^{n+m}|\nabla^m k(x)|
\quad (m \in {\mathbb N} \cup \{0\}).
\]
Define a convolution operator $T$ by
\[
T f(x) \equiv k*f(x) \quad (f \in {\mathcal S}'({\mathbb R}^n)).
\]
Then, $T$, restricted 
to $H{\mathcal M}^{{\varphi}_q}({\mathbb R}^n)$,
is an
$H{\mathcal M}^{{\varphi}_q}({\mathbb R}^n)$-bounded 
operator
and the norm depends only on $\|\mathcal{F}k\|_{L^\infty}$
and a finite number of collections $A_1,A_2,\ldots,A_N$
with $N$ depending only on ${\varphi}$.
\end{lemma}

\begin{proof}
We follow \cite[Proposition 5.3]{NaSa12-2}.
Let $Q$ be a cube and $b$ be a measurable function
satisfying
$|b| \le \chi_Q$
and
$\displaystyle
\int_{{\mathbb R}^n}x^\alpha b(x)\,dx=0
$
for all $|\alpha| \le L$.
In view of the moment condition, we have
\[
|e^{t\Delta}[k*b](x)|
\lesssim
\chi_{2Q}(x)
|e^{t\Delta}[k*b](x)|
+
\frac{\ell(Q)^{n+L+1}}{\ell(Q)^{n+L+1}+|x-c(Q)|^{n+L+1}}
\]
for all $t>0$ and hence
\[
{\mathcal M}[k*b](x)
\lesssim
\chi_{2Q}(x)
{\mathcal M}[k*b](x)
+
\frac{\ell(Q)^{n+L+1}}{\ell(Q)^{n+L+1}+|x-c(Q)|^{n+L+1}},
\]
as was to be shown.
\end{proof}

Once Lemma \ref{t5.4} is proved,
we can obtain the Littlewood-Paley decomposition
in the same way
as \cite[Theorem 5.7]{NaSa12-2} and \cite[Theorem 5.10]{NaSa2014}.
See \cite[Theorem 3.8]{AGNS15} when $0<{q} \le 1$.
The same proof works when $1<{q}<\infty$
since we have
$H{\mathcal M}^{\varphi}_{q}({\mathbb R}^n)
={\mathcal M}^{\varphi}_{q}({\mathbb R}^n)$
from \cite[Proposition 5.1]{GHS15}.
\begin{corollary}\label{thm:LP}
Let ${\varphi} \in {\mathcal G}_{q}$ satisfy
$(\ref{eq:140327-3})$.
\begin{enumerate}
\item
Let $0<{q} \le 1$.
Then
$\dot{\mathcal E}^0_{{\mathcal M}^{\varphi}_{q},2}({\mathbb R}^n)
\simeq
H{\mathcal M}^{\varphi}_{q}({\mathbb R}^n)$.
\item
Let $1<{q}<\infty$.
Then
$\dot{\mathcal E}^0_{{\mathcal M}^{\varphi}_{q},2}
({\mathbb R}^n)
\simeq
{\mathcal M}^{\varphi}_{q}({\mathbb R}^n)$.
\end{enumerate}
\end{corollary}

\subsection{Observations of the space ${\mathcal S}_\infty'({\mathbb R}^n)$}

Since ${\mathcal S}_\infty({\mathbb R}^n)$ is continuously embedded
into ${\mathcal S}({\mathbb R}^n)$,
the dual operator $R$, called the restriction,
is continuous from ${\mathcal S}'({\mathbb R}^n)$
to ${\mathcal S}'_\infty({\mathbb R}^n)$.
The following theorem is a folklore fact{:}
\begin{theorem}\label{thm:150301-1}
The restriction mapping 
$R:f \in {\mathcal S}'({\mathbb R}^n) \mapsto f|{\mathcal S}_\infty'({\mathbb R}^n)
\in {\mathcal S}'({\mathbb R}^n)$ is open,
namely the image $R(U)$ is open in ${\mathcal S}_\infty'({\mathbb R}^n)$
for any open set $U$ in ${\mathcal S}'({\mathbb R}^n)$.
\end{theorem}
However, the proof can not be found in any literature.
{It seems that we can not readily use
the closed graph theorem for a certain class
of topological vector spaces.}
We therefore supply the {self-contained and elementary} proof.
{
Theorem \ref{thm:150301-1} can be used
to consider the function spaces of homogeneous type;
see Section \ref{s6}.
}

First,
we will invoke the following statement of the Hahn-Banach theorem 
from \cite[p. 75 1.9.7 Corollary]{Meg}.
\begin{lemma}\label{lem:150228-1}
Let $Y$ be a closed subspace of a normed space $X$.
Suppose that $x \in X \setminus Y$.
Then, there is a bounded linear functional $f$ on $X$
such that $\|f\|=1$, that $f(x)=d(x,Y)$ and that $Y \subset \ker(f)$.
\end{lemma}

Denote by
${\mathcal V}_N({\mathbb R}^n)$
the closure of 
${\mathcal S}({\mathbb R}^n)$
and by
${\mathcal V}_{N,\infty}({\mathbb R}^n)$
the closure of 
${\mathcal S}_\infty({\mathbb R}^n)$,
where the closure is considered
with respect to $p_N$.
\begin{remark}\label{rem:150228-1}
Since $\displaystyle \int_{{\mathbb R}^n}f(x)\,dx=0$
for all $f \in {\mathcal V}_{N+1,\infty}({\mathbb R}^n)$,
the Gaussian function $f(x)=e^{-|x|^2}$ does not belong
to ${\mathcal V}_{N+1,\infty}({\mathbb R}^n)$.
Thus,
${\mathcal S}({\mathbb R}^n) \setminus {\mathcal V}_{N+1,\infty}({\mathbb R}^n)
\ne \emptyset$.
\end{remark}

\begin{proposition}\label{prop:150228-1}
Let $\Phi_1,\Phi_2,\ldots,\Phi_L \in {\mathcal V}_N({\mathbb R}^n)$
be a finite sequence.
Then the space
${\mathcal V}_{N,\infty}({\mathbb R}^n)+
{\rm Span}(\{\Phi_1,\Phi_2,\ldots,\Phi_L\})$
is a closed subspace of ${\mathcal V}_N({\mathbb R}^n)$.
\end{proposition}

\begin{proof}
We start with a setup.
We may assume that the system
\begin{equation}\label{eq:150311-21}
{\mathcal K}=\{[\Phi_1],[\Phi_2],\ldots,[\Phi_L]\}
\end{equation}
are linearly independent 
in ${\mathcal V}_N({\mathbb R}^n)/{\mathcal V}_{N,\infty}({\mathbb R}^n)$
According to Lemma \ref{lem:150228-1},
we can find a bounded linear functional 
$\ell_1:{\mathcal V}_N({\mathbb R}^n) \to {\mathbb C}$
such that ${\mathcal V}_{N,\infty}({\mathbb R}^n) \subset \ker(\ell_1)$
and that $\ell_1(\Phi_1)=1$.
Inductively, we can construct
$\ell_2,\ldots,\ell_L:{\mathcal V}_N({\mathbb R}^n) \to {\mathbb C}$
such that
$\ell_l(\Phi_l)=1$
and that
${\mathcal V}_{N,\infty}({\mathbb R}^n) \cup\{\Phi_1,\Phi_2,\ldots,\Phi_{l-1}\} 
\subset \ker(\ell_1)$.
If we consider some linear combinations,
we can suppose that $\ell_l(\Phi_{l'})=\delta_{l,l'}$
for all $1 \le l,l' \le L$.

Let $\{\tau_k\}_{k=1}^\infty$ be a sequence
in 
${\mathcal V}_{N,\infty}({\mathbb R}^n)+
{\rm Span}(\{\Phi_1,\Phi_2,\ldots,\Phi_L\})$
convergent to
$\tau \in {\mathcal V}_N({\mathbb R}^n)$.
Then, we have
\[
\tau_k=\sum_{l=1}^N a_{lk}\Phi_l+\zeta_k
\]
for some $\zeta_k \in {\mathcal V}_{N,\infty}({\mathbb R}^n)$
and $a_{lk} \in {\mathbb C}, l=1,2,\ldots,L$.
In terms of the $\ell_l$'s,
we have
\[
\tau_k=\sum_{l=1}^N \ell_l(\tau_k)\Phi_l+\zeta_k
\mbox{ or equivalently }
\zeta_k=\tau_k-\sum_{l=1}^N \ell_l(\tau_k)\Phi_l.
\]
Since by letting $k \to \infty$,
we have that $\zeta_k,k=1,2,\ldots$ converges
to a function $\zeta \in {\mathcal V}_N({\mathbb R}^n)$.
Since $\zeta_k \in {\mathcal V}_{N,\infty}({\mathbb R}^n)$,
we have $\zeta \in {\mathcal V}_{N,\infty}({\mathbb R}^n)$.
Thus,
\[
\tau=\sum_{l=1}^N \ell_l(\tau)\Phi_l+\zeta
\in {\mathcal V}_{N,\infty}({\mathbb R}^n)+
{\rm Span}(\{\Phi_1,\Phi_2,\ldots,\Phi_L\}),
\]
{as was to be shown.}
\end{proof}

The following lemma is somehow well known.
However, for convenience for readers we supply the proof.
\begin{lemma}
Let $g \in {\mathcal S}_\infty'({\mathbb R}^n)$.
Then there exists $N \in {\mathbb N}$ such that
\begin{equation}\label{eq:090227-1006}
|\langle g,\varphi \rangle| \le Np_N(\varphi)
\end{equation}
for all 
$\varphi\in {\mathcal S}_\infty({\mathbb R}^n)$.
\end{lemma}

\begin{proof}
Suppose that
$g:{\mathcal S}_\infty({\mathbb R}^n) \to {\mathbb C}$
is continuous;
our task is to find $N \in {\mathbb N}$
such that
(\ref{eq:090227-1006}) holds.
By the continuity of $g$,
the set
\begin{equation}
g^{-1}(\{z \in {\mathbb C} \, : \, |z|<1 \})
=
\{\varphi \in {\mathcal S}_\infty({\mathbb R}^n) \, : \, |g(\varphi)|< 1 \}
\end{equation}
is an open set of ${\mathcal S}_\infty({\mathbb R}^n)$ that contains $0$.

Therefore,
if $L \in {\mathbb N}$ is sufficiently large,
we conclude
\begin{equation}\label{eq:130522-1}
\{\varphi \in {\mathcal S}_\infty({\mathbb R}^n) \, : \, L\,p_L(\varphi)<1 \}
\subset
\{\varphi \in {\mathcal S}_\infty({\mathbb R}^n) \, : \, |g(\varphi)|< 1 \}.
\end{equation}
Hence,
$|g(\varphi)| \le 1$
as long as
$\varphi \in {\mathcal S}$
satisfies
$\displaystyle L\,p_{2L}(\varphi)=\frac{1}{2}$.

Now,
we suppose that we are given
$\varphi \in {\mathcal S}_\infty({\mathbb R}^n) \setminus \{0\}$.
Then,
$\displaystyle
\psi\equiv \frac{1}{2L\,p_{2L}(\varphi)}\varphi
$
satisfies
$\displaystyle L\,p_{2L}(\psi)=\frac12$.
Thus,
$|g(\psi)| \le 1$.
In view of the definition of $\psi$,
we have
$|g(\varphi)| \le 2L\,p_{2L}(\varphi), \varphi \in {\mathcal S}_\infty({\mathbb R}^n) \setminus \{0\}$.
The case when $\varphi=0$
can be readily incorporated.
Thus, by letting $N=2L$,
we can choose $N$ satisfying
(\ref{eq:090227-1006}).
\end{proof}

Note that
${\mathcal V}_N({\mathbb R}^n)$
is a subset of
all $C^N({\mathbb R}^n)$-functions and that
\begin{equation}\label{eq:150222-1}
\int_{{\mathbb R}^n}x^\alpha g(x)\,dx=0
\end{equation}
for all $g \in {\mathcal V}_{N,\infty}({\mathbb R}^n)$
and
$|\alpha| \le N-n-1$.
With this in mind, let us prove the following theorem:
\begin{theorem}\label{thm:150311-2}
Let 
$
\varphi_1,\varphi_2,\ldots,\varphi_K \in {\mathcal S}_\infty({\mathbb R}^n)
$
and
$
\Phi_1,\Phi_2,\ldots,\Phi_L \in {\mathcal S}({\mathbb R}^n)
\setminus {\mathcal S}_\infty({\mathbb R}^n).
$
Assume in addition that
$
[\Phi_1],[\Phi_2],\ldots,[\Phi_L] \in {\mathcal S}({\mathbb R}^n)
\setminus {\mathcal S}_\infty({\mathbb R}^n)
$
is linearly independent
in ${\mathcal S}({\mathbb R}^n)/{\mathcal S}_\infty({\mathbb R}^n)$.
Then the image of
\[
{\mathcal U}
{\equiv}
\left(\bigcap_{k=1}^K
\{F \in {\mathcal S}'({\mathbb R}^n)\,:\,
|\langle F,\varphi_k \rangle|<1 \}\right)
\cap
\left(\bigcap_{l=1}^L
\{F \in {\mathcal S}'({\mathbb R}^n)\,:\,
|\langle F,\Phi_l \rangle|<1 \}\right)
\]
by $R$ is exactly
\[
U{\equiv}
\bigcap_{k=1}^K
\{f \in {\mathcal S}_\infty'({\mathbb R}^n)\,:\,
|\langle f,\varphi_k \rangle|<1 \}.
\]
\end{theorem}

\begin{proof}
It is trivial that
$R({\mathcal U}) \subset U$.
Let $f \in U$ to prove 
$R({\mathcal U}) \supset U$.
Since
\[
\Phi_1,\Phi_2,\ldots,\Phi_L \in {\mathcal S}({\mathbb R}^n)
\setminus {\mathcal S}_\infty({\mathbb R}^n),
\]
there exist
$\alpha_1,\alpha_2,\ldots,\alpha_L \in {\mathbb N}_0{}^n$
such that
\begin{equation}\label{eq:150222-2}
\int_{{\mathbb R}^n}x^{\alpha_l} \Phi_l(x)\,dx \ne 0
\end{equation}
Since $f \in {\mathcal S}_\infty'({\mathbb R}^n)$,
there exists an integer $N$ such that
\[
|\langle f,\varphi \rangle|
\le Np_N(\varphi)
\]
for all $\varphi \in {\mathcal S}_\infty({\mathbb R}^n)$.
Note that $f$ continuously extends
to ${\mathcal V}_{N,\infty}({\mathbb R}^n)$.
So, we regard $f$ as a continuous linear mapping
defined on ${\mathcal V}_{N,\infty}({\mathbb R}^n)$.

If necessary by replacing $N$ with a large one,
we may assume
\[
N>|\alpha_1|+|\alpha_2|+\cdots+|\alpha_L|+n+1.
\]
It follows
from (\ref{eq:150222-1}) and (\ref{eq:150222-2})
that
\[
\Phi_l \in 
{\mathcal V}_N({\mathbb R}^n) \setminus {\mathcal V}_{N,\infty}({\mathbb R}^n)
\quad
(l=1,2,\ldots,L).
\]
Let $p$ be the projection
from
${\mathcal V}_{N,\infty}({\mathbb R}^n)+
{\rm Span}(\{\Phi_1,\Phi_2,\ldots,\Phi_L\})$
to
${\mathcal V}_{N,\infty}({\mathbb R}^n)$.
Let us set
$$H\equiv
f \circ p:{\mathcal V}_{N,\infty}({\mathbb R}^n)+
{\rm Span}(\{\Phi_1,\Phi_2,\ldots,\Phi_L\})
\to {\mathbb C}.
$$
Then since $p$ is continuous
and ${\mathcal V}_{N,\infty}({\mathbb R}^n)+
{\rm Span}(\{\Phi_1,\Phi_2,\ldots,\Phi_L\})$
is a closed subspace of
${\mathcal V}_N({\mathbb R}^n)$,
$H$ extends to a bounded linear functional $F$
on ${\mathcal V}_N({\mathbb R}^n)$.
Note that
\begin{equation}\label{eq:150828-101}
\langle F,\Phi_l \rangle
=\langle H,\Phi_l \rangle=0
\end{equation}
for all $l=1,2,\ldots,L$
and that 
$f=R(F)$, which implies
$F \in {\mathcal U}$ as well.
Therefore, $R({\mathcal U})=U$.
\end{proof}

\begin{theorem}\label{thm:150311-29}
Let 
$
\varphi_1,\varphi_2,\ldots,\varphi_K \in {\mathcal S}_\infty({\mathbb R}^n)
$
and
$
\Phi_1,\Phi_2,\ldots,\Phi_L,\ldots,\Phi_{L^*} \in {\mathcal S}({\mathbb R}^n)
\setminus {\mathcal S}_\infty({\mathbb R}^n).
$
Assume in addition that
$
[\Phi_1],[\Phi_2],\ldots,[\Phi_L] \in {\mathcal S}({\mathbb R}^n)
\setminus {\mathcal S}_\infty({\mathbb R}^n)
$
is linearly independent
in ${\mathcal S}({\mathbb R}^n)/{\mathcal S}_\infty({\mathbb R}^n)$
and that
${\mathcal S}_\infty({\mathbb R}^n)$
and
$
\Phi_1,\Phi_2,\ldots,\Phi_L \in {\mathcal S}({\mathbb R}^n)
\setminus {\mathcal S}_\infty({\mathbb R}^n)
$
span
$
\Phi_{L+1},\ldots,\Phi_{L^*}.
$
More precisely, we assume
\begin{equation}\label{eq:150505-1}
\Phi_l=\varphi^*_l+
\sum_{l=1}^L \beta_{l,k}\Phi_l
\quad (l=L+1,\ldots,L^*)
\end{equation}
for some $\varphi^*_{L+1},\ldots,\varphi^*_{L^*}
\in {\mathcal S}_\infty({\mathbb R}^n)$.
Then the image of
\begin{align*}
{\mathcal U}
&{\equiv}
\left(\bigcap_{k=1}^K
\{F \in {\mathcal S}'({\mathbb R}^n)\,:\,
|\langle F,\varphi_k \rangle|<1 \}\right)
\cap
\left(\bigcap_{l=1}^{L^*}
\{F \in {\mathcal S}'({\mathbb R}^n)\,:\,
|\langle F,\Phi_l \rangle|<1 \}\right)
\end{align*}
by $R$ contains
\[
U{\equiv}
\left(\bigcap_{k=1}^K
\{f \in {\mathcal S}_\infty'({\mathbb R}^n)\,:\,
|\langle f,\varphi_k \rangle|<1 \}\right)
\cap
\left(\bigcap_{k=L+1}^{L^*}
\{f \in {\mathcal S}_\infty'({\mathbb R}^n)\,:\,
|\langle f,\varphi^*_k \rangle|<1 \}
\right).
\]
\end{theorem}

\begin{proof}
We know that the image of
\begin{align*}
\tilde{\mathcal U}
&{\equiv}
\left(\bigcap_{k=1}^K
\{F \in {\mathcal S}'({\mathbb R}^n)\,:\,
|\langle F,\varphi_k \rangle|<1 \}\right)
\cap
\left(\bigcap_{l=1}^L
\{F \in {\mathcal S}'({\mathbb R}^n)\,:\,
|\langle F,\Phi_l \rangle|<1 \}\right)\\
&\quad
\cap
\left(\bigcap_{k=L+1}^{L^*}
\{F \in {\mathcal S}'({\mathbb R}^n)\,:\,
|\langle F,\varphi^*_k \rangle|<1 \}\right)
\end{align*}
by $R$ is exactly
\[
U{\equiv}
\left(\bigcap_{k=1}^K
\{f \in {\mathcal S}_\infty'({\mathbb R}^n)\,:\,
|\langle f,\varphi_k \rangle|<1 \}\right)
\cap
\left(\bigcap_{k=L+1}^{L^*}
\{f \in {\mathcal S}_\infty'({\mathbb R}^n)\,:\,
|\langle f,\varphi^*_k \rangle|<1 \}\right)
\]
thanks to Theorem \ref{thm:150311-2}.
According to (\ref{eq:150828-101})
and (\ref{eq:150505-1}),
we can say that
the image of 
\begin{align*}
\tilde{\mathcal U}^*
&{\equiv}
\left(\bigcap_{k=1}^K
\{F \in {\mathcal S}'({\mathbb R}^n)\,:\,
|\langle F,\varphi_k \rangle|<1 \}\right)
\cap
\left(\bigcap_{l=1}^{L^*}
\{F \in {\mathcal S}'({\mathbb R}^n)\,:\,
|\langle F,\Phi_l \rangle|<1 \}\right)\\
&\quad
\cap
\left(\bigcap_{k=L+1}^{L^*}
\{F \in {\mathcal S}'({\mathbb R}^n)\,:\,
|\langle F,\varphi^*_k \rangle|<1 \}\right)
\end{align*}
is $U$.
Since ${\mathcal U}$ contains $\tilde{\mathcal U}^*$,
the image of ${\mathcal U}$ by $R$ contains
$U$.
\end{proof}

Now the proof of Theorem \ref{thm:150301-1} is easy.
In fact,
let ${\mathcal U}_0$ be a neighborhood of $0$.
Then according to Theorem \ref{thm:150311-29},
${\mathcal U}_0$ contains a set of the form
${\mathcal U}$ described in Theorem \ref{thm:150311-29}.
According to Theorem \ref{thm:150311-29},
we know that $0 \in U=R({\mathcal U}) \subset R({\mathcal U}_0)$.
Thus, $0$ is an interior point of $R({\mathcal U}_0)$.
By the translation, we can show that any point $R(f)$ with $f \in R({\mathcal U}_0)$
can be proved to be an interior point of $R({\mathcal U}_0)$.

Let us rephrase Theorem \ref{thm:150301-1} in terms of the quotient topology.
To begin with let us recall some elementary facts
on general topology.

\begin{definition}
An equivalence relation
of a set $X$ is a subset $R$ of $X \times X$
satisfying the following.
Below, for $x,y \in X$, $x \sim y$ means that $(x,y) \in R$.
\begin{enumerate}
\item $x \sim x$ for all $x \in X$ (Reflexivity).
\item Let $x,y \in X$.
Then $x \sim y$ implies $y \sim x$ (Symmetry).
\item Let $x,y,z \in X$.
Then $x \sim y$ and $y \sim z$
implies $x \sim z$ ({Transitivity}).
\end{enumerate}
In this case
$\sim$ is an equivalence relation of $X$.
Given an equivalence relation of $X$,
we write
\begin{equation}
[x]
\equiv 
\left\{
y \in X \, : \,x \sim y
\right\}
\in 2^X
\end{equation}
for $x \in X$
and
\begin{equation}
X/\sim
\equiv 
\left\{
[x] \, : \,x \in X
\right\}
\subset 2^X.
\end{equation}
\end{definition}

\begin{definition}[Quotient topology]
Let $\sim$ be an equivalence relation of a topological space $X$.
Then the quotient topology of $X$
with respect to $\sim$ is the weakest topology
such that
\begin{equation}
p:X \to X/\sim, x \mapsto [x]
\end{equation}
is continuous.
The mapping $p$ is said to be the (canonical/natural) projection.
\end{definition}

According to the definition,
we see that $U \subset X/\sim$ is open
if and only if $p^{-1}(U)$ is open.

As for this topology
the following is elementary.
\begin{theorem}
\label{thm:quotient topology}
Let $X$ and $Y$ be topological spaces
and $\sim$ an equivalence relation of $X$.
A mapping $f:X/\sim \to Y$
is continuous,
if and only if $f \circ p:X \to Y$
is continuous.
\end{theorem}
Equip ${\mathcal S}'({\mathbb R}^n)/{\mathcal P}({\mathbb R}^n)$
with the quotient topology.
\begin{theorem}
The spaces
${\mathcal S}'({\mathbb R}^n)/{\mathcal P}({\mathbb R}^n)$
and
${\mathcal S}_\infty'({\mathbb R}^n)$
are homeomorphic.
\end{theorem}

\begin{proof}
According to Theorem \ref{thm:quotient topology},
the mapping 
$\Phi:[f] \in {\mathcal S}'({\mathbb R}^n)/{\mathcal P}({\mathbb R}^n)
\mapsto R(f) \in {\mathcal S}_\infty'({\mathbb R}^n)$
is continuous.
Let $O$ be an open set in 
${\mathcal S}'({\mathbb R}^n)/{\mathcal P}({\mathbb R}^n)$.
Then $O=p({\mathcal O}+{\mathcal P}({\mathbb R}^n))$
for some open set ${\mathcal O}$ in ${\mathcal S}'({\mathbb R}^n)$.
Thus,
$\Phi(O)=R({\mathcal O})$ is open according to
Theorem \ref{thm:150301-1}.
\end{proof}

Finally to conclude this section,
we prove another property of ${\mathcal S}_\infty({\mathbb R}^n)$.

\begin{theorem}\label{thm:150312-22}
Assume that $A$ is a bounded set in ${\mathcal S}_\infty({\mathbb R}^n)$,
that is,
$$
a_N=\sup_{f \in A}p_N(f)<\infty
$$
for all $N \in {\mathbb N}$.
Then $A$ is a relatively compact set.
\end{theorem}

\begin{proof}
Since ${\mathcal S}_\infty({\mathbb R}^n)$ is metrizable,
we have only to show that
any sequence $\{f_j\}_{j=1}^\infty$ in $A$
has a convergent subsequence.
Since $a_{N+1}<\infty$,
we can use the Ascoli-Arzel\'{a} theorem
to have a subsequence convergent 
in ${\mathcal V}_N({\mathbb R}^n)$ from $\{f_j\}_{j=1}^\infty$.
Cantor's diagonal argument yields a subsequence
convergent in ${\mathcal S}({\mathbb R}^n)$.
Thus, $A$ is relatively compact.
\end{proof}

\begin{remark}
See \cite[Theorem 2.2]{GoLo10}
for the extension of Theorem \ref{thm:150312-22}.
\end{remark}

\section{Appendix: Comparison with other function spaces--history and possible extension of our results}
\label{s6}

\subsection{The characterization by means of the Peetre maximal operator}

A recent trend on theory of function spaces
such as Morrey spaces, Herz spaces and Orlicz spaces
is to connect these spaces with the Littlewood-Paley
decomposition.
The idea of replacing $L^p({\mathbb R}^n)$ spaces
with other function spaces are expanded
in \cite{HeNe07,LSUYY2,RU,Ullrich10}.
Such attempts are made 
for 
$B_\sigma$-function spaces,
variable Lebesgue spaces
and
Orlicz spaces.
See
\cite{KMNS13-2},
\cite{MNSaSh083,NaSa12-2} 
and
\cite{NaSa2014},
respectively.

To recall the results,
we use the notation based on \cite{SaTa2007}.
Let $f \in {\mathcal S}'({\mathbb R}^n)$.
Define the {\it $($nonhomogeneous$)$ Besov-Morrey norm} by:
\begin{equation}
\label{eq:nonhomogeneous-Besov-Morrey}
\rVert f \lVert_{\mathcal{N}_{pqr}^s}
\equiv
\rVert \psi(D)f \lVert_{\mathcal{M}^p_q}
+
\left(
\sum_{j=1}^\infty
2^{jsr}\rVert \varphi_j(D)f \lVert_{\mathcal{M}^p_q}^r
\right)^{1/r}
\end{equation}
for $0<q \le p<\infty, \, 0<r \le \infty$
and the {\it $($nonhomogeneous$)$ Triebel-Lizorkin norm} by:
\begin{equation}
\label{eq:nonhomogeneous-Triebel-Lizorkin-Morrey}
\rVert f \lVert_{\mathcal{E}_{pqr}^s}
\equiv
\rVert \psi(D)f \lVert_{\mathcal{M}^p_q}
+
\left\|
\left(
\sum_{j=1}^\infty
2^{jsr}|\varphi_j(D)f|^r
\right)^{1/r}\right\|_{\mathcal{M}^p_q}
\end{equation}
for $0<q \le p<\infty$ and $0<r \le \infty$,
where a natural modification is made
in 
(\ref{eq:nonhomogeneous-Besov-Morrey})
and 
(\ref{eq:nonhomogeneous-Triebel-Lizorkin-Morrey}).

Now we can characterize
our function spaces by means of the Peetre maximal operator
in the spirit of \cite{LSUYY1,LSUYY2,Ullrich10}.
\begin{lemma}\label{lem:140817-1}
{
Let $\psi_j$ be as in Proposition $\ref{prop:Rychkov}$.
}
We define
\begin{equation}\label{eq:150205-1}
(\psi_jf)_*(x)
\equiv 
\sup_{y \in {\mathbb R}^n}\frac{|f*\psi_j(y)|}{(1+2^j|x-y|)^N}
\end{equation}
for $j \in {\mathbb N}_0$ and $f \in {\mathcal S}'({\mathbb R}^n)$.
Let $0<q<\infty$, $0<r \le \infty$ and 
$\varphi\in {\mathcal G}_q$.
\begin{enumerate}
\item
Let 
\begin{equation}\label{eq:140821-1}
N>\dfrac{n}{\min(1,q)}+n.
\end{equation}
For any $f \in {\mathcal N}^s_{{\mathcal M}^\varphi_q,r}({\mathbb R}^n)$,
we have
\begin{equation}\label{eq:140817-201}
\left(\sum_{j=0}^\infty
(\|2^{js}(\psi_jf)_*\|_{{\mathcal M}^\varphi_q})^r
\right)^{\frac1r}
\lesssim
\|f\|_{{\mathcal N}^s_{{\mathcal M}^\varphi_q,r}}.
\end{equation}
\item
Let 
\begin{equation}\label{eq:140821-2}
N>\dfrac{n}{\min(1,q,r)}+n.
\end{equation}
For any $f \in {\mathcal E}^s_{{\mathcal M}^\varphi_q,r}({\mathbb R}^n)$,
we have
\begin{equation}\label{eq:140817-202}
\left\|\left(\sum_{j=0}^\infty
2^{jsr}((\psi_jf)_*)^r
\right)^{\frac1r}\right\|_{{\mathcal M}^\varphi_q}
\lesssim
\|f\|_{{\mathcal E}^s_{{\mathcal M}^\varphi_q,r}}.
\end{equation}
\end{enumerate}
\end{lemma}
The proof is a direct consequence of Theorem \ref{thm:PPN}.

Once such a characterization is obtained,
we are in the position of applying a result 
in \cite[{Section 4}]{LSUYY2}
to obtain the decomposition results.
However, the condition on $L$ in (\ref{eq:L}) is 
milder than that obtained from \cite[Theorem 4.5]{LSUYY2}.

\subsection{Besov-Morrey spaces and Triebel-Lizorkin-Morrey spaces}
\label{s7.2}

We can say that \cite{KoYa94} is oldest among such attempts.
The motivation was to obtain the solution starting 
from larger function spaces,
which are called Besov Morrey spaces.
Let $0<q \le p<\infty$, $0<r \le \infty$ and $s \in {\mathbb R}$.
The Besov-Morrey space
${\mathcal N}^s_{pqr}({\mathbb R}^n)$, 
which is a mixture of the Besov space and the Morrey space, 
emerged originally 
in the context of the analysis in the time-local solutions 
of the Navier-Stokes equations \cite{KoYa94}. 
To investigate the time-local solutions of the equation 
H. Kozono and M. Yamazaki introduced the Besov-Morrey space
${\mathcal N}^s_{pqr}({\mathbb R}^n)$
for the range
$1\le q \le p \le \infty$, 
$1 \le r \le \infty$ 
and 
$s \in {\mathbb R}$; see \cite[Definition 2.3]{KoYa94}. 
Later on, this function space is extended and investigated intensively. 
A. Mazzucato investigated the decomposition of this function space
\cite{An1}. 
She also investigated the pseudo-differential operators
as well as
the Besov-Morrey spaces on compact oriented Riemannian manifolds 
\cite{Mazzucato02}. 
It is L. Tang and J. Xu that defined Triebel-Lizorkin-Morrey spaces 
${\mathcal E}^s_{pqr}({\mathbb R}^n)$
as well as they extended the parameters to the range 
$0 < q \le p \le \infty$, 
$0 < r \le \infty$ 
and 
$s \in {\mathbb R}$ 
\cite{TaXu}. 
L. Tang and J. Xu investigated a different type of
pseudo-differential operators as well \cite[Section 4]{TaXu};
see \cite{Sa12} for further information.
{One} also investigated various decompositions such as atomic
decomposition, molecular decomposition and quarkonial decomposition 
with parameters 
$0 < q \le p \le \infty$, 
$0 < r \le \infty$ 
and 
$s \in {\mathbb R}$; see
 \cite{Sawano08-3} and \cite[Theorems 4.1, 4.12, 5.3 and 5.9]{SaTa2007}.
Xu and Fu characterized the Besov-Morrey space ${\mathcal N}^s_{pqr}({\mathbb R}^n)$
and the Triebel-Lizorkin-Morrey space ${\mathcal E}^s_{pqr}({\mathbb R}^n)$
when $p$ and $q$ are variable exponents in \cite{FuXu11}.
Much was investigated for Morrey spaces a little earlier.
Najafov considered Besov-Morrey spaces
and Sobolev-Morrey spaces in \cite{Najafov05-1}
and \cite{Najafov05-2}, respectively.
Najafov also considered the embedding results
for Sobolev-Morrey spaces in \cite{Najafov06}.
Sawano and Tanaka considered the complex interpolation
of Morrey spaces, Besov-Morrey spaces in \cite{SaTa09-1}
but there was a mistake in \cite[Proposition 5.3]{SaTa09-1}.
The method introduced in the book \cite{BeLo76} was not
used in \cite{SaTa09-1}.
Yuan, Sickel and Yang overcame this problem
in \cite{YSY15}.
Other applications to PDE can be found in \cite{Keller12,KhYe10,KhYe13,XuFu12}.

\subsection{Herz-type Besov spaces}
\label{s7.3}

Although Kozono and Yamazaki introduced Besov-Morrey spaces
in 1994, much more was investigated 
for Herz spaces;
Xu defined Herz-type Besov spaces.
Let $1<p,q<\infty$ {and $\alpha\in\R$.}
We let $Q_0{\equiv}[-1,1]^n$ and $C_j{\equiv}[-2^j,2^j]^n \setminus [-2^{j-1},2^{j-1}]^n$
for $j \in {\mathbb N}$.
Recall that the Herz space $K_{p,q}^\alpha({\mathbb R}^n)$ is 
the set of all measurable functions $f$ for which the norm
$$\displaystyle
\| f \|_{K_{p,q}^\alpha}
{\equiv}
\| \chi_{Q_0} \cdot f\|_{L^p} 
+
\left(
\sum_{j=1}^\infty 
(2^{j\alpha}\| \chi_{C_j} \cdot f \|_{L^p})^q
\right)^\frac{1}{q}
$$ 
is finite.
Let $0<r \le \infty$ and $s \in {\mathbb R}$.
The Herz-type Besov space $K_{p,q}^\alpha B^s_r({\mathbb R}^n)$
is the set of all $f \in {\mathcal S}'({\mathbb R}^n)$
for which the quasi-norm
\[
\|f\|_{{\mathcal N}_{K_{p,q}^\alpha B^s_r}}
\equiv 
\begin{cases}
\displaystyle
\|\theta(D)f\|_{K_{p,q}^{\alpha}}
+
\left(\sum_{j=1}^\infty
2^{jsr}\|\tau_j(D)f\|_{K_{p,q}^{\alpha}}^r
\right)^{\frac1r}&(r<\infty),\\
\displaystyle
\|\theta(D)f\|_{K_{p,q}^{\alpha}}
+
\sup_{j \in {\mathbb N}}
2^{js}\|\tau_j(D)f\|_{K_{p,q}^{\alpha}}
&(r=\infty)
\end{cases}
\]
is finite.
Xu considered the boundedness property of the Fourier multiplier
in \cite{Xu01} for Herz-type Triebel-Lizorkin spaces
and proved the boundedness property of the lift operator
as well as the embedding property
of the {Schwartz} class
in \cite{Xu03-1}.
The boundedness property of the pointwise multiplier,
described in Section \ref{s5}, is obtained in \cite{Xu04-2,Xu06,XuYa03-2}.
Xu proved the boundedness property of the pseudo-differential operators
in \cite{Xu04,Xu06,XuYa03-2}.
We say that a quasi-normed space X is called to be admissible, if for
every compact subset $E \subset X$ and 
for every $\varepsilon > 0$, 
there exists a continuous map
$T:E\to X$ such that $T(E)$ 
is contained in a finite dimensional subset of $X$ and
$\|Tx-x\|_X \le \varepsilon$
for all $x \in E$.
Xu characterized the Herz-type Besov spaces 
by means of the Peetre maximal operator in \cite{Xu05-1,Xu05-2}
and used this characterization to prove the admissibility in \cite{Xu09-2}.
Xu obtained the atomic decomposition, the molecular decomposition,
and the wavelet decomposition in \cite{Xu14-2}.
We can find
applications of Herz-type Triebel-Lizorkin spaces
to partial differential equations,
more precisely,
to the Beal-Kato-Majda type and the Moser type inequalities
in \cite{Xu14-1}.
Dong and Xu considered the case when $\alpha$ and $p$ are variable exponents
in \cite{DoXu12}.
Shi and Xu considered Herz-type Triebel-Lizorkin spaces
with $\alpha$ and $p$ variable
in \cite{ShXu13}.
Likewise we can consider Herz-Morrey spaces.
Recall that the Herz-Morrey space $K_{p,q}^{\alpha,\lambda}({\mathbb R}^n)$ is 
the set of all measurable functions $f$ for which the norm
$$\displaystyle
\| f \|_{MK_{p,q}^{\alpha,\lambda}}
{\equiv}
\sup_{L \in {\mathbb N}_0}
2^{-L\lambda}
\left(
\| \chi_{Q_0} \cdot f\|_{L^p} 
+
\left(
\sum_{j=1}^L 
(2^{j\alpha}\| \chi_{C_j} \cdot f \|_{L^p})^q
\right)^\frac{1}{q}
\right)
$$ 
is finite.
Dong and Xu considered the case when $\alpha$ and $p$ are variable exponents
in \cite{DoXu15}.

\subsection{Besov-type spaces and Triebel-Lizorkin type spaces}

Yang and Yuan investigated Besov type spaces
and Triebel-Lizorkin type spaces
in \cite{YaYu08-2,YaYu10-1}.
Let $0<p,q \le \infty, \, s \in {\mathbb R}, \, \rho \ge 0$.
For 
$f \in {\mathcal S}_\infty'({\mathbb R}^n)$
one defines the {homogeneous} Besov-type norm and the homogeneous Triebel-Lizorkin type norm by:
\begin{align*}
\| f \|_{\dot{B}_{pq}^{s,\rho}}
&\equiv
\sup_{Q \in {\mathcal D}}
\frac{1}{|Q|^\rho}
\left(
\sum_{j=-\log_2 \ell(Q)}^\infty
2^{sqj}
\left(\int_{Q}|\tau_j(D)f(x)|^p\,dx
\right)^{\frac{q}{p}}\right)^{\frac{1}{q}}, \,\\
\| f \|_{\dot{F}_{pq}^{s,\rho}}
&\equiv
\sup_{Q \in {\mathcal D}}
\frac{1}{|Q|^\rho}
\left(
\int_Q
\left(
\sum_{j=-\log_2 \ell(Q)}^\infty
2^{sqj} 
|\tau_j(D)f(x)|^q\right)^{\frac{p}{q}}\,dx \, 
\right)^{\frac{1}{p}},
\end{align*}
respecitively.
The spaces
$B_{pq}^{s,\rho}({\mathbb R}^n), \, F_{pq}^{s,\rho}({\mathbb R}^n)$
stand for linear spaces of functions $f \in {\mathcal S}_\infty'({\mathbb R}^n)$
for which the quantities
$
\| f \|_{B_{pq}^{s,\rho}}, \,
\| f \|_{F_{pq}^{s,\rho}}<\infty$
respectively.
The notation $\dot{A}_{pq}^{s,\rho}$ stands for
either $\dot{B}_{pq}^{s,\rho}$ or $\dot{F}_{pq}^{s,\rho}$.
The {inhomogeneous} spaces are defined analogously.
To connect these scales with ours,
we prove the following proposition:
\begin{proposition}\label{prop:150312-2}
Let $0<q<\infty$, $0<r \le \infty$ and $\varphi \in {\mathcal G}_q$
satisfies $(\ref{eq:Nakai-3})$.
For $f \in {\mathcal S}'_\infty({\mathbb R}^n)$.
define
\[
\| f \|_{\dot{F}_{qr}^{s,\varphi}}
\equiv
\sup_{Q \in {\mathcal D}}
\varphi(\ell(Q))
\left(
\frac{1}{|Q|}
\int_Q
\left(
\sum_{j=-\log_2 \ell(Q)}^\infty
2^{sqj} 
|\tau_j(D)f(x)|^r\right)^{\frac{q}{r}}\,dx \, 
\right)^{\frac{1}{q}}.
\]
Then
$
\|f\|_{{\mathcal E}^s_{{\mathcal M}^\varphi_s,r}}
\sim
\|f\|_{\dot{F}_{qr}^{s,\varphi}}
$
for all $f \in {\mathcal S}_\infty'({\mathbb R}^n)$.
\end{proposition}
Sawano, Yang and Yuan proved
that the above scale
${\mathcal E}$
turned out to be Triebel-Lizorkin type spaces \cite{SYY10}.
Proposition \ref{prop:150312-2} is an analogy 
to this fact.
In view of \cite[Theorem 1.1]{SYY10},
the counterpart for generalized Besov-Morrey spaces
is not available.
\begin{proof}[Proof of Proposition \ref{prop:150312-2}]
To this end, comparing {these norms},
we have only to show
\begin{equation}\label{eq:150312-1}
\varphi(\ell(Q))
\left(
\frac{1}{|Q|}
\int_Q
\left(
\sum_{j=-\infty}^{-\log_2 \ell(Q)}
2^{sqj} 
|\tau_j(D)f(x)|^r\right)^{\frac{q}{r}}\,dx \, 
\right)^{\frac{1}{q}}
\lesssim
\|f\|_{\dot{F}_{qr}^{s,\varphi}}.
\end{equation}
Similarly to Example \ref{example:140820-3},
we have
\begin{equation}\label{eq:150312-2}
\|\tau_j(D)f\|_{L^\infty} \le \varphi(2^{-j})^{-1}\|f\|_{\dot{F}_{qr}^{s,\varphi}}.
\end{equation}
Thus, by combining Proposition \ref{prop:150312-1}
and (\ref{eq:150312-2}),
we obtain (\ref{eq:150312-1}).
\end{proof}

The space
${\mathcal E}^0_{pq2}({\mathbb R}^n)$
with $0<q \le p<\infty$ has many other equivalent norms.
In \cite{Sa12},
it was shown that this function space
${\mathcal E}^0_{pq2}({\mathbb R}^n)$
is equivalent to the Hardy-Morrey space
defined by Jia and Wang in \cite{JW}.
See \cite{AGNS15,IST-14,JW,SaWa13}
for more about Hardy Morrey spaces.
In \cite[Theorem 4.2]{An1},
Mazzucato proved that the Triebel-Lizorkin space
${\mathcal E}^0_{pq2}({\mathbb R}^n)$
with $1<q \le p<\infty$
is equivalent to the Morrey space
${\mathcal M}^p_q({\mathbb R}^n)$.
This type of norm equivalence is extended 
to many other function spaces in \cite{LSUYY2}
and
many authors applied this equivalence
to PDEs.
See \cite{Ho12-1} for
the boundedness of singular integral operators. 
By combining 
the wavelet characterization of Besov-Morrey spaces 
\cite[Theorem 1.5]{Sawano08-2}
and
the embedding criterion 
of the corresponding sequence space \cite[Theorem 3.2]{HaSk12},
Haroske and Skrzypczak
obtained
the necessary and sufficient conditions
on the parameters
$p_0,q_0,r_0,s_0$,
$p_1,q_1,r_1$ and $s_1$
for 
the embedding
$
{\mathcal N}^{s_0}_{p_0q_0r_0}({\mathbb R}^n)
\hookrightarrow
{\mathcal N}^{s_1}_{p_1q_1r_1}({\mathbb R}^n)
$ 
to hold.
In the context of the generalized Triebel-Lizorkin Morrey spaces,
we can improve \cite[Theorem 5.1]{SaWa13} as follows:
\begin{proposition}\label{prop:150312-100}
Let $1 \le q \le p<\infty$ and $0<r<{q}$.
Then by definining
\begin{equation}\label{eq:150312-201}
\varphi(t) \equiv \left[\log\left(2+\frac{1}{t}\right)\right]^{-1/\min(1,r)}
\quad (t>0),
\end{equation}
we have
$
\|f\|_{{\mathcal E}^0_{{\mathcal M}^\varphi_q,r}}
\lesssim
\|f\|_{{\mathcal E}^{n/p}_{pq\infty}}
$
for all $f \in {\mathcal E}^{n/p}_{pq\infty}({\mathbb R}^n)$.
\end{proposition}

\begin{proof}
According to Theorems \ref{thm:decomposition-1} and \ref{thm:decomposition-2},
we have only to prove
\[
{\bf e}^{n/p}_{pq\infty}({\mathbb R}^n)
\equiv {\bf e}^{n/p}_{{\mathcal M}^\psi_q,\infty}({\mathbb R}^n)
\subset
{\bf e}^0_{{\mathcal M}^\varphi_q,r} ({\mathbb R}^n),
\] 
where $\psi(t) \equiv t^{n/p}, t>0$.
This amounts to showing that
\begin{eqnarray}
\nonumber
&&\varphi(\ell(Q))
\left(\frac{1}{|Q|}
\int_Q 
\left(
\sum_{\nu=0}^\infty
\left|\sum_{m \in {\mathbb Z}^n}
\lambda_{\nu m}\chi_{Q_{\nu m}}(x)\right|^r
\right)^{q/r}\,dx
\right)^{1/q}\\
\label{eq:150312-101}
&&\lesssim
|Q|^{\frac1p-\frac1q}
\left(\int_Q 
\left(
\sup_{\nu \in {\mathbb N}_0}
2^{\nu n/p}\left|\sum_{m \in {\mathbb Z}^n}
\lambda_{\nu m}\chi_{Q_{\nu m}}(x)\right|
\right)^{q}\,dx
\right)^{1/q}
\end{eqnarray}
for all complex sequences $\{\lambda_{\nu m}\}_{\nu \in {\mathbb N}_0, m \in {\mathbb Z}^n}$.

If $\ell(Q) \ge 1$,
then (\ref{eq:150312-101}) is easy to prove;
we just combine
\[
\varphi(\ell(Q)) \lesssim |Q|^{1/p}, \quad
\sum_{\nu=0}^\infty
\left|\sum_{m \in {\mathbb Z}^n}
\lambda_{\nu m}\chi_{Q_{\nu m}}(x)\right|^r
\lesssim
\sup_{\nu \in {\mathbb N}_0}
2^{\nu nr/p}\left|\sum_{m \in {\mathbb Z}^n}
\lambda_{\nu m}\chi_{Q_{\nu m}}(x)\right|^r.
\]
If $\ell(Q) \le 1$,
we shall prove;
\begin{eqnarray}
\nonumber
&&\varphi(\ell(Q))
\left(\frac{1}{|Q|}
\int_Q 
\left(
\sum_{\nu=j_Q}^\infty
\left|\sum_{m \in {\mathbb Z}^n}
\lambda_{\nu m}\chi_{Q_{\nu m}}(x)\right|^r
\right)^{q/r}\,dx
\right)^{1/q}\\
\label{eq:150312-102}
&&\lesssim
|Q|^{\frac1p-\frac1q}
\left(\int_Q 
\left(
\sup_{\nu \in {\mathbb N}_0}
2^{\nu n/p}\left|\sum_{m \in {\mathbb Z}^n}
\lambda_{\nu m}\chi_{Q_{\nu m}}(x)\right|
\right)^{q}\,dx
\right)^{1/q}
\end{eqnarray}
and 
\begin{eqnarray}
\nonumber
&&\varphi(\ell(Q))
\left(\frac{1}{|Q|}
\int_Q 
\left(
\sum_{\nu=0}^{j_Q}
\left|\sum_{m \in {\mathbb Z}^n}
\lambda_{\nu m}\chi_{Q_{\nu m}}(x)\right|^r
\right)^{q/r}\,dx
\right)^{1/q}\\
\label{eq:150312-103}
&&\lesssim
|Q|^{\frac1p-\frac1q}
\left(\int_Q 
\left(
\sup_{\nu \in {\mathbb N}_0}
2^{\nu n/p}\left|\sum_{m \in {\mathbb Z}^n}
\lambda_{\nu m}\chi_{Q_{\nu m}}(x)\right|
\right)^q\,dx
\right)^{1/q}.
\end{eqnarray}
We have (\ref{eq:150312-102}) since,
for any $Q \in {\mathcal D}$,
\[
\varphi(\ell(Q)) \lesssim 1, 
\quad
\sum_{\nu=j_Q}^\infty
\left|\sum_{m \in {\mathbb Z}^n}
\lambda_{\nu m}\chi_{Q_{\nu m}}(x)\right|^r
\lesssim|Q|^{1/p}
\sup_{\nu \in {\mathbb N}_0}
2^{\nu nr/p}\left|\sum_{m \in {\mathbb Z}^n}
\lambda_{\nu m}\chi_{Q_{\nu m}}(x)\right|^r.
\]
As for (\ref{eq:150312-103}),
we write $m(\nu) \in {\mathbb Z}^n, \nu \le j_Q$
for the unique element $m \in {\mathbb Z}^n$ such that
$Q_{\nu m} \supset Q$. 
we use {the Minkowski inequality and}
the triangle inequality to have;
\begin{eqnarray*}
\nonumber
&&\varphi(\ell(Q))
\left(\frac{1}{|Q|}
\int_Q 
\left(
\sum_{\nu=0}^{j_Q}
\left|\sum_{m \in {\mathbb Z}^n}
\lambda_{\nu m}\chi_{Q_{\nu m}}(x)\right|^r
\right)^{q/r}\,dx
\right)^{1/q}\\
\nonumber
&&\le
\varphi(\ell(Q))
\left(
\sum_{\nu=0}^{j_Q}
\left(\frac{1}{|Q|}
\int_Q 
\left|\sum_{m \in {\mathbb Z}^n}
\lambda_{\nu m}\chi_{Q_{\nu m}}(x)\right|^{q}\,dx
\right)^{r/q}
\right)^{1/r}\\
&&=
\varphi(\ell(Q))
\left(
\sum_{\nu=0}^{j_Q}
\left(\frac{1}{|Q_{m(\nu)\nu}|}
\int_{Q_{m(\nu)\nu}} 
\left|\sum_{m \in {\mathbb Z}^n}
\lambda_{\nu m}\chi_{Q_{\nu m}}(x)\right|^{q}\,dx
\right)^{r/q}
\right)^{1/r}\\
&&\le
\varphi(\ell(Q))
\left(
\sum_{\nu=0}^{j_Q}
|Q_{m(\nu)\nu}|^{\frac{r}{p}-\frac{r}{q}}
\left(\int_{Q_{m(\nu)\nu}}
\left(
\sup_{\nu \in {\mathbb N}_0}
2^{\nu n/p}\left|\sum_{m \in {\mathbb Z}^n}
\lambda_{\nu m}\chi_{Q_{\nu m}}(x)\right|
\right)^{q}\,dx
\right)^{r/q}
\right)^{1/r}\\
&&\lesssim
\varphi(\ell(Q))(j_Q)^{1/r} \|f\|_{{\mathcal E}^{n/p}_{pq\infty}}
\lesssim
\|f\|_{{\mathcal E}^{n/p}_{pq\infty}}.
\end{eqnarray*}
Thus, (\ref{eq:150312-101}) is proved.
\end{proof}

A couple of remarks on Proposition \ref{prop:150312-100} may be in order.
\begin{remark}\label{rem:150821-2}
Let $1<q \le p<\infty$.
\begin{enumerate}
\item
Let $r=1$ in Proposition \ref{prop:150312-100}.
The authors in \cite{SaWa13} showed that
\[
\|f\|_{{\mathcal M}^\varphi_1} \lesssim \|(1-\Delta)^{n/2p}f\|_{{\mathcal M}^p_q}
\]
for all $f \in {\mathcal M}^p_q({\mathbb R}^n)$
with $\Delta f \in {\mathcal M}^p_q({\mathbb R}^n)$.
Since
\[
\|f\|_{{\mathcal M}^\varphi_1} \lesssim
\|f\|_{{\mathcal E}^s_{{\mathcal M}^\varphi_q,1}}, \quad
\|f\|_{{\mathcal E}^{n/p}_{pq\infty}} \lesssim 
\|(1-\Delta)^{n/2p}f\|_{{\mathcal M}^p_q}
\]
for all $f \in {\mathcal M}^p_q({\mathbb R}^n)$
with $\Delta f \in {\mathcal M}^p_q({\mathbb R}^n)$,
Proposition \ref{prop:150312-100} improves \cite[Theorem 5.1]{SaWa13}.
\item
According to the necessary and sufficient condition 
in \cite[Theorem 2]{EGNS14},
one can not have
\[
\|f\|_{\infty} 
\lesssim 
\|(1-\Delta)^{n/2p}f\|_{{\mathcal M}^p_q}.
\]
\item
One can not replace $\min(1,r)$ by $1$ in $(\ref{eq:150312-201})$
when $r \in (0,1)$.
Assume to the contrary that this is possible.
Let $x_0=(3/2,0,\ldots,0)$
{
and 
}
$\eta \in C^\infty_{\rm c}({\mathbb R}^n)$
be such that $\chi_{Q(x_0,1/100)} \le \eta \le \chi_{Q(x_0,1/10)}$.
Choose $\tau \in C^\infty_{\rm c}({\mathbb R}^n)$
so that 
$\chi_{Q(2) \setminus Q(1)} \le 
\tau \le \chi_{Q(2+1/10) \setminus Q(9/10)}$. 
Then
for 
\[
f_N{\equiv}\sum_{l=1}^N {\mathcal F}^{-1}\eta_k
\quad (N=1,2,\ldots),
\]
{
one has
}
\[
\|f_N\|_{{\mathcal E}^0_{{\mathcal M}^\varphi_q,r}}
\gtrsim
\varphi(2^{-N})
\left(\frac{1}{|Q|}\int_Q 
\left(\sum_{k=1}^N |{\mathcal F}^{-1}\eta_k(x)|^r\right)^{q/r}\,dx
\right)^{1/q}
\gtrsim
\varphi(2^{-N})N^{1/r}.
\]
Meanwhile, $\|f_N\|_{{\mathcal E}^{n/p}_{pq\infty}} \lesssim 1$.
Thus, this is a contradiction since $N$ is arbitrary.
\end{enumerate}
\end{remark}

Haroske and Skrzypczak 
{
work also with
}
the setting of bounded open sets $\Omega$
and prove similar results;
the necessary and sufficient conditions
on the parameters
$p_0,q_0,r_0,s_0$,
$p_1,q_1,r_1$ and $s_1$
for 
the embedding
$
{\mathcal N}^{s_0}_{p_0q_0r_0}(\Omega)
\hookrightarrow
{\mathcal N}^{s_1}_{p_1q_1r_1}(\Omega)
$ 
to hold.
Here for the definition of the function space
${\mathcal N}^{s}_{pqr}(\Omega)$
can be found in 
\cite[Definition 5.1]{Sawano10-1}
and
\cite[Definition 2.7]{HaSk13}. 
See \cite{Sawano10-2,SST09-2,YSY13-2}
for more results on the Sobolev embedding theorem.

In \cite{Triebel13},
Triebel 
{
introduced the so called
}
local space
${\mathcal L}^r A^s_{pq}({\mathbb R}^n)$.
See \cite{Triebel13} and \cite[Section 3]{YSY13-1} for the definition.
In \cite{YSY13-1}, Yang, Sickel and Yuan proved that
the scale $A^{s,\tau}_{pq}({\mathbb R}^n)$ comes about naturally
as a result of the localization of $A^s_{pq}({\mathbb R}^n)$.
The local means are considered in \cite{Rosenthal13}.
Yang, Yuan and Zhuo investigated the boundedness property
of the Fourier multiplier precisely in \cite{YYZ12}.
Yang and Yuan characterized $A^{s,\rho}_{pq}({\mathbb R}^n)$
in terms of the Peetre maximal operator in \cite{YaYu10-2}.
See \cite{YaYu13-1,YSY10-2} for more recent advances 
in Triebel-Lizorkin type spaces,
and see \cite{HaSk13,Sawano10-1} 
for the extension of this scale to domains. 
We refer to \cite{Sickel12,Sickel13,YSY10,YaYu13-2} for
an exhaustive account of these function spaces 
as well as of results on decompositions.

\subsection{Orlicz--Morrey spaces and Musielak-Orlicz Triebel--Lizorkin-type spaces}\label{s7.5}

Recall that for a function $\varphi:{\mathbb R}^n \times [0,\infty) \to [0,\infty)$,
and a measurable function defined on ${\mathbb R}^n$,
the Musielak-Orlicz {norm} is given by:
\[
\|f\|_{L^\varphi}
\equiv
\inf\left\{\lambda>0\,:\,
\int_{{\mathbb R}^n}\varphi\left(x,\frac{|f(x)|}{\lambda}\right)\,dx \le 1\right\}.
\]
Here and below we let
$I_1,I_2,i_1,i_2>0$, $q_1,q_2, \delta_1,\delta_2>1$
and assume that 
$\varphi_1,\varphi_2:{\mathbb R}^n \times [0,\infty) \to [0,\infty)$
are functions satisfying the upper and lower type conditions
\[
s^{i_k}\varphi_k(x,t) 
{\lesssim} 
\varphi_k(x,st) 
{\lesssim} 
s^{I_k}\varphi_{k}(x,t)
\quad (x \in {\mathbb R}^n, s \ge 1, t>0, k=1,2),
\] 
the uniformly $A_\infty$-condition
\[
\frac{1}{|Q|}\int_Q \varphi_k(x,t)\,dx
\left(
\frac{1}{|Q|}\int_Q \varphi_k(x,t)^{1-q_k}\,dx
\right)^{q_k-1} {\lesssim 1} \quad
(t \ge 0, Q \in {\mathcal Q})
\]
and the reverse H\"{o}lder condition
\[
\left(
\frac{1}{|Q|}\int_Q \varphi_k(x,t)^{\delta_k}\,dx
\right)^{1/\delta_k}
\le C
\frac{1}{|Q|}\int_Q \varphi_k(x,t)\,dx;
\]
see \cite[p. 96]{YYZ14}.
In view of Proposition \ref{prop:150312-2},
we can say that 
${\mathcal E}^s_{{\mathcal M}^\varphi_s,r}({\mathbb R}^n)$
is the 
Musielak-Orlicz Triebel-Lizorkin--type space
of $\dot{F}^{s,\tau}_{\varphi_1,\varphi_2,q}({\mathbb R}^n)$
defined in \cite[Definition 2.1]{YYZ14}.
Let us recall the definition.
\begin{definition}
Let $s \in {\mathbb R}$, $\tau \in [0,\infty)$ and $q \in (0,\infty]$.
Then define the Musielak-Orlicz Triebel-Lizorkin-type space
$\dot{F}^{s,\tau}_{\varphi_1,\varphi_2,q}({\mathbb R}^n)$
as the set of all 
$f \in {\mathcal S}_\infty'({\mathbb R}^n)$
for which the quasi-norm
\[
\|f\|_{\dot{F}^{s,\tau}_{\varphi_1,\varphi_2,q}({\mathbb R}^n)}
\equiv
\sup_{Q \in {\mathcal D}}
\frac{1}{(\|\chi_Q\|_{L^{\varphi_1}})^\tau}
\left\|\left[
\chi_Q\sum_{j=j_Q}^\infty (2^{js}|\tau_j(D)f|)^q
\right]^{\frac1q}\right\|_{L^{\varphi_2}}
\]
is finite.
\end{definition}

In this sense,
the results for ${\mathcal E}^s_{{\mathcal M}^\varphi_s,r}({\mathbb R}^n)$
can be covered by \cite{YYZ14}.
For example,
Proposition \ref{prop:140820-1} can be understood
as the inhomogeneous version of \cite[{Proposition} 2.19]{YYZ14}.
Observe that \cite[Theorem 3.1]{YYZ14}
characterizes Musielak-Orlicz Triebel-Lizorkin--type spaces
by means of the Peetre maximal operator given below.
Our atomic decompsoition results,
Theorem{s} 
\ref{thm:decomposition-1} and \ref{thm:decomposition-2}
correspond to \cite[Theorem 5.1]{YYZ14}.
By using the idea of \cite[Theorem 6.9]{YYZ14}
or \cite{Sawano09-2},
we can prove the pseudo-differential operators 
with symbol in $S^0$ is bounded in 
${\mathcal A}^s_{{\mathcal M}^\varphi_q,r}({\mathbb R}^n)$.

Let us check that Musielak-Orlicz Triebel-Lizorkin--type spaces
come from (one of) generalized Orlicz Morrey spaces.
To the best knowledge of the authors,
there exists three generalized Orlicz-Morrey spaces.
\begin{definition}
\label{def1.1} 
Let $\Phi\in{\mathbb R}^n \times [0,\infty) \to [0,\infty)$ 
and $\varphi:{\mathcal Q} \to [0,\infty)$ be suitable functions.
\begin{enumerate}
\item[(1)] 
For a cube $Q\in{\mathcal{Q}}$ define
the $(\varphi,\Phi)$-average over $Q$ of the measurable function $f$ by 
\begin{equation*}
\|f\|_{(\varphi,\Phi);Q} {\equiv} \inf\left\{\lambda>0\,:\, \frac{\varphi(Q)}{|Q|%
}\int_{Q}\Phi\left(x,\frac{|f(x)|}{\lambda}\right)\,dx\le 1 \right\}. 
\end{equation*}
Define the generalized Orlicz-Morrey space ${\mathcal{L}}^{\varphi,\Phi}({%
\mathbb{R}}^n)$ to be a Banach space equipped with the norm 
\begin{equation*}
\|f\|_{{\mathcal{L}}^{\varphi,\Phi}} {\equiv} \sup_{Q\in{\mathcal{Q}}}
\|f\|_{(\varphi,\Phi);Q}. 
\end{equation*}
\item[(2)] For a cube $Q\in{\mathcal{Q}}$ define the $\Phi$-average over $Q$ of
the measurable function $f$ by 
\begin{equation*}
\|f\|_{\Phi;Q} {\equiv} \inf\left\{\lambda>0\,:\, \frac1{|Q|}\int_{Q}\Phi\left(x,
\frac{|f(x)|}{\lambda}\right)\,dx\le 1 \right\}. 
\end{equation*}
Define the generalized Orlicz-Morrey
space ${\mathcal{M}}_{\varphi,\Phi}({\mathbb{R}}^n)$ to be a Banach space
equipped with the norm 
\begin{equation*}
\|f\|_{{\mathcal{M}}_{\varphi,\Phi}} {\equiv} \sup_{Q\in{\mathcal{Q}}}
\varphi(Q)\|f\|_{\Phi;Q}. 
\end{equation*}
\item[(3)]
Define the generalized Orlicz-Morrey
space ${\mathcal{Z}}_{\varphi,\Phi}({\mathbb{R}}^n)$ to be a Banach space
equipped with the norm 
\begin{equation*}
\|f\|_{{\mathcal{Z}}_{\varphi,\Phi}} {\equiv} \sup_{Q\in{\mathcal{Q}}}
\varphi(Q)\|\chi_Q f\|_{L^\Phi}. 
\end{equation*}
\end{enumerate}
\end{definition}
The spaces 
${\mathcal L}^{\varphi,\Phi}({\mathbb R}^n)$, 
${\mathcal M}_{\varphi,\Phi}({\mathbb R}^n)$
and 
${\mathcal Z}_{\varphi,\Phi}({\mathbb R}^n)$ 
are defined
by Nakai in \cite{Nakai06}
{(with $\Phi$ independent of $x$)},
by Sawano, Sugano and Tanaka in \cite{SST12}
{(with $\Phi$ independent of $x$)}
and
by Deringoz, Guliyev and Samko in \cite[Definition 2.3]{DGS14},
respectively.
According to the examples in \cite{GST15},
we can say that the scales
${\mathcal L}$
and 
${\mathcal M}$
are different
and that
${\mathcal M}$
and 
${\mathcal Z}$
are different.
However, it is not known that
${\mathcal L}$
and 
${\mathcal Z}$
are different.

In Proposition \ref{prop:140820-1},
we rephrased the notion of generalized Triebel-Lizorkin-Morrey spaces
in the language of Triebel-Lizorkin type spaces.
We can do the vice {versa} as the following lemma implies.
\begin{lemma}{\rm \cite[Theorem 4.1]{YYZ14}}
Assume that $\tau$ satisfy
\[
0 \le \tau <\frac{i_1(\delta_2-1)}{q_1I_2\delta_2}.
\]
Then
\[
\|f\|_{\dot{F}^{s,\tau}_{\varphi_1,\varphi_2,q}({\mathbb R}^n)}
\sim
\sup_{Q \in {\mathcal D}}
\frac{1}{(\|\chi_Q\|_{L^{\varphi_1}})^\tau}
\left\|\left[
\chi_Q\sum_{j=-\infty}^\infty (2^{js}|\tau_j(D)f|)^q
\right]^{\frac1q}\right\|_{L^{\varphi_2}}.
\]
\end{lemma}

Therefore, by setting
\[
\varphi(x,r) \equiv \frac{1}{(\|\chi_{Q(x,r)}\|_{L^{\varphi_1}})^\tau}, \quad
\Phi(x,r) \equiv \varphi_2(x,r),
\]
we can say that Musielak-Orlicz Triebel-Lizorkin--type space
$\dot{F}^{s,\tau}_{\varphi_1,\varphi_2,q}({\mathbb R}^n)$
come from the generalized Orlicz Morrey space
${\mathcal Z}_{\varphi,\Phi}({\mathbb R}^n)$.

\subsection{Hausdorff Besov-type spaces and Hausdorff Triebel-Lizorkin type spaces}

Let us recall the definition of Besov-Hausdorff spaces
and Triebel-Lizorkin-Hausdorff spaces,
which are the predual spaces of 
Besov-type spaces and Triebel-Lizorkin type spaces.
Let $E \subset {\mathbb R}^n$ and $d\in(0,\,n]$.
The {\it $d$-dimensional Hausdorff capacity of $E$} is defined by
\begin{equation*}
H^d(E)\equiv \inf \left\{\sum_jr_j^d:\,E\subset
\bigcup_jB(x_j,\,r_j)\right\},
\end{equation*} 
where the infimum is taken over all covers $\{B(x_j,\,r_j)\}_{j=1}^\infty$ of
$E$ by countable families of open balls.
It is well known that $H^d$
is monotone, countably subadditive and vanishes on empty sets.
Moreover, the notion of $H^d$ can be extended to $d=0$. In this
case, $H^0$ has the property that for all nonempty sets 
$E\subset{\mathbb R}^n$,
$H^0(E)\ge1$, and $H^0(E)=1$ if and only if $E$ is bounded.
\begin{definition} [Choquet integral]
For any function $f: {\mathbb R}^n \mapsto [0,\,\infty]$, the \textit{Choquet
integral of $f$ with respect to $H^d$} is defined by
$$\int_{{\mathbb R}^n}f\,d H^d\equiv \int_0^{\infty}
H^d(\{x\in{\mathbb R}^n:\ f(x)>\lambda\})\,d\lambda.$$ This functional is not
sublinear, so sometimes we need to use an equivalent integral with
respect to the $d$-dimensional dyadic Hausdorff capacity
$\widetilde{H}^d$, which is sublinear.
\end{definition}
To define the spaces, we also need the nontangential maximal operator.
\begin{definition}[Nontangential maximal operator]
Let 
$
{\mathbb R}_+^{n+1}\equiv {\mathbb R}^n \times (0,\infty).
$ 
For any measurable function $\omega$ 
on ${\mathbb R}_+^{n+1}$ and $x\in{\mathbb R}^n$, we define its
\textit{nontangential maximal function} $N\omega (x)$ 
by setting
$\displaystyle
N\omega(x)\equiv \sup_{|y-x|<t} |\,\omega(y,t)|.
$
\end{definition}

\begin{definition}
Let $p\in(1,\infty)$ and $s \in {\mathbb R}$.
\begin{enumerate}
\item
If $q\in[1,\infty)$ and $\tau\in
{
[0, \frac{1}{ \max{(p,q)}'}]}$,
{\it the Besov-Hausdorff space
$B\dot{H}_{p,q}^{s,\tau}({\mathbb R}^n)$} is the set of all
$f \in {\mathcal S}_\infty'({\mathbb R}^n)$ such that
$$
\| \, f \, \|_{B\dot{H}_{p,q}^{s,\tau}({\mathbb R}^n)} \equiv \inf_\omega
\left\{ \sum_{j \in {\mathbb Z}}2^{jsq} 
\left\| \tau_j(D)f \cdot [\omega(\cdot,2^{-j})]^{-1} \right\|^q_{L^p({\mathbb R}^n)}
\right\}^\frac1q
$$
is finite,
where $\omega$ runs over all nonnegative Borel measurable functions
on ${\mathbb R}_+^{n+1}$ such that
\begin{equation}
\label{1.1} 
\int_{{\mathbb R}^n} [N\omega(x)]^{{\max{(p,q)}'}} \,dH^{{n\tau\max{(p, q)}'}}(x) \le 1
\end{equation}
and with the restriction that for any $j\in{\mathbb Z}$, 
$\omega(\cdot,2^{-j})$ is allowed to vanish only 
where $\tau_j(D) f$ vanishes.
\item
If $q\in(1,\infty)$ and 
$\displaystyle \tau\in\left[0,\frac{1}{{\max{(p, q)}'}}\right]$,
{\it the Triebel-Lizorkin-Hausdorff
space $F\dot{H}_{p,q}^{s,\tau}({\mathbb R}^n)$} is 
the set of
all $f \in {\mathcal S}_\infty'({\mathbb R}^n)$ such that
\begin{align*}
\| \, f \, \|_{F\dot{H}_{p,q}^{s,\tau}({\mathbb R}^n)}
&\equiv
\inf_\omega
\left\|
\left\{
\sum_{j \in {\mathbb Z}}2^{jsq}
\left|\tau_j(D)f \cdot [\omega(\cdot,2^{-j})]^{-1}\right|^q
\right\}^\frac1q\right\|_{L^p({\mathbb R}^n)}
\end{align*}
is finite,
where $\omega$ runs over all nonnegative Borel measurable functions
on ${\mathbb R}_+^{n+1}$ such that $\omega$ 
satisfies \eqref{1.1} and with
the restriction that 
for any $j\in{\mathbb Z}$, $\omega(\cdot, 2^{-j})$ is
allowed to vanish only where $\tau_j(D) f$ vanishes.
\end{enumerate}
\end{definition}

Yang and Yuan proved that these spaces
are realized as the dual space of a subspace of
$A^{s,\rho}_{pq}({\mathbb R}^n)$;
see \cite{YaYu11}.
{
It is proved in \cite{ZYY14-1} that
}
$FH^{s,\rho}_{p,q}({\mathbb R}^n)$ covers the predual space
of Morrey spaces.

\subsection{Function spaces in the metric measure settings}

These function spaces carry over to the weighted settings
or to the metric measure setting.
For $x\in{\mathbb R}^n$ and $r>0$ or more generally
for $x$ in the metric measure space $(X,d)$ and $r>0$,
we write 
\[
B(x,\,r)\equiv \{y\in X\,:\,{d(x,y)}<r\}.
\] 
Based on the Morrey space defined in \cite{SaTa05},
Sawano and Tanaka considered
Besov-Morrey spaces and Triebel-Lizorkin-Morrey spaces
in ${\mathbb R}^n$ equipped
with a Radon measure satisfying 
\begin{equation}\label{growth}
\mu(B(x,r)) \le Cr^D
\end{equation}
for some $0<D \le n$; see \cite{SaTa09-1}.
Probably, 
{
the theory can be generalized to
}
the setting
of geometrically doubling measure space
equipped with a Radon measure satisfying $(\ref{growth})$.
See \cite{FLYY15,LiYa14} for some 
{
results in
}
such setting.
In particular, in \cite{FLYY15},
the authors developed the theory of $H^p({\mathbb R}^n)$ spaces 
with $p \in (0,1)$
with nondoubling measures, which was supposed to be difficult.
Izuki, Sawano and Tanaka investigated
the weighted spaces with weights in $A^\infty_{\rm loc}$ \cite{IST10}.
These types of function spaces will {shed} light on 
because ${\mathcal S}({\mathbb R}^n)$ are not dense in them.

\section*{Acknowledgement}

The authors are thankful to many people,
especially {the} group in Beijing {Normal Nniversity,}
for the stimulating discussion on the proof of Theorem \ref{thm:150301-1}.
First, the third author is thankful to Professor Hidemitsu Wadade
for his pointing out that the proof in \cite{Sawano-book} is not correct.
{Professor Wen Yuan in Beijing Normal University
pointed out that the manuscript for the proof of Theorem \ref{thm:150301-1}
in the manuscript from 2011
}
was not correct.
The authors are thankful to Dr. L. Laura for a discussion
of topological vector spaces {concerning} \cite{GoLo10} and Theorem \ref{thm:150312-22}.
The authors are grateful to Dr. Yiyu Liang and Professor Dachun Yang
for {the} discussion {on}
Remark \ref{rem:150228-1} and Proposition \ref{prop:150228-1}.
The authors are thankful to Dr. Ziji He in Beijing Normal University 
for his kind comment on the set ${\mathcal K}$
{concerning} (\ref{eq:150311-21}).
Finally, the third author is thankful to Professor Akira Kaneko
for this suggestion in Proposition \ref{prop:150228-1}.

This work is done during the stay of the third author
in Beijing Normal University.
The third author is thankful to Professor Dachun Yang
and his group for {their} hospitality.
{
Finally, we are thankful to anonymous reviewers
for their careful reading of {this} long paper.}

\end{document}